\pdfoutput=1

\RequirePackage{plautopatch}
\documentclass[openany]{article}

\usepackage{adjustbox}
\usepackage{amsmath,amssymb}
\usepackage{amsthm} 
\usepackage{ascmac}
\usepackage{autobreak}
\usepackage{bbm}
\usepackage{bm} 
\usepackage{braket}
\usepackage{cite}
\usepackage{dcolumn} 
\usepackage{enumitem}
\usepackage{footnote}
\usepackage{framed}
\usepackage{graphicx}
\usepackage{here}
\usepackage{hyperref}
\usepackage{ifthen}
\usepackage{longtable}
\usepackage{mathtools}
\usepackage{multirow}
\usepackage[square,numbers,sort&compress]{natbib}
\usepackage{pict2e}
\newboolean{arxiv}
\setboolean{arxiv}{true}
\ifthenelse{\boolean{arxiv}}{}{\usepackage{pxjahyper}}
\usepackage{pxrubrica}
\usepackage[usestackEOL]{stackengine}
\usepackage{txfonts} 
\usepackage{xparse}
\usepackage[only,llparenthesis,rrparenthesis,mapsfrom]{stmaryrd}

\usepackage{color}
\usepackage[most]{tcolorbox}
\usepackage{tikz}
\usetikzlibrary{arrows}
\usetikzlibrary{decorations.pathmorphing}

\usepackage{settings_common}
\usepackage{settings}

\newboolean{commentout}
\setboolean{commentout}{true}
\newcommand{\CommentOut}[1]{\ifthenelse{\boolean{commentout}}{}{#1}}

\title{\titlename}
\ifthenelse{\boolean{arxiv}}{\date{}}{\date{\today}}
\author{Kenji Nakahira}

\begin{document}

\maketitle

\begin{abstract}
 In category theory, the use of string diagrams is well known to aid in
 the intuitive understanding of certain concepts, particularly when
 dealing with adjunctions and monoidal categories.
 We show that string diagrams are also useful in exploring fundamental properties
 of basic concepts in category theory, such as universal properties,
 (co)limits, Kan extensions, and (co)ends.
 For instance, string diagrams are utilized to represent
 visually intuitive proofs of the Yoneda lemma,
 necessary and sufficient conditions for being adjunctions,
 the fact that right adjoints preserve limits (RAPL),
 and necessary and sufficient conditions for having pointwise Kan extensions.
 We also introduce a method for intuitively calculating (co)ends using diagrammatic representations
 and employ it to prove several properties of (co)ends and weighted (co)limits.
 This paper proposes that using string diagrams is an effective approach
 for beginners in category theory to learn the fundamentals of the subject
 in an intuitive and understandable way.
\end{abstract}

\tableofcontents


\section{Introduction} \label{sec:intro}

String diagrams are commonly used as graphical representations in category theory
(e.g., \cite{Str-1995,Abr-Coe-2009,Sel-2011,Mar-2014,Hin-Mar-2023}).
They serve as a tool for visually depicting concepts such as adjunctions and monoidal categories
in an easily understandable way.
In recent years, string diagrams have found widespread use in various fields beyond mathematics
that apply category theory, including computer science, physics, control theory, and linguistics.

The aim of this paper is to demonstrate the usefulness of string diagrams
in learning basic concepts of category theory.
Currently, commutative diagrams are predominantly used in standard category theory textbooks,
while string diagrams are seldom employed.
This paper proposes that using string diagrams is an effective approach
for beginners to intuitively grasp the fundamentals of category theory.
String diagrams can not only represent concepts intuitively
but also aid in proving various statements.
For instance, we use string diagrams to prove the Yoneda lemma (see Theorem~\ref{thm:Yoneda}),
necessary and sufficient conditions for being adjunctions
(see Theorems~\ref{thm:AdjNasSnake} and \ref{thm:AdjNasUniv}),
the fact that right adjoints preserve limits (RAPL) (see Theorem~\ref{thm:LimitAdjPreserve}),
and necessary and sufficient conditions for having pointwise Kan extensions
(see Theorem~\ref{thm:KanPointwise}).
We also introduce diagrams for intuitively calculating (co)ends and employ them
to demonstrate some relationships between (co)limits, (co)ends, weighted (co)limits,
and pointwise Kan extensions.
It is well known that, with string diagrams, calculations involving functors and natural transformations
can be performed without explicitly considering functoriality or naturality.
Additionally, we show that string diagrams are useful for verifying whether given data
satisfies functoriality, naturality, universal property, etc.

We explain basic concepts in category theory (specifically, universal property, adjunctions,
(co)limits, Kan extensions, (co)ends, etc.) using string diagrams.
For more information on these concepts, please refer to sources such as Refs.~\cite{Mac-2013,Rie-2017}.
Reference~\cite{Lor-2015} is also helpful for (co)end calculations.
This paper is largely self-contained; however, we assume that the reader has a basic understanding
of the fundamental properties of category theory.
The primary objective of this paper is to demonstrate the usefulness of string diagrams;
we do not discuss various interesting examples or some basic concepts
(e.g., epi- or monomorphisms and examples of (co)limits such as equalizers)
where string diagrams may not be particularly helpful.
We do not address monoidal categories because the method of representing them
with string diagrams has already been extensively studied (e.g., \cite{Sel-2011} and its references).
Our focus is on categories or 2-categories consisting of categories,
excluding enriched categories or general bicategories.

\subsection{Features of string diagrams}

Diagrams commonly employed in standard textbooks can be classified into two types from a certain perspective:
(a) those that represent morphisms with arrows and
(b) those that consider 2-categories consisting of categories and represent functors with arrows.
In what follows, we refer to them as (a) arrow diagrams and (b) pasting diagrams \cite{Ben-1967,Pow-1990}.
We here provide several examples of them and briefly describe their relationship with
the corresponding string diagrams used in this paper.

\myparagraph{Relationship with arrow diagrams}

\noindent
As the first example comparing arrow diagrams with string diagrams,
we give the following expression representing naturality of a natural transformation
$\alpha \colon F \nto G$ with $F,G \colon \cC \to \cD$:
\begin{alignat}{2}
 &\text{(A) arrow diagram:} &\qquad\qquad& \text{(B) string diagram:} \nonumber \\ &\quad
 \begin{tikzpicture}[auto,->,baseline=0.6cm]
  \node (Fa) at (0, 1.2) {$Fa$};
  \node (Fb) at (0, 0) {$Fb$};
  \node (Ga) at (1.2, 1.2) {$Ga$};
  \node (Gb) at (1.2, 0) {$Gb$};
  \draw (Fa) to node {$\scriptstyle \alpha_a$} (Ga);
  \draw (Fb) to node[swap] {$\scriptstyle \alpha_b$} (Gb);
  \draw (Fa) to node[swap] {$\scriptstyle Ff$} (Fb);
  \draw (Ga) to node {$\scriptstyle Gf$} (Gb);
 \end{tikzpicture} \raisebox{-2em}{,}
 &&\quad \InsertMidPDF{intro_nat.pdf} \raisebox{-2em}{.}
 \label{eq:intro_nat}
\end{alignat}
This arrow diagram is commutative, which means that the two paths
from $Fa$ to $Gb$, i.e., $Gf \c \alpha_a$ and $\alpha_b \c Ff$, are equal
(where $\c$ is the composition of morphisms).
Information such as that $F$ and $G$ are functors from $\cC$ to $\cD$ is usually not included
in arrow diagrams, but it is present in string diagrams.
Note that it can also be represented by a string diagram as follows
(see in detail in Eq.~\eqref{eq:basic_natural_trans_string2}):
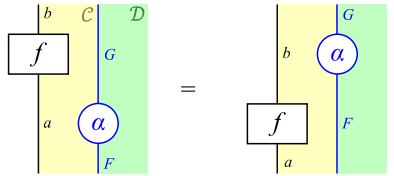
\begin{alignat}{1}
 \InsertPDF{report_intro_nat.pdf} \raisebox{1em}{.}
 \label{eq:basic_intro_nat}
\end{alignat}

As the second example, a universal morphism $\braket{u,\eta}$ from $c \in \cC$
to $G \colon \cD \to \cC$ is represented by
\begin{alignat}{2}
 &\text{(A) arrow diagram:} &\qquad\qquad& \text{(B) string diagram:} \nonumber \\ &\quad
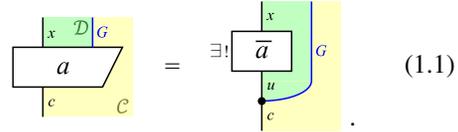
 \begin{tikzpicture}[auto,->,baseline=0.6cm]
  \node (c) at (0, 1.2) {$c$};
  \node (Gu) at (1.8, 1.2) {$Gu$};
  \node (Gx) at (1.8, 0) {$Gx$};
  \node (u) at (3.0, 1.2) {$u$};
  \node (x) at (3.0, 0) {$x$};
  \draw (c) to node {$\scriptstyle \eta$} (Gu);
  \draw (c) to node[swap] {$\scriptstyle a$} (Gx);
  \draw (Gu) to node {$\scriptstyle G\ol{a}$} (Gx);
  \draw (u) to node {$\scriptstyle \exists!\,\ol{a}$} (x);
 \end{tikzpicture} \raisebox{-2em}{,}
 &&\quad \InsertMidPDF{intro_univ.pdf} \raisebox{-2em}{.}
 \label{eq:intro_univ2}
\end{alignat}
In the string diagram, the black circle (dot) represents $\eta$.
In the arrow diagrams, two diagrams are drawn: one on the left representing morphisms
in the category $\cC$ and one on the right representing morphisms in the category $\cD$.
The diagram on the left is commutative (i.e., $a = G\ol{a} \c \eta$).
Thus, when representing morphisms from multiple categories in arrow diagrams,
multiple diagrams are generally required.
As can be seen from this example, the correspondence between arrow diagrams
and string diagrams is not very direct.

The main advantages of string diagrams over arrow diagrams are as follows%
\footnote{There are also situations where arrow diagrams are preferable to string diagrams.
The main disadvantages of string diagrams are as follows:
(a) It is difficult to effectively represent commutative expressions when they appear frequently.
(b) The style of drawing string diagrams can differ significantly among individuals.}:
\begin{itemize}
 \item A diagram can convey a large amount of information in an easy-to-understand form.
       While, in arrow diagrams, typically, only objects and morphisms within a single category can be
       represented in a single diagram, it is often necessary to represent
       objects and morphisms in multiple categories,
       as well as functors and natural transformations.
       String diagrams can represent all of them in a single diagram.
 \item Expressions involving both vertical and horizontal composition
       can be straightforwardly represented.
       Related to Advantage~1, horizontal composition is generally represented using
       equational forms in arrow diagrams.
       In contrast, both vertical and horizontal composition can be clearly represented
       in a string diagram.
       String diagrams are well-suited for representing the composition of morphisms,
       categories, functors, and natural transformations,
       revealing their structural relationships.
\end{itemize}
For a comparison of arrow diagrams and string diagrams from the perspective of a calculational
approach, see the Introduction of Ref.~\cite{Mar-2014}.

\myparagraph{Relationship with pasting diagrams}

\noindent
The following is an example comparing pasting diagrams with the string diagrams:
\begin{alignat}{2}
 &\text{(A) pasting diagram:} &\qquad\qquad& \text{(B) string diagram:} \nonumber \\ &\quad
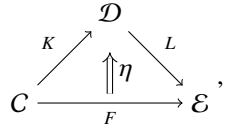
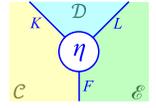
 \begin{tikzpicture}[auto,->,baseline=0.6cm]
  \node (C) at (2.4, 0) {$\cC$};
  \node (D) at (1.2, 1.2) {$\cD$};
  \node (E) at (0, 0) {$\cE$};
  \node (a1) at (1.2, 0) {};
  \node (a2) at (1.2, 0.8) {};
  \draw (C) to node[swap] {$\scriptstyle K$} (D);
  \draw (C) to node {$\scriptstyle F$} (E);
  \draw (D) to node[swap] {$\scriptstyle L$} (E);
  \draw[double equal sign distance,-implies] (a1) to node[swap] {$\eta$} (a2);
 \end{tikzpicture} \raisebox{-1em}{,}
 &&\quad \InsertMidPDF{intro_Kan_left.pdf} \raisebox{-1em}{.}
 \label{eq:intro_Kan_left}
\end{alignat}
Note that this diagram is drawn with Eq.~\eqref{eq:Kan_left_universal} in mind,
which represents a left Kan extension of $F \colon \cC \to \cE$ along $K \colon \cC \to \cD$.
There is a close relationship between these diagrams.
We can overlay the two diagrams as follows, omitting the overlapping information:
\begin{alignat}{1}
 \InsertPDF{report_intro_Kan_left2.pdf} \raisebox{1em}{.}
 \label{eq:intro_Kan_left_commute2}
\end{alignat}
In the pasting diagram, two types of arrows appear: arrows ``$\rightarrow$'' representing
functors and arrows ``$\Rightarrow$'' representing natural transformations.
Also, categories $\cC$, $\cD$, and $\cE$ are represented by points;
functors $K \colon \cC \to \cD$, $F \colon \cC \to \cE$, and $L \colon \cD \to \cE$ by arrows;
and a natural transformation $\eta \colon F \nto L \b K$ (where $\b$ denotes horizontal composition)
by a triangle with vertices $\cC$, $\cD$, and $\cE$ having an arrow ``$\Rightarrow$''.
When overlaid as in Eq.~\eqref{eq:intro_Kan_left_commute2}, it can be seen that
each arrow (in the pasting diagram) and the corresponding wire (in the string diagram) can be drawn
in such a way that they are orthogonal to each other.
The differences between pasting and string diagrams can be summarized as shown
in Table~\ref{tab:intro_diagram}.
\begin{table}[hbt]
 \centering
 \caption{Differences between pasting and string diagrams.}
 \label{tab:intro_diagram}
 \begin{tabular}{c|c|c}
  \bhline{1pt}
   & \multicolumn{1}{c|}{pasting diagrams} &
          \multicolumn{1}{c}{string diagrams} \\ \hline
  categories & points (0-dimensional) & surfaces (2-dimensional) \\
  functors & arrows ``$\to$'' (1-dimensional) & wires (1-dimensional) \\
  natural transformations & surfaces with arrows ``$\nto$'' (2-dimensional) & blocks (0-dimensional) \\
  \bhline{1pt}
 \end{tabular}
\end{table}

Another example is
\begin{alignat}{2}
 &\text{(A) pasting diagram:} &\qquad\qquad& \text{(B) string diagram:} \nonumber \\ &\quad
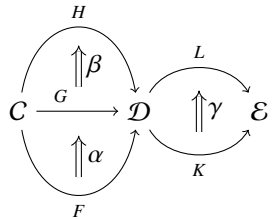
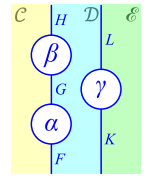
 \begin{tikzpicture}[auto,->,baseline=0cm]
  \node (C) at (3.2, 0) {$\cC$};
  \node (D) at (1.6, 0) {$\cD$};
  \node (E) at (0, 0) {$\cE$};
  \draw (C) to[out=100, in=80, looseness=2] node[swap] {$\scriptstyle H$} (D);
  \draw (C) to node[swap,pos=0.7] {$\scriptstyle G$} (D);
  \draw (C) to[out=-100, in=-80, looseness=2] node {$\scriptstyle F$} (D);
  \draw (D) to[out=120, in=60] node[swap] {$\scriptstyle L$} (E);
  \draw (D) to[out=-120, in=-60] node {$\scriptstyle K$} (E);
  \node (a1) at (2.4, -1.0) {};  \node (a2) at (2.4, -0.2) {};
  \node (b1) at (2.4, 0.2) {};  \node (b2) at (2.4, 1.0) {};
  \node (c1) at (0.8, -0.4) {};  \node (c2) at (0.8, 0.4) {};
  \draw[double equal sign distance,-implies] (a1) to node[swap] {$\alpha$} (a2);
  \draw[double equal sign distance,-implies] (b1) to node[swap] {$\beta$} (b2);
  \draw[double equal sign distance,-implies] (c1) to node[swap] {$\gamma$} (c2);
 \end{tikzpicture} \raisebox{-2em}{,}
 &&\quad \InsertMidPDF{intro_circ_bullet.pdf} \raisebox{-2em}{,}
 \label{eq:intro_circ_bullet}
\end{alignat}
where we represent the horizontal composition of a natural transformation $\gamma$
with the vertical composition of two natural transformations $\alpha$ and $\beta$,
i.e., $\gamma \b \beta\alpha$.
Both diagrams represent vertical composition by stacking elements vertically and
horizontal composition by placing them side by side.
As with the previous example, it is immediately clear that the two diagrams can be overlaid.

Pasting and string diagrams are dual to each other and can be converted interchangeably.
However, when it comes to discussions regarding the foundations of category theory,
there are often instances where calculations involving natural transformations need to be carried out.
Therefore, it is crucial to visually represent natural transformations in a clear form
when performing calculations.
Keeping this in mind, the main advantage of string diagrams over pasting diagrams is as follows:
\begin{itemize}
 \item The same type of natural transformation can be represented in the same way.
       In pasting diagrams, it can often be difficult to represent identical or the same type
       of natural transformations in the same shape.
       For example, consider the scenario where $\alpha = \beta = \gamma$ holds
       in Eq.~\eqref{eq:intro_circ_bullet}.
       In contrast, in string diagrams, these natural transformations can always (and easily)
       be represented in the same shape.
       Therefore, when calculating natural transformations, string diagrams should be
       visually easier to understand.
       Also, thanks to this advantage, it is easy to assign a special shape to
       a special type of natural transformation (such as the black circle
       in Eq.~\eqref{eq:intro_univ2}) and distinguish it from others.
\end{itemize}

\subsection{Key aspects of this paper}

The article most closely related to this paper is Refs.~\cite{Mar-2014,Hin-Mar-2023},
in which it has been argued that string diagrams have utility
in calculational approaches to category theory.
This paper aims to demonstrate the usefulness of string diagrams
for various concepts mentioned in standard textbooks on category theory,
which were not extensively explored in Refs.~\cite{Mar-2014,Hin-Mar-2023}.
Several key aspects of this paper include:
\begin{itemize}
 \item Set-valued functors: We introduce several diagrams for set-valued functors.
       Using these diagrams, we prove the Yoneda lemma in an intuitively understandable form
       (see Theorem~\ref{thm:Yoneda}).
 \item Universal properties: We represent universal properties with simple diagrams
       (see Subsubsection~\ref{subsubsec:repr_repr_diagram}).
       Note that this representation is similar to that of Ref.~\cite{Hin-2012}.
       Diagrams for adjoints, (co)limits, Kan extensions, etc. are depicted
       as specific instances of this representation.
 \item (Co)limits: We show some basic properties
       (e.g., any right adjoint and functor of the form $\cC(c,\Endash)$ preserve limits).
 \item Kan extensions: We mainly show some basic properties of pointwise Kan extensions.
 \item (Co)ends: We introduce a new diagram that visually represents (co)end-related
       expressions such as $\Func{\cC}{\cD}(F,G) \cong \int_{c \in \cC} \cD(Fc,Gc)$
       in an easy-to-understand form and perform several (co)end calculations graphically.
\end{itemize}


\section{Basics of string diagrams} \label{sec:preliminary}

In this section, we provide an overview of how to represent very basic concepts in category theory,
such as categories, functors, and natural transformations, using string diagrams.

\subsection{Categories, functors, and natural transformations} \label{subsec:category_category}

\subsubsection{Categories} \label{subsubsec:category_category_category}

For morphisms $f \colon a \to b$ and $g \colon b \to c$ in a category $\cC$,
the composite of $f$ and $g$, denoted by $g \c f \colon a \to c$
(or simply written as $gf$), is represented by the following diagram:
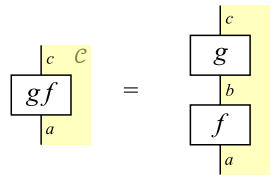
\begin{alignat}{1}
 \InsertPDF{basic_cat_circ.pdf} \raisebox{1em}{.}
 \label{eq:basic_cat_circ}
\end{alignat}
In our string diagrams, morphisms are represented by blocks
(which are often depicted with rectangular or trapezoidal blocks but can also take on other shapes),
while objects are represented by wires.
The wires extending from the bottom and top of a morphism represent its domain and codomain,
respectively.
The yellow region labeled with ``$\cC$'' indicates that the morphism $gf$ belongs to
the category $\cC$.
Regions of the same color within each diagram represent the same category unless otherwise stated,
and their duplicated labels are often omitted.
The composition operation is associative, i.e.,
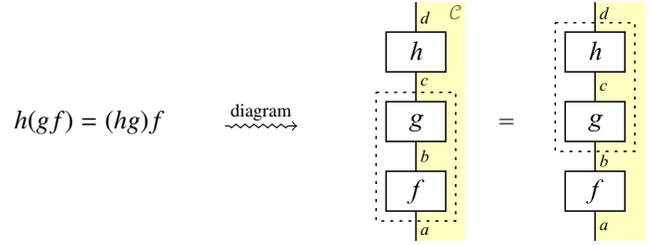
\begin{alignat}{1}
 h(gf) = (hg)f &\qquad\diagram\qquad
 \InsertMidPDF{basic_cat_associative.pdf},
 \label{eq:basic_cat_associative}
\end{alignat}
and unital, i.e.,
\begin{alignat}{1}
 f \c \id_a = f = \id_b \c f &\qquad
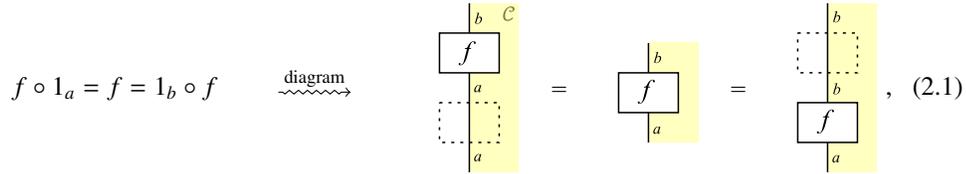
\diagram\qquad
 \InsertMidPDF{basic_cat_id.pdf},
 \label{eq:basic_cat_id}
\end{alignat}
where dashed lines (used as auxiliary lines) are only intended to guide the eye.
The symbol ``$\diagram$'' is used to show a corresponding diagram.
As shown in Eq.~\eqref{eq:basic_cat_id}, identity morphisms are simply represented by wires.
Note that it is not possible to distinguish between the identity morphism $\id_a$ and the object $a$
in diagrams, but this is not a substantial problem.
In this paper, we assume that all categories are locally small unless otherwise stated.
Let $\cC(a,b)$ be the hom-set from $a \in \cC$ to $b \in \cC$, i.e., the collection of all morphisms
with domain $a$ and codomain $b$.

The composite of two morphisms $f \in \cC^\op(b,a)$ and $g \in \cC^\op(c,b)$, $fg$, in
the opposite category $\cC^\op$ of $\cC$, is represented by
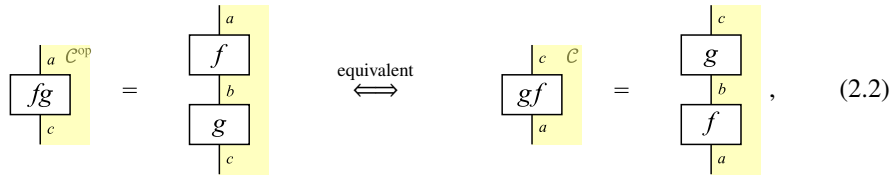
\begin{alignat}{1}
 \InsertMidPDF{basic_contra_circ.pdf}
 \qquad\myLeftrightarrow{equivalent}\qquad
 \InsertMidPDF{basic_contra_circ2.pdf},
 \label{eq:basic_contra_circ}
\end{alignat}
where the right-hand side of the symbol ``$\myLeftrightarrow{equivalent}$''
represents $f$ and $g$ as morphisms in $\cC$.
As shown in this diagram, it is common for the background colors of $\cC$ and $\cC^\op$ to be the same.
Note that $fg \in \cC^\op(c,a)$ on the left-hand side of Eq.~\eqref{eq:basic_contra_circ}
is the composite of $g \in \cC^\op(c,b)$ and $f \in \cC^\op(b,a)$, while
$gf \in \cC(a,c)$ on the right-hand side is the composite of $f \in \cC(a,b)$ and $g \in \cC(b,c)$.
For each category $\cC$, a morphism in its opposite category $\cC^\op$ is often represented
as the corresponding morphism in $\cC$, in which case the diagram is flipped vertically.
For each morphism $f \in \cC(a,b)$, the corresponding $f \in \cC^\op(b,a)$
may be written as $f^\op$ if there is a risk of confusion.

\subsubsection{Functors} \label{subsubsec:category_category_functor}

\myparagraph{Diagrammatic notation}

\noindent
Let us represent a functor $F \colon \cC \to \cD$ as a wire with the label ``$F$'' like this:
\begin{alignat}{1}
 \InsertPDF{basic_functor_easy.pdf} \raisebox{1em}{.}
 \label{eq:basic_functor_easy}
\end{alignat}
The regions on the left and right sides represent the categories $\cD$ and $\cC$, respectively.
As shown in this example, regions with different colors generally indicate different categories.
In a diagram, placing the wire representing the functor $F$ to the left of an object or morphism
indicates that $F$ is applied to it.
For example, for a morphism $f \in \cC(a,b)$, $Ff \in \cD(Fa,Fb)$ is diagrammatically represented by
\begin{alignat}{1}
 \InsertPDF{basic_functor_Ff.pdf} \raisebox{1em}{.}
 \label{eq:basic_functor_Ff}
\end{alignat}

Each functor preserves the structure of associativity of composition.
That is, any functor $F \colon \cC \to \cD$ and
any two composable morphisms $f \in \cC(a,b)$ and $g \in \cC(b,c)$ satisfy $F(gf) = (Fg)(Ff)$%
\footnote{When we say ``for any $f \in \cC(a,b)$'', unless otherwise stated,
we assume that both $a$ and $b$ are also arbitrary.}.
This can be depicted by
\begin{alignat}{1}
 \InsertMidPDF{basic_functor_string.pdf}
 &\qquad\myLeftrightarrow{equivalent}\qquad
 \InsertMidPDF{basic_functor.pdf},
 \label{eq:basic_functor}
\end{alignat}
where we show two equivalent representations.
Note that if we remove the auxiliary lines from the left-hand side
of the symbol ``$\myLeftrightarrow{equivalent}$'',
then it becomes impossible to distinguish between its left- and right-hand sides,
which poses no issue at all since both sides are equal.
Also, $F$ maps identity morphisms to identity morphisms.

\myparagraph{Several fundamental functors}

\noindent
The identity functor on $\cC$, denoted by $\id_\cC$, is represented as
\begin{alignat}{1}
 \InsertPDF{basic_functor_id.pdf} \raisebox{1em}{.}
 \label{eq:basic_functor_id}
\end{alignat}
As shown on the right-hand side, the wire representing the identity functor is often omitted.

A functor from a category $\cC$ to a discrete category, denoted as $\cOne$, with only one object,
is uniquely determined by mapping all morphisms in $\cC$ to the identity morphism $\id_*$
of the unique object $*$ in $\cOne$.
This functor is often denoted by ${!}$.
We represent ${!}$ with a gray dotted wire;
for a morphism $f$ in $\cC$, ${!}f = \id_*$ is represented by
\begin{alignat}{1}
 \InsertPDF{basic_functor_terminal.pdf} \raisebox{1em}{.}
 \label{eq:basic_functor_terminal}
\end{alignat}
Intuitively, the functor ${!}$ acts to erase all information about the input.
Note that the background of the region representing category $\cOne$ is white,
and the label ``$\cOne$'' is often omitted.

The dual of a functor $F \colon \cC \to \cD$ is denoted by the same symbol as
$F \colon \cC^\op \to \cD^\op$;
however, it may be written as $F^\op$ when confusion is likely to arise.

\myparagraph{Fully faithful functors}

\noindent
Let us consider a functor $F \colon \cC \to \cD$.
If the map $\cC(a,b) \ni f \mapsto Ff \in \cD(Fa,Fb)$ is bijective for each $a,b \in \cC$,
then $F$ is called \termdef{fully faithful}.
$F$ is fully faithful if and only if when we fix arbitrary $a,b \in \cC$,
for any $g \in \cD(Fa,Fb)$,
there exists a unique $\ol{g} \in \cC(a,b)$ such that
\begin{alignat}{1}
 \InsertPDF{basic_full_faithful.pdf} \raisebox{1em}{.}
 \label{eq:basic_full_faithful}
\end{alignat}
In diagrams, the symbol ``$\exists !$'' means that the block immediately to its right
is uniquely determined to satisfy the equation.

\subsubsection{Natural transformations}

For two functors $F,G \colon \cC \to \cD$, a natural transformation $\alpha \colon F \nto G$
is defined as a collection of indexed morphisms
$\alpha \coloneqq \{ \alpha_a \in \cD(Fa,Ga) \}_{a \in \cC}$ satisfying
\begin{alignat}{1}
 Gf \c \alpha_a = \alpha_b \c Ff
 &\qquad\diagram\qquad 
 \InsertMidPDF{basic_natural_trans_string.pdf}
 \tag{nat}
 \label{eq:nat}
\end{alignat}
for any morphism $f \colon a \to b$ in $\cC$.
In diagrams, we often represent natural transformations with circular blocks
(but they can also take on other shapes) like this one:
\begin{alignat}{1}
 \InsertPDF{basic_natural_trans_def.pdf} \raisebox{1em}{.}
 \label{eq:basic_natural_trans_def}
\end{alignat}
We often represent the component $\alpha_a$ of the natural transformation $\alpha$
by placing $\alpha$ immediately to the left of the object $a$.
When the components of a natural transformation are represented as in Eq.~\eqref{eq:basic_natural_trans_def},
Eq.~\eqref{eq:nat} can be rewritten as follows:
\begin{alignat}{1}
 \InsertPDF{report_basic_natural_trans_string2.pdf} \raisebox{1em}{.}
 \label{eq:basic_natural_trans_string2}
\end{alignat}
Intuitively, this equation indicates that the morphism $f$ and the natural transformation $\alpha$
can be moved freely along the wires.
This equation is often called the \termdef{sliding rule}.
The right-hand side of Eq.~\eqref{eq:basic_natural_trans_string2}
is sometimes denoted by $\alpha \b f$,
where the operator $\b$ represents the horizontal composition, which will be mentioned later.

Objects and morphisms in a category $\cC$ can be regarded as functors and natural transformations
from $\cOne$ to $\cC$, respectively.
For example, a functor $F \colon \cOne \to \cC$ can be identified with an object $F(*)$ in $\cC$,
and a natural transformation $\alpha \colon F \nto G$ with $F,G \colon \cOne \to \cC$ can be
identified with its unique component $\alpha_* \in \cC(F(*),G(*))$.
In this paper, we identify them.

\subsection{Vertical and horizontal composition}

\subsubsection{Vertical composition}

For two natural transformations $\alpha \colon F \nto G$ and $\beta \colon G \nto H$
with $F,G,H \colon \cC \to \cD$,
we write the natural transformation defined by
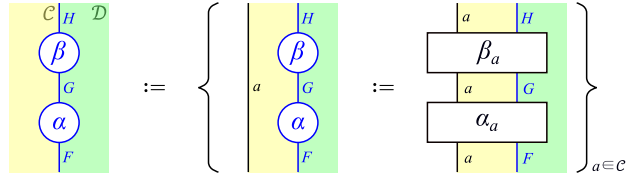
\begin{alignat}{1}
 \InsertPDF{basic_natural_trans2_def.pdf}
 \label{eq:basic_natural_trans2_def}
\end{alignat}
as $\beta \c \alpha$ or simply as $\beta\alpha$, and call it the \termdef{vertical composite}
of $\alpha$ and $\beta$.
It is clear that $\beta\alpha$ is a natural transformation because the sliding rule of
Eq.~\eqref{eq:basic_natural_trans_string2} gives
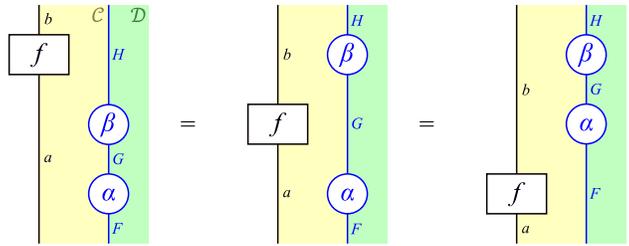
\begin{alignat}{1}
 \InsertPDF{basic_natural_trans_circ.pdf} \raisebox{1em}{.}
 \label{eq:basic_natural_trans_circ}
\end{alignat}

\subsubsection{Horizontal composition}

Since functors can be regarded as maps (for objects and morphisms), they can be composed
using ordinary map composition.
Specifically, for two functors $F \colon \cC \to \cD$ and $G \colon \cD \to \cE$,
we can consider a map that maps each object $a$ in $\cC$ to the object $G(Fa)$ and
each morphism $f$ in $\cC$ to the morphism $G(Ff)$.
This map, denoted by $G \b F$, is a functor known as the \termdef{horizontal composite} of $F$ and $G$.
We write $G \b F$ as
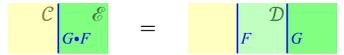
\begin{alignat}{1}
 \InsertPDF{basic_functor_bullet.pdf} \raisebox{1em}{.}
 \label{eq:basic_functor_bullet}
\end{alignat}
For a functor $F \colon \cC \to \cD$,
if there exists $H \colon \cD \to \cC$ such that $H \b F = \id_\cC$ and $F \b H = \id_\cD$,
then $H$ is called the \termdef{inverse} of $F$ and is denoted as $F^{-1}$.
In this case, $\cC$ and $\cD$ is called \termdef{isomorphic}.

The horizontal composite of two natural transformations can also be defined.
We write the horizontal composite of $\alpha \colon F \nto F'$ and $\beta \colon G \nto G'$
(where $F,F' \colon \cC \to \cD$ and $G,G' \colon \cD \to \cE$)
as $\beta \b \alpha$, which is depicted by
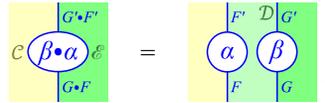
\begin{alignat}{1}
 \InsertPDF{basic_natural_bullet.pdf} \raisebox{1em}{.}
 \label{eq:basic_natural_bullet}
\end{alignat}
Specifically, it is defined as $\beta \b \alpha \coloneqq \{ G' \alpha_a \c \beta_{Fa} \}_{a \in \cC}
= \{ \beta_{F'a} \c G \alpha_a \}_{a \in \cC}$.
Its each component is written as
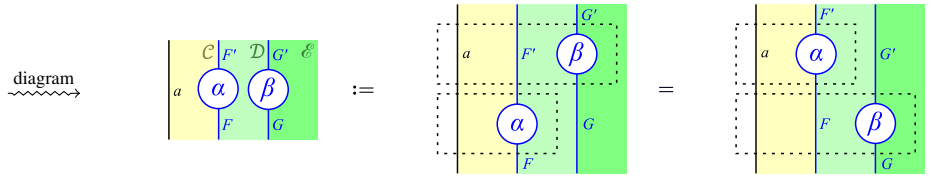
\begin{alignat}{1}
 \lefteqn{ (\beta \b \alpha)_a \coloneqq G' \alpha_a \c \beta_{Fa}
 = \beta_{F'a} \c G \alpha_a } \nonumber \\
 &\diagram\qquad 
 \InsertMidPDF{basic_natural_trans_alpha_beta_a.pdf},
 \label{eq:basic_natural_trans_alpha_beta_a}
\end{alignat}
where the equality follows from naturality of $\beta$.
Thus, the following equation holds:
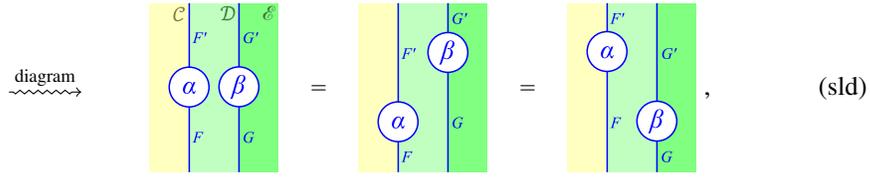
\begin{alignat}{1}
 \lefteqn{\beta \b \alpha = (G' \b \alpha) \c (\beta \b F) = (\beta \b F') \c (G \b \alpha)}
 \nonumber \\
 &\diagram\qquad 
 \InsertMidPDF{basic_bullet_sliding.pdf},
 \tag{sld}
 \label{eq:sliding}
\end{alignat}%
where $G' \b \alpha$ means $\id_{G'} \b \alpha$ and $\beta \b F$ means $\beta \b \id_F$.
Note that the horizontal composite with the identity natural transformation is often denoted like this.
This equation is also called the \termdef{sliding rule}.
Equation~\eqref{eq:sliding} intuitively means that $\alpha$ and $\beta$ can be moved freely
along the wires.
Since any morphism can be regarded as a natural transformation,
Eq.~\eqref{eq:basic_natural_trans_string2} can be regarded as a special case of
Eq.~\eqref{eq:sliding}.

In what follows, when we write something like $\beta \b \alpha$ or $\beta \c \alpha$
for natural transformations $\alpha$ and $\beta$, we implicitly assume that
they are horizontally or vertically composable.
The same applies to functors and morphisms.

\begin{proposition}{}{}
 For any categories $\cC$, $\cD$, and $\cE$, functors $F,F',F'' \colon \cC \to \cD$ and
 $G,G',G'' \colon \cD \to \cE$, and natural transformations $\alpha \colon F \nto F'$,
 $\alpha' \colon F' \nto F''$, $\beta \colon G \nto G'$, and $\beta' \colon G' \nto G''$,
 we have
 \begin{alignat}{1}
  (\beta' \c \beta) \b (\alpha' \c \alpha) = (\beta' \b \alpha') \c (\beta \b \alpha)
  &\qquad
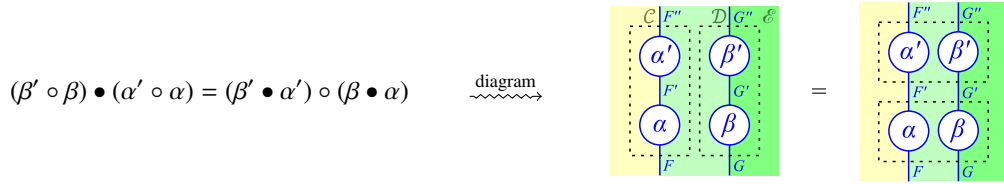
\diagram\qquad 
  \InsertMidPDF{basic_bullet_circ.pdf}.
  \label{eq:basic_bullet_circ}
 \end{alignat}
\end{proposition}
\begin{proof}
 From the sliding rule of Eq.~\eqref{eq:sliding}, we have
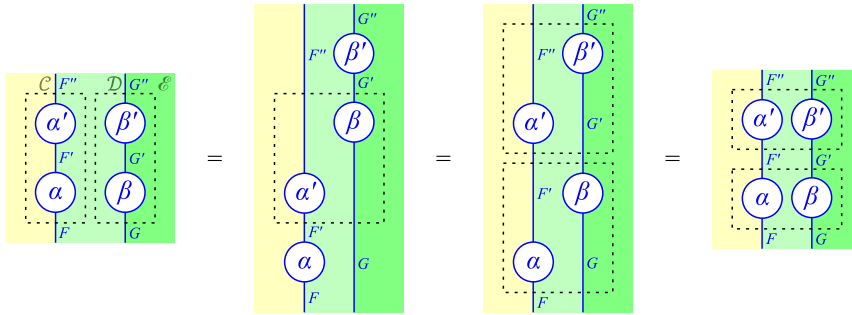
 \begin{alignat}{1}
  \footnoteinset{-3.19}{1.7}{\eqref{eq:basic_natural_trans_alpha_beta_a}}{%
  \footnoteinset{0.12}{1.7}{\eqref{eq:basic_natural_trans2_def}}{%
  \footnoteinset{3.43}{1.7}{\eqref{eq:basic_natural_trans_alpha_beta_a}}{%
  \footnoteinset{0.12}{-1.95}{\eqref{eq:basic_natural_trans_alpha_beta_a}}{%
  \footnoteinset{3.43}{-1.95}{\eqref{eq:basic_natural_trans2_def}}{%
  \InsertPDF{basic_bullet_circ_proof.pdf}}}}}} \raisebox{1em}{.}
  \label{eq:basic_bullet_circ_proof}
 \end{alignat}
\end{proof}

\begin{lemma}{}{BasicNatNowired}
 For two functors $F,G \colon \cC \to \cC$ and two natural transformations
 $\alpha \colon F \nto \id_\cC$, $~\beta \colon \id_\cC \nto G$,
 we have $\beta \alpha = \beta \b \alpha = \alpha \b \beta$.
\end{lemma}
\begin{proof}
 $\beta \alpha = \beta \b \alpha$ can be shown by
 \begin{alignat}{1}
  \InsertPDF{basic_natural_nowired.pdf} \raisebox{1em}{,}
  \label{eq:basic_natural_nowired}
 \end{alignat}
 whose corresponding mathematical formula is
 \begin{alignat}{1}
  \beta \alpha &= (\beta \b \id_\cC) \c (\id_\cC \b \alpha)
  = (\beta \c \id_\cC) \b (\id_\cC \c \alpha) = \beta \b \alpha.
 \end{alignat}
 Similarly, $\beta \alpha = \alpha \b \beta$ can be shown by
 \begin{alignat}{1}
  \InsertPDF{basic_natural_nowired2.pdf} \raisebox{1em}{,}
  \label{eq:basic_natural_nowired2}
 \end{alignat}
 whose corresponding mathematical formula is
 \begin{alignat}{1}
  \beta \alpha &= (\id_\cC \b \beta) \c (\alpha \b \id_\cC)
  = (\id_\cC \c \alpha) \b (\beta \c \id_\cC) = \alpha \b \beta.
 \end{alignat}
\end{proof}

\begin{proposition}{}{BasicFuncSimeqFF}
 If a functor $F \colon \cC \to \cD$ induces a category equivalence $\cC \simeq \cD$,
 then $F$ is fully faithful.
\end{proposition}
\begin{proof}
 By the definition of category equivalence, there exist a functor $G \colon \cD \to \cC$
 and natural isomorphisms $\phi \colon G \b F \cong \id_\cC$ and
 $\psi \colon F \b G \cong \id_\cD$.
 $\phi$ satisfies $\phi \phi^{-1} = \id_{\id_\cC}$ and $\phi^{-1} \phi = \id_{G \b F}$, i.e.,
 \begin{alignat}{1}
  \InsertPDF{basic_simeq_FF_full_id.pdf} \raisebox{1em}{,}
  \label{eq:basic_simeq_FF_full_id}
 \end{alignat}
 where the white and blue diamond blocks represent $\phi$ and $\phi^{-1}$, respectively.
 Similarly, $\psi$ satisfies $\psi \psi^{-1} = \id_{\id_\cD}$, i.e.,
 \begin{alignat}{1}
  \InsertPDF{basic_simeq_FF_full_id2.pdf} \raisebox{1em}{,}
  \label{eq:basic_simeq_FF_full_id2}
 \end{alignat}
 where the white and blue rectangular blocks represent $\psi$ and $\psi^{-1}$, respectively.

 Arbitrarily choose $g \in \cD(Fa,Fb)$ (where $a,b \in \cC$ are also arbitrary);
 it suffices to show that there exists a unique $\ol{g} \in \cC(a,b)$ that satisfies
 Eq.~\eqref{eq:basic_full_faithful}.
 If such $\ol{g}$ exists, then
 \begin{alignat}{1}
  \footnoteinset{-2.19}{0.3}{\eqref{eq:basic_simeq_FF_full_id}}{%
  \footnoteinset{1.13}{0.3}{\eqref{eq:basic_full_faithful}}{%
  \InsertPDF{basic_simeq_FF_g_uniq.pdf}}}
  \label{eq:basic_simeq_FF_g_uniq}
 \end{alignat}
 uniquely determines $\ol{g}$.
 Such $\ol{g}$ satisfies $F \ol{g} = g$ since
 \begin{alignat}{1}
  \footnoteinset{-3.38}{2.60}{\eqref{eq:basic_simeq_FF_full_id2}}{%
  \footnoteinset{1.78}{2.60}{\eqref{eq:basic_simeq_FF_full_id}}{%
  \footnoteinsets{-3.38}{-1.10}{\eqref{eq:basic_natural_nowired}}{\eqref{eq:basic_natural_nowired2}}{%
  \footnoteinset{0.59}{-1.10}{\eqref{eq:basic_simeq_FF_full_id}}{%
  \footnoteinset{4.56}{-1.10}{\eqref{eq:basic_simeq_FF_full_id2}}{%
  \InsertPDF{basic_simeq_FF_full.pdf}}}}}} \raisebox{1em}{,}
  \label{eq:basic_simeq_FF_full}
 \end{alignat}
 where we use the fact that the areas enclosed by auxiliary lines represent
 $\phi^{-1} \phi = \id_{G \b F}$.
\end{proof}

\subsubsection{Functor categories}

For two categories $\cC$ and $\cD$, there is a functor category, denoted by $\Func{\cC}{\cD}$,
whose objects are functors from $\cC$ to $\cD$ and whose morphisms are natural transformations.
A natural transformation $\alpha \colon F \nto G$ with $F,G \colon \cC \to \cD$
can be represented as a morphism in $\Func{\cC}{\cD}$ as follows:
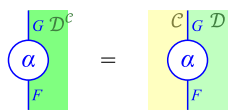
\begin{alignat}{1}
 \InsertPDF{basic_functor_category.pdf} \raisebox{1em}{.}
 \label{eq:basic_functor_category}
\end{alignat}
In this paper, when considering the functor category $\Func{\cC}{\cD}$, it is often assumed
without notice that $\cC$ is small (to ensure that $\Func{\cC}{\cD}$ is locally small).

\begin{ex}{pre-composition with a functor}{FuncPreComp}
 Pre-composition with a functor $F \colon \cC \to \cD$ is denoted by $\Endash \b F$.
 For any category $\cE$, this becomes a functor from $\Func{\cD}{\cE}$ to $\Func{\cC}{\cE}$.
 In diagrams, $\Endash \b F$ corresponds to the action of placing $F$ on the right side.
 The action on morphisms of $F$ can be expressed as follows
 (intuitively, it works like ``drawing the wire $F$ on the right side'').
 \begin{alignat}{2}
  \InsertMidPDF{basic_functor_precomposite.pdf}
  &\qquad\xmapsto{\Endash \b F}\qquad
  \InsertMidPDF{basic_functor_precomposite2.pdf}.
  \label{eq:basic_natural_trans_functorprecomposite}
 \end{alignat}
 Also, for a natural transformation $\gamma \colon F \nto G$ with $F,G \colon \cC \to \cD$,
 pre-composition with $\gamma$ can be denoted as $\Endash \b \gamma$,
 which becomes a natural transformation from $\Endash \b F$ to $\Endash \b G$.
\end{ex}

\begin{ex}{post-composition with a functor}{FuncPostComp}
 Similarly to Example~\ref{ex:FuncPreComp},
 post-composition with a functor $F \colon \cD \to \cE$ is denoted by $F \b \Endash$.
 For any category $\cC$, $F \b \Endash$ becomes a functor from
 $\Func{\cC}{\cD}$ to $\Func{\cC}{\cE}$.
 In diagrams, $F \b \Endash$ corresponds to the action of placing $F$ on the left side.
 Similarly, for a natural transformation $\gamma \colon F \nto G$ with $F,G \colon \cD \to \cE$,
 post-composition with $\gamma$ is denoted by $\gamma \b \Endash$,
 which becomes a natural transformation from $F \b \Endash$ to $G \b \Endash$.
 In this paper, when representing $F \b \Endash$ and/or $\gamma \b \Endash$ in diagrams,
 we often simply use labels such as ``$F$'' and/or ``$\gamma$'' to represent them like this:
 \begin{alignat}{1}
  \InsertPDF{basic_functor_postcomposite.pdf} \raisebox{1em}{.}
  \label{eq:basic_functor_postcomposite}
 \end{alignat}
 When confusion is likely to occur,
 we use labels such as ``$F \b \Endash$'' and/or ``$\gamma \b \Endash$''.
\end{ex}

For two functors $F \colon \cC \to \cD$ and $G \colon \cD \to \cE$,
we can introduce a diagram that represents the identity natural transformation $\id_{G \b F}$
as shown in the following left- and right-hand sides.
\begin{alignat}{1}
 \InsertPDF{basic_functor_precomp_cross.pdf} \raisebox{1em}{,}
 \label{eq:basic_functor_precomp_cross}
\end{alignat}
where on each side, $\id_{G \b F}$ is represented as a diagram where two wires cross.
While these expressions may feel like they just complicate the central expression,
this kind of notation can be convenient when you want to replace the wire $\Endash \b F$
with the wire $F$ in the diagram.

\begin{ex}{evaluation functors}{FuncEval}
 For any object $d$ in a category $\cD$, we can define a functor $\Endash \b d$
 by considering the case of $\cC = \cOne$ in Example~\ref{ex:FuncPreComp}.
 $\Endash \b d$ is called an \termdef{evaluation functor} and is written as $\ev_d$.
 The functor $\ev_d \colon \Func{\cD}{\cE} \to \cE$ maps
 each functor $F \colon \cD \to \cE$ to the object $Fd \in \cE$
 and each natural transformation $\alpha \colon F \nto G$ with $F,G \colon \cD \to \cE$
 to the morphism $\alpha_d \in \mor \cE$.
 $\ev_d \alpha = \alpha \b d$ can be represented by
 \begin{alignat}{1}
  \footnoteinset{1.11}{0.3}{\eqref{eq:basic_natural_trans_def}}{%
  \InsertPDF{basic_functor_eval.pdf}} \raisebox{1em}{.}
  \label{eq:basic_functor_eval}
 \end{alignat}
\end{ex}

\begin{ex}{diagonal functors}{FuncDelta}
 When considering the case of $\cD = \cOne$ in Example~\ref{ex:FuncPreComp},
 $F$ is a functor from $\cC$ to $\cOne$, and thus is the unique functor ${!}$,
 which maps all morphisms in $\cC$ to $\id_*$
 (see Subsubsection~\ref{subsubsec:category_category_functor}).
 The functor $\Endash \b {!} \colon \cE \to \Func{\cC}{\cE}$ is called a \termdef{diagonal functor}
 and is denoted by $\Delta_\cC$.
 $\Delta_\cC f = f \b {!}$ is represented by
 \begin{alignat}{1}
  \InsertPDF{basic_functor_delta.pdf} \raisebox{1em}{,}
  \label{eq:basic_functor_delta}
 \end{alignat}
 where the gray dotted wire is ${!}$ (recall Eq.~\eqref{eq:basic_functor_terminal}).
\end{ex}

\begin{lemma}{}{FullFaithfulPostcomp}
 For categories $\cC$, $\cD$, and $\cE$, if a functor $H \colon \cD \to \cE$ is fully faithful,
 then $H \b \Endash \colon \Func{\cC}{\cD} \to \Func{\cC}{\cE}$ is also fully faithful.
 That is, for any natural transformation $\tau \colon H \b F \nto H \b G$ with
 any $F,G \colon \cC \to \cD$,
 there exists a natural transformation $\ol{\tau} \colon F \nto G$ such that
 \begin{alignat}{1}
  \InsertPDF{basic_func_fullfaithful_tau_univ.pdf} \raisebox{1em}{.}
  \label{eq:basic_func_fullfaithful_tau_univ}
 \end{alignat}
\end{lemma}
\begin{proof}
 We show Eq.~\eqref{eq:basic_func_fullfaithful_tau_univ}.
 Since $H$ is fully faithful, for each $c \in \cC$ there exists a unique
 $\ol{\tau_c} \in \cD(Fc,Gc)$ satisfying
 \begin{alignat}{1}
  \footnoteinset{-0.07}{0.3}{\eqref{eq:basic_full_faithful}}{%
  \InsertPDF{basic_func_fullfaithful_tau.pdf}}.
  \label{eq:basic_func_fullfaithful_tau}
 \end{alignat}
 Therefore, it suffices to show that $\ol{\tau} \coloneqq \{ \ol{\tau_c} \}_{c \in \cC}$
 is a natural transformation.
 We have for any $f \in \cC(c,c')$,
 \begin{alignat}{1}
  \footnoteinset{-3.57}{0.3}{\eqref{eq:basic_func_fullfaithful_tau}}{%
  \footnoteinset{0.00}{0.3}{\eqref{eq:sliding}}{%
  \footnoteinset{3.57}{0.3}{\eqref{eq:basic_func_fullfaithful_tau}}{%
  \InsertPDF{basic_func_fullfaithful_tau_nat.pdf}}}}.
  \label{eq:basic_func_fullfaithful_tau_nat}
 \end{alignat}
 Thus, since $H$ is fully faithful, we have
 \begin{alignat}{1}
  \InsertPDF{basic_func_fullfaithful_tau_nat2.pdf} \raisebox{1em}{.}
  \label{eq:basic_func_fullfaithful_tau_nat2}
 \end{alignat}
 Therefore, $\ol{\tau}$ is a natural transformation.
\end{proof}

Let us apply the diagram introduced in Eq.~\eqref{eq:basic_functor_precomp_cross}
to the evaluation functor $\ev_d$ from Example~\ref{ex:FuncEval}.
The identity natural transformation $\id_{\ev_d}$ can be depicted as
\begin{alignat}{1}
 \id_{\ev_d} &\qquad\diagram\qquad
 \InsertMidPDF{basic_functor_eval_cross.pdf},
 \label{eq:basic_functor_eval_cross}
\end{alignat}
where the left-hand side represents
\begin{alignat}{1}
 \id_{\ev_d} = \{ \id_{\ev_d G} \}_{G \in \Func{\cD}{\cE}} &\qquad\diagram\qquad
 \InsertMidPDF{basic_functor_eval_cross_def.pdf}
 \label{eq:basic_functor_eval_cross_def}
\end{alignat}
and likewise for the right-hand side.
Equation~\eqref{eq:basic_functor_eval_cross} may be useful when you want to swap the wire $\ev_d$
and the wire $d$.
\begin{alignat}{1}
 \ev_d \alpha = \alpha_d &\qquad\diagram\qquad
 \InsertMidPDF{basic_functor_eval_cross_nat.pdf}
 \label{eq:basic_functor_eval_cross_nat}
\end{alignat}
obviously holds.
Also, the first equality of Eq.~\eqref{eq:basic_functor_eval} can be shown by
\begin{alignat}{1}
 \footnoteinset{-2.92}{0.3}{\eqref{eq:basic_functor_eval_cross}}{%
 \footnoteinset{0.12}{0.3}{\eqref{eq:basic_functor_eval_cross_nat}}{%
 \footnoteinset{3.03}{0.3}{\eqref{eq:basic_functor_eval_cross}}{%
 \InsertPDF{basic_functor_eval_proof.pdf}}}} \raisebox{1em}{.}
 \label{eq:basic_functor_eval_proof}
\end{alignat}
These diagrams for the evaluation functors $\ev_d = \Endash \b d$ can be generalized
to the functor $\Endash \b F$ with any $F \colon \cC \to \cD$.
For example, as a generalization of Eq.~\eqref{eq:basic_functor_eval_cross_nat},
for any $F,F' \colon \cC \to \cD$, $~G,G' \colon \cD \to \cE$,
$~\alpha \colon F \nto F'$, and $\beta \colon G \nto G'$,
$\beta \b \alpha \colon G \b F \nto G' \b F'$ can be depicted by
\begin{alignat}{1}
 \InsertPDF{basic_functor_cross_nat.pdf} \raisebox{1em}{.}
 \label{eq:basic_functor_cross_nat}
\end{alignat}

\subsubsection{Objects, morphisms, categories, and functors are all instances of natural transformations}

We already mentioned that morphisms can be regarded as special cases of natural transformations.
Here, we mention that functors, categories, and objects can also be regarded as special cases
of natural transformations.
This fact indicates that they can be treated uniformly as natural transformations.
Specifically, we can make the following identifications:
\begin{itemize}
 \item Functors: Any functor $F$ can be identified with its identity natural transformation $\id_F$.
 \item Categories: Any category $\cC$ can be identified with its identity functor $\id_\cC$,
       and thus can be identified with the identity natural transformation $\id_{\id_\cC}$.
 \item Objects: Any object $a$ in a category $\cC$ can be regarded as a functor from $\cOne$ to $\cC$,
       and thus can be identified with the identity natural transformation $\id_a$.
\end{itemize}
As a result, the diagrams presented in this paper can be seen as representing
natural transformations that are formed through vertical and/or horizontal composition
of natural transformations.

\subsection{Product of categories and bifunctors} \label{subsec:category_prod}

\subsubsection{Product of categories} \label{subsubsec:category_prod_prod}

In this subsection and Section~\ref{sec:end}, we employ certain diagrams
that exhibit slight variations from those previously introduced.

An object $\braket{c,d}$ (where $c \in \cC$ and $d \in \cD$)
in a product category $\cC \times \cD$ is represented by the following two expressions:
\begin{alignat}{1}
 \InsertPDF{basic_prod_ob.pdf} \raisebox{1em}{.}
 \label{eq:basic_prod_ob}
\end{alignat}
The left-hand side employs the same expression as before, while the right-hand side
introduces a new expression.
For convenience, we will call the latter expression ``the expression for direct products''.
As shown on the right-hand side, we represent $\braket{c,d}$
by arranging the two arrows representing $c \in \cC$ and $d \in \cD$ side by side.
Also, to indicate that it is a direct product, a gray dashed line is drawn between the two arrows.
Note that in this expression, we represent objects with arrows instead of wires;
such arrows will become useful in Section~\ref{sec:end}.

A morphism with domain $\braket{c,d}$ and codomain $\braket{c',d'}$
in $\cC \times \cD$ is represented as a pair $\braket{f,g}$,
where $f \in \cC(c,c')$ and $g \in \cD(d,d')$.
This morphism is represented by
\begin{alignat}{1}
 \InsertPDF{basic_prod_mor.pdf} \raisebox{1em}{,}
 \label{eq:basic_prod_mor}
\end{alignat}
where the right-hand side is the expression for direct products.
In this way, we represent $\braket{f,g}$
by arranging the two morphisms $f$ and $g$ side by side.
The composite of two morphisms $\braket{f,g} \colon \braket{c,d} \to \braket{c',d'}$ and
$\braket{f',g'} \colon \braket{c',d'} \to \braket{c'',d''}$ is
\begin{alignat}{1}
 \braket{f',g'} \c \braket{f,g} \coloneqq \braket{f'f,g'g}
 &\qquad\diagram\qquad
 \InsertMidPDF{basic_prod_mor_circ.pdf},
 \label{eq:basic_prod_mor_circ}
\end{alignat}
where a composite of morphisms is represented by connecting them.
This diagram can be interpreted as arranging the two morphisms $f'f$ and $g'g$ side by side.
Using Eq.~\eqref{eq:basic_prod_mor_circ} and the properties of identity morphisms,
we obtain
\begin{alignat}{1}
 \lefteqn{\braket{\id_{c'},g} \c \braket{f,\id_d} = \braket{f,g} = \braket{f,\id_{d'}} \c \braket{\id_c,g}}
 \nonumber \\
 &\diagram\qquad
 \InsertMidPDF{basic_bifunc_sliding.pdf}.
 \label{eq:basic_bifunc_sliding}
\end{alignat}

\subsubsection{Bifunctors}

The mapping of a morphism $\braket{f,g} \colon \braket{c,d} \to \braket{c',d'}$ in $\cC \times \cD$
with a bifunctor $F \colon \cC \times \cD \to \cE$ is often written by $F(f,g)$.
$F(f,g)$ is represented by the following two expressions:
\begin{alignat}{1}
 \InsertPDF{basic_bifunc.pdf} \raisebox{1em}{,}
 \label{eq:basic_func}
\end{alignat}
where the right-hand side is the expression for direct products.
In this expression, we represent the functor $F$ by two blue arrows
(with the arrowhead omitted on the right one).
Note that such an expression can be found in, for example, Refs.~\cite{Mel-2006, McCur-2009}.
The area enclosed by these two arrows represents the input to $F$;
in the case of Eq.~\eqref{eq:basic_func}, $\braket{f,g}$ is the input to $F$.

For two bifunctors $F,G \colon \cC \times \cD \to \cE$,
naturality of a natural transformation $\alpha \colon F \nto G$ can be expressed by
\begin{alignat}{1}
 \InsertPDF{basic_bifunc_nat_cd.pdf}
 \label{eq:basic_bifunc_nat_cd}
\end{alignat}
for any $f \in \cC(c,c')$ and $g \in \cD(d,d')$.
The label on the left or right arrow of each functor is often omitted.

\begin{ex}{}{FunctorBifunc}
 For any bifunctor $F \colon \cC \times \cD \to \cE$ and object $c \in \cC$,
 there exists a functor, denoted by $F(c,\Endash) \colon \cD \to \cE$, that
 \begin{itemize}
  \item maps each object $d \in \cD$ to $F(c,d)$.
  \item maps each morphism $g \in \mor \cD$ to $F(\id_c,g)$,
        which we often write as $F(c,g)$.
 \end{itemize}
 Such a functor can be intuitively understood as fixing one variable of a two-variable map.
 $F(c,\Endash)$ can be represented as
 \begin{alignat}{1}
  F(c,\Endash) &\qquad\diagram\qquad
  \InsertMidPDF{basic_prod_Fc.pdf}.
  \label{eq:basic_prod_Fc}
 \end{alignat}
 In this diagram, objects and morphisms in $\cD$ are placed in regions representing $\cD$.
 For example, the morphism, $F(c,g)$, obtained by mapping a morphism $g \in \cD(d,d')$
 with this functor is represented by
 \begin{alignat}{1}
  \InsertPDF{basic_prod_Fc_mor.pdf} \raisebox{1em}{.}
  \label{eq:basic_prod_Fc_mor}
 \end{alignat}
 It is easy to see that $F(c,\Endash)$ is a functor.
 Indeed,
 \begin{alignat}{1}
  F(c,g'g) = F(c,g') \c F(c,g) &\qquad\diagram\qquad
  \InsertMidPDF{basic_prod_Fc_func.pdf}
  \label{eq:basic_prod_Fc_func}
 \end{alignat}
 holds.
 It is also clear that an identity morphism is mapped to an identity morphism.
 Similarly, for each $d \in \cD$, the functor $F(\Endash,d) \colon \cC \to \cE$ is obtained.
\end{ex}

\begin{ex}{}{BiFuncFf}
 For any bifunctor $F \colon \cC \times \cD \to \cE$ and
 morphism $f \in \cC(c,c')$,
 let us define $F(f,\Endash) \coloneqq \{ F(f,d) \}_{d \in \cD}$.
 Then, $F(f,\Endash)$ is a natural transformation from $F(c,\Endash)$ to $F(c',\Endash)$;
 indeed, we have for each $g \in \cD(d,d')$,
 \begin{alignat}{1}
  \footnoteinset{-1.99}{-1.0}{\eqref{eq:basic_bifunc_sliding}}{%
  \InsertPDF{basic_bifunc_nat1.pdf}} \raisebox{1em}{,}
  \label{eq:basic_bifunc_nat1}
 \end{alignat}
 where the second and third expressions are for direct products.
\end{ex}

\begin{ex}{}{FuncEval2}
 The evaluation functor discussed in Example~\ref{ex:FuncEval} can be extended to
 the bifunctor, denoted by $\ev \colon \Func{\cD}{\cE} \times \cD \to \cE$, that
 \begin{itemize}
  \item maps each object $\braket{K,d}$ to the object $K \b d = Kd$.
  \item maps each morphism $\braket{\alpha,f}$ to the morphism $\alpha \b f$.
 \end{itemize}
 This bifunctor is often referred to as the \termdef{evaluation functor}.
\end{ex}

\begin{lemma}{}{BifuncNat}
 Let us consider two bifunctors $F,G \colon \cC \times \cD \to \cE$.
 For a collection of morphisms
 $\alpha \coloneqq \{ \alpha_{c,d} \colon F(c,d) \to G(c,d) \}_{c \in \cC, d \in \cD}$
 in $\cE$, the following are equivalent:
 \begin{enumerate}
  \item $\alpha$ is a natural transformation from $F$ to $G$.
  \item The following two conditions hold:
        \begin{enumerate}[label=(2.\alph*),leftmargin=2em]
         \item For each $c \in \cC$, $\{ \alpha_{c,d} \}_{d \in \cD}$ is a natural
               transformation from $F(c,\Endash)$ to $G(c,\Endash)$.
         \item For each $d \in \cD$, $\{ \alpha_{c,d} \}_{c \in \cC}$ is a natural
               transformation from $F(\Endash,d)$ to $G(\Endash,d)$.
        \end{enumerate}
 \end{enumerate}
\end{lemma}

\begin{proof}
 $(1) \Rightarrow (2)$:
 Condition~(2.a) is equivalent to satisfying
 \begin{alignat}{1}
  \InsertPDF{basic_bifunc_nat_c.pdf}
  \label{eq:basic_bifunc_nat_c}
 \end{alignat}
 for any $g \in \cD(d,d')$ and $c \in \cC$, and 
 Condition~(2.b) is equivalent to satisfying
 \begin{alignat}{1}
  \InsertPDF{basic_bifunc_nat_d.pdf}
  \label{eq:basic_bifunc_nat_d}
 \end{alignat}
 for any $f \in \cC(c,c')$ and $d \in \cD$.
 Substituting $f = \id_c$ into Eq.~\eqref{eq:basic_bifunc_nat_cd} yields Eq.~\eqref{eq:basic_bifunc_nat_c}.
 Also, substituting $g = \id_d$ into Eq.~\eqref{eq:basic_bifunc_nat_cd} yields
 Eq.~\eqref{eq:basic_bifunc_nat_d}.

 $(2) \Rightarrow (1)$:
 For any $f \in \cC(c,c')$ and $g \in \cD(d,d')$, we have%
 \footnote{In mathematical notation,
 \begin{alignat}{1}
  G(f,g) \c \alpha_{c,d} &= G(f,d') \c G(c,g) \c \alpha_{c,d}
  = G(f,d') \c \alpha_{c,d'} \c F(c,g) \nonumber \\
  &= \alpha_{c',d'} \c F(f,d') \c F(c,g)
  = \alpha_{c',d'} \c F(f,g).
  \label{eq:basic_bifunc_nat}
 \end{alignat}
 }
 \begin{alignat}{1}
  \footnoteinset{-1.98}{2.1}{\eqref{eq:basic_bifunc_sliding}}{%
  \footnoteinset{1.98}{2.1}{\eqref{eq:basic_bifunc_nat_c}}{%
  \footnoteinset{-1.98}{-1.45}{\eqref{eq:basic_bifunc_nat_d}}{%
  \footnoteinset{1.98}{-1.45}{\eqref{eq:basic_bifunc_sliding}}{%
  \InsertPDF{basic_bifunc_nat_proof.pdf}}}}} \raisebox{1em}{.}
  \label{eq:basic_bifunc_nat_proof}
 \end{alignat}
 Therefore, Eq.~\eqref{eq:basic_bifunc_nat_cd} holds, which means that
 $\alpha$ is a natural transformation.
\end{proof}

Intuitively, Condition~(1) can be interpreted as allowing morphisms in $\cC$ and morphisms in $\cD$
to pass through $\alpha$ as shown in Eq.~\eqref{eq:basic_bifunc_nat_cd},
while Condition~(2) can be interpreted as allowing morphisms in $\cC$ to pass through
$\alpha$ as shown in Eq.~\eqref{eq:basic_bifunc_nat_c},
and also allowing morphisms in $\cD$ to pass through
$\alpha$ as shown in Eq.~\eqref{eq:basic_bifunc_nat_d}.
It should be easy to conceive that these conditions are equivalent.

Note that it might be easier to visually understand if we represent the natural transformation $\alpha$
as a pair of blocks, and each component $\alpha_{c,d}$ as
\begin{alignat}{1}
 \InsertPDF{basic_bifunc_nat_block.pdf} \raisebox{1em}{.}
 \label{eq:basic_bifunc_nat_block}
\end{alignat}
Due to the fact that we represents a functor with a pair of arrows, here we represents
a natural transformation with a pair of blocks.
Naturality of Eq.~\eqref{eq:basic_bifunc_nat_cd} is represented as
\begin{alignat}{1}
 \lefteqn{ G(f,g) \c \alpha_{c,d} = \alpha_{c',d'} \c F(f,g) } \nonumber \\
 &\diagram\qquad
 \InsertMidPDF{basic_bifunc_nat_proof_block.pdf}
 \label{eq:basic_bifunc_nat_proof_block}
\end{alignat}


\section{The Yoneda lemma and universal properties} \label{sec:repr}

Adjunctions, (co)limits, and Kan extensions, which will be explained in the following sections,
can be understood as certain types of universal morphisms.
In this sense, a universal property, which universal morphisms possess, is a vital concept
in category theory.
A universal property is closely related to representability, which can be proven
through the Yoneda lemma.
The Yoneda lemma is ``arguably the most important result in category theory'' \cite{Rie-2017}.
After proving the Yoneda lemma using string diagrams, we state some basic properties
of universal morphisms.

\subsection{The Yoneda lemma} \label{subsec:repr_yoneda}

\subsubsection{Preliminaries: diagrams for set-valued functors and natural transformations between them}
\label{subsubsec:repr_yoneda_preliminary1}

\myparagraph{Functor $\yoneda{c}$ and natural transformation $\yoneda{p}$}

\noindent
Let $\Set$ denote the category whose objects are sets and whose morphisms from a set $X$
to a set $Y$ are maps from $X$ to $Y$.
For each $c \in \cC$, we can define a set-valued functor $\cC(c,\Endash) \colon \cC \to \Set$
that
\begin{itemize}
 \item maps each object $a$ in $\cC$ to the hom-set $\cC(c,a)$.
 \item maps each morphism $f \colon a \to b$ in $\cC$ to the map
       $f \c \Endash \colon \cC(c,a) \ni g \mapsto fg \in \cC(c,b)$.
\end{itemize}
We often write $\cC(c,\Endash)$ as $\yoneda{c}$.
In what follows, we represent $\cC(c,a)$ with
\begin{alignat}{1}
 \InsertPDF{basic_natural_pre_composite_hom.pdf} \raisebox{1em}{.}
 \label{eq:basic_natural_pre_composite_hom}
\end{alignat}
Intuitively, the dotted box acts like a hole, which takes any morphism with domain $c$ and codomain $a$.
It is convenient to represent the functor $\yoneda{c}$ with
\begin{alignat}{1}
 \InsertPDF{basic_functor_yoneda_c.pdf} \raisebox{1em}{.}
 \label{eq:basic_functor_yoneda_c}
\end{alignat}
The right-hand side of this equation will contain objects or morphisms in $\cC$.
Specifically, $\yoneda{c} \b a = \cC(c,a)$ with $a \in \cC$ and
$\yoneda{c} \b f = f \c \Endash$ with $f \in \cC(a,b)$ are represented by%
\footnote{We identify a map $h \colon X \to Y$ with the indexed collection
$\{ h(x) \}_{x \in X}$.
Then, $f \c \Endash$ is identified with $\{ fg \}_{g \in \cC(c,a)}$.}
\begin{alignat}{1}
 \InsertPDF{basic_functor_postcomp.pdf} \raisebox{1em}{.}
 \label{eq:basic_functor_postcomp}
\end{alignat}
The last expression can be easily understood as the following map:
\begin{alignat}{1}
 \InsertMidPDF{basic_functor_postcomp_f1.pdf}
 &\qquad\xmapsto{f \c\Endash}\qquad
 \InsertMidPDF{basic_functor_postcomp_f2.pdf}.
 \label{eq:basic_functor_postcomp_f}
\end{alignat}
We can interpret $f \c \Endash$ as the map applying $f$ from the top side in a diagram.

For each morphism $p \in \cC(c,d)$, a natural transformation
$\yoneda{p} \colon \yoneda{d} \nto \yoneda{c}$ is defined as
\begin{alignat}{1}
 \yoneda{p} \coloneqq \{ \yoneda{p}{}_a \coloneqq \Endash \c p \in \Set(\cC(d,a),\cC(c,a)) \}_{a \in \cC},
 \label{eq:basic_nat_precomp}
\end{alignat}
where $\Endash \c p$ is the map $\cC(d,a) \ni g \mapsto gp \in \cC(c,a)$.
We represent $\yoneda{p}$ as
\begin{alignat}{1}
 \InsertPDF{basic_natural_pre_composite_p.pdf} \raisebox{1em}{.}
 \label{eq:basic_natural_pre_composite_p}
\end{alignat}
In the central expression, we represent $\Endash \c p$ in the same way as $f \c \Endash$
in Eq.~\eqref{eq:basic_functor_postcomp}.
The right-hand side is a visually easier representation of the central expression.
The dotted box on the right-hand side represents the functor $\yoneda{d}$
(recall Eq.~\eqref{eq:basic_functor_yoneda_c}).
Intuitively, it can be interpreted as the transformation of a functor $\yoneda{d}$
into a functor $\yoneda{c}$ by the pre-composition of $p$.
It would become clear if we consider placing each object in $\cC$ within the region
indicated by $\Endash$ on the right-hand side.
We have for each $f \in \cC(a,b)$,
\begin{alignat}{1}
 \footnoteinset{0.00}{0.3}{\eqref{eq:nat}}{%
 \InsertMidPDF{basic_natural_pre_composite0.pdf}}
 \qquad\myLeftrightarrow{equivalent}\qquad
 \InsertMidPDF{basic_natural_pre_composite.pdf},
 \label{eq:basic_natural_pre_composite}
\end{alignat}
and thus naturality of $\yoneda{p}$ holds.
In each of the left- and right-hand sides of the symbol ``$\myLeftrightarrow{equivalent}$'',
the left-hand side is the map that first applies the map $\Endash \c p$ and then applies
the map $\yoneda{c}(f) = f \c \Endash$,
while the right-hand side is the map that first applies the map $\yoneda{d}(f) = f \c \Endash$
and then applies the map $\Endash \c p$.
Both maps are equal to the map $f \c \Endash \c p \colon \cC(d,a) \ni h \mapsto fhp \in \cC(c,b)$.

A set-valued functor from an opposite category $\cC^\op$ is called a \termdef{presheaf}.
By replacing category $\cC$ with its opposite category $\cC^\op$,
we can define the presheaf $\yonedaop{c} \coloneqq \cC(\Endash,c) \colon \cC^\op \to \Set$.
It would be easy to understand if we represent functor $\yonedaop{c}$ with
\begin{alignat}{1}
 \InsertPDF{basic_functor_yonedaop_c.pdf} \raisebox{1em}{,}
 \label{eq:basic_functor_yonedaop_c}
\end{alignat}
which can be seen as the ``upside-down'' version of Eq.~\eqref{eq:basic_functor_yoneda_c}.
Note that, to avoid confusion, we have changed the background color of categories $\cC$ and $\cC^\op$.
$\yonedaop{c} \b a = \cC(a,c)$ with $a \in \cC^\op$ and
$\yonedaop{c} \b f = \Endash \c f$ with $f \in \cC^\op(a,b) = \cC(b,a)$ are represented by
\begin{alignat}{1}
 \InsertPDF{basic_functor_precomp.pdf} \raisebox{1em}{.}
 \label{eq:basic_functor_precomp}
\end{alignat}
We can interpret $\Endash \c f$ as the map applying $f \in \cC(b,a)$ from the bottom side in a diagram.
The domain $\cC(a,c)$ of $\Endash \c f$ corresponds to the codomain $a$ of $f$,
resulting in an upside-down diagram.
Using a similar approach to $\yoneda{p}$, we can define the natural transformation
$\yonedaop{q}$ for each $q \in \cC(c,d)$ as
\begin{alignat}{1}
 \yonedaop{q} \coloneqq \{ \yonedaop{q}{}_a \coloneqq q \c \Endash \in
 \Set(\yonedaop{c}(a),\yonedaop{d}(a)) \}_{a \in \cC}.
 \label{eq:basic_nat_postcomp}
\end{alignat}

Similar to Eq.~\eqref{eq:basic_functor_yoneda_c}, it would be natural to represent
the hom functor $\cC(\Endash,\Enndash)$ as
\begin{alignat}{1}
 \InsertPDF{basic_functor_Hom2.pdf} \raisebox{1em}{.}
 \label{eq:basic_functor_Hom2}
\end{alignat}
The hom functor $\cC(\Endash,\Enndash)$ maps an object $\braket{a,b} \in \cC^\op \times \cC$
to the hom-set $\cC(a,b)$ and a morphism $\braket{f,g} \in \mor (\cC^\op \times \cC)$
(with $f \in \cC(a',a)$ and $g \in \cC(b,b')$) to the map
$g \c \Endash \c f \colon \cC(a,b) \ni h \mapsto ghf \in \cC(a',b')$,
which can be represented by
\begin{alignat}{1}
 \InsertPDF{basic_functor_Hom.pdf} \raisebox{1em}{.}
 \label{eq:basic_functor_Hom}
\end{alignat}

\myparagraph{Rules for dotted box notation}

\noindent
Using the following examples, we clarify the rules for notation using a dotted box
(referred to as dotted box notation).
\begin{alignat}{4}
 &\text{(A) hom-set:} &~~~& \text{(B) map:}
 &~~~& \text{(C) set-valued functor:} &~~~& \text{(D) natural transformation:} \nonumber \\
 &\quad \text{$\cC(c,a)$ [Eq.~\eqref{eq:basic_natural_pre_composite_hom}]}
 &&~~~ \text{$f \c \Endash$ [Eq.~\eqref{eq:basic_functor_postcomp}]}
 &&~~~ \text{$\yoneda{c}$ [Eq.~\eqref{eq:basic_functor_yoneda_c}]}
 &&~~~ \text{$\yoneda{p}$ [Eq.~\eqref{eq:basic_natural_pre_composite_p}]} \nonumber \\
 &\quad \InsertMidPDF{basic_setvalued_ob.pdf}
 &&\quad \InsertMidPDF{basic_setvalued_mor.pdf}
 &&\quad \InsertMidPDF{basic_setvalued_functor.pdf}
 &&\quad \InsertMidPDF{basic_setvalued_nat.pdf}
\end{alignat}
As discussed below, depending on whether the wire ``$\Endash$'' is included and
whether blocks other than the dotted box are included, it can be classified into
$2 \times 2 = 4$ types of diagrams, such as the above four examples (A), (B), (C), and (D).
\begin{enumerate}[label=Rule \arabic*:, labelwidth=-2em]
 \item A diagram that does not include the wire ``$\Endash$'' represents a hom-set
       (which is an object of $\Set$)
       when it does not include any blocks other than a dotted box, as in Example~(A).
       When it includes such blocks, as in Example~(B), it represents a map (which is a morphism of $\Set$).
 \item A diagram that includes the wire ``$\Endash$'' represents a set-valued functor
       when it does not include any blocks other than a dotted box, as in Example~(C).
       When it includes such blocks, as in Example~(D), it represents a natural transformation
       between set-valued functors (more precisely, it also represents a natural transformation
       when a wire representing a functor is included next to the dotted box)%
       \footnote{The diagram of the identity map $\id_a \c \Endash$, which is obtained by substituting
       $f = \id_a$ into the diagram of Example~(B), cannot be distinguished from the diagram of
       Example~(A).
       This corresponds to the fact that in $\Set$, the identity morphism
       $\id_a \c \Endash = \id_{\cC(c,a)}$ can be essentially identified with the object $\cC(c,a)$.
       Similarly, the diagram of the identity natural transformation $\yoneda{\id_c}$,
       which is obtained by substituting $p = \id_c$ into the diagram of Example~(D),
       cannot be distinguished from the diagram of Example~(C).
       Thus, identity maps and identity natural transformations are exceptionally represented
       by the same diagrams as the corresponding objects and functors.}.
 \item In a diagram representing a map, such as Example~(B), when a morphism $g$
       (with appropriate domain and codomain) is input into the dotted box,
       and the morphism represented by the entire diagram is denoted as $s(g)$, the original diagram
       represents the map $g \mapsto s(g)$.
       For example, in Example~(B), when a morphism $g \in \cC(c,a)$ is input into the dotted box,
       the entire diagram represents the morphism $fg \in \cC(c,b)$,
       so the diagram of Example~(B) represents the map $\cC(c,a) \ni g \mapsto fg \in \cC(c,b)$
       (i.e., the map $f \c \Endash$).
 \item In a diagram such as Example~(C) or (D), when the wire ``$\Endash$'' is replaced with
       a block representing a morphism $g$, and the map represented by the entire diagram is
       denoted as $s(g)$, the original diagram represents a set-valued functor or natural transformation
       whose action on morphisms is the map $g \mapsto s(g)$.
       Note that functors and natural transformations are uniquely determined by their actions
       on morphisms.
       For example, in the diagram of Example~(C), when the wire ``$\Endash$'' is replaced with
       the block $f$, it becomes the diagram of Example~(B) (i.e., the map $f \c \Endash$).
       Therefore, the diagram of Example~(C) represents the set-valued functor (i.e., $\yoneda{c}$)
       whose action on morphisms is $\mor \cC \ni f \mapsto f \c \Endash \in \mor \Set$.
       Also, in the diagram of Example~(D), when the wire ``$\Endash$'' is replaced with the block $f$,
       it becomes the diagram of the right-hand side of Eq.~\eqref{eq:basic_natural_pre_composite}
       (i.e., the map $f \c \Endash \c p$).
       Thus, the diagram of Example~(D) represents a natural transformation (i.e., $\yoneda{p}$)
       whose action on morphisms
       is $\mor \cC \ni f \mapsto f \c \Endash \c p \in \mor \Set$%
       \footnote{As with the action on morphisms, the action on objects can be obtained by
       replacing the wire ``$\Endash$'' with a wire representing an object.
       For example, when the wire ``$\Endash$'' of the diagram of Example~(C) is replaced
       with the wire $a$, the diagram of Example~(A) is obtained.
       This is consistent with the fact that the action on objects of the functor $\yoneda{c}$
       is $\ob \cC \ni a \mapsto \cC(c,a) \in \ob \Set$.
       Also, when the wire ``$\Endash$'' of the diagram of Example~(D) is replaced with the wire $a$,
       the map $\Endash \c p$ is obtained (see Eq.~\eqref{eq:basic_natural_pre_composite_p}).
       This is consistent with the fact that the action on objects of the natural transformation
       $\yoneda{p}$ is $\ob \cC \ni a \mapsto \Endash \c p \in \mor \Set$.}.
 \item In a diagram that includes the wire ``$\Endash$'', one end of the wire is connected
       to the dotted box, and the other end is not connected anywhere.
       For example, in the diagrams of Examples~(C) and (D), the lower end of the wire ``$\Endash$''
       is connected to the dotted box, and the upper end is not connected.
\end{enumerate}
Note that diagrams that include the wire ``$\Enndash$'' in addition to the wire ``$\Endash$'',
as in the right-hand side of Eq.~\eqref{eq:basic_functor_Hom2}, can be considered,
in which case the same rules as above apply.

Dotted box notation follows the same composition rules as standard string diagrams
since it can be regarded as a mere morphism by substituting a morphism into the dotted box
(and replacing the wire ``$\Endash$'' with an object or morphism if it exists).
Intuitively, dotted box notation can be said to be a notation that is obtained by substituting
a part of a standard string diagram with a dotted box or the wire ``$\Endash$''.
When representing a set-valued functor or natural transformation as a diagram including
the wire ``$\Endash$'', we can easily check that it satisfies the conditions
as a functor or natural transformation.

\myparagraph{Typical set-valued functors and natural transroamtions betweem them}

\noindent
We show some examples of set-valued functors and natural transformations between them.

\begin{ex}{}{FuncCcG}
 For any object $c \in \cC$ and functor $G \colon \cD \to \cC$,
 we can consider the horizontal composite, $\yoneda{c} \b G \colon \cD \to \Set$, of
 $G$ and $\yoneda{c} = \cC(c,\Endash) \colon \cC \to \Set$.
 We often write $\yoneda{c} \b G$ as $\cC(c,G\Endash)$.
 For clarity, let us explicitly state this functor:
 \begin{itemize}
  \item It maps each object $d$ in $\cD$ to the hom-set $\cC(c,Gd)$.
  \item It maps each morphism  $f \colon a \to b$ in $\cD$ to the map $Gf \c \Endash$,
        which is depicted by
        \begin{alignat}{1}
         \InsertPDF{basic_natural_trans_CcG_Gf.pdf} \raisebox{1em}{.}
         \label{eq:basic_natural_trans_CcG_Gf}
        \end{alignat}
 \end{itemize}
 It would be natural to represent $\cC(c,G\Endash)$ in the following diagram
 (see Eq.~\eqref{eq:basic_functor_yoneda_c}):
 \begin{alignat}{1}
  \InsertPDF{basic_natural_trans_CcG.pdf} \raisebox{1em}{.}
  \label{eq:basic_natural_trans_CcG}
 \end{alignat}
 Indeed, replacing the wire ``$\Endash$'' of the right-hand side of Eq.~\eqref{eq:basic_natural_trans_CcG}
 with the block $f \in \cD(a,b)$ yields the right-hand side of Eq.~\eqref{eq:basic_natural_trans_CcG_Gf}.
\end{ex}

\begin{ex}{}{FuncCcG2}
 Similar to Example~\ref{ex:FuncCcG}, we often write $\yonedaop{c} \b G \colon \cD^\op \to \Set$
 (where $G \colon \cD^\op \to \cC^\op$ is the dual of $G \colon \cD \to \cC$)
 as $\cC(G\Endash,c)$.
 For each $f \in \cC^\op(a,b)$, the map $Gf \c \Endash \colon \cC^\op(c,Ga) \to \cC^\op(c,Gb)$
 is represented by
 \begin{alignat}{1}
  \InsertMidPDF{basic_natural_trans_CcG_Gf_op.pdf}
  &\qquad\myLeftrightarrow{equivalent}\qquad
  \InsertMidPDF{basic_natural_trans_CcG_Gf_op2.pdf}.
  \label{eq:basic_natural_trans_CcG_Gf_op}
 \end{alignat}
 This map can also be represented as the map $\Endash \c Gf \colon \cC(Ga,c) \to \cC(Gb,c)$
 with $Gf \in \cC(Gb,Ga)$.
 The functor $\cC(G\Endash,c)$ is depicted by
 \begin{alignat}{1}
  \cC(G\Endash,c) \quad\diagram\quad
  \InsertMidPDF{basic_natural_trans_CcG_op.pdf}.
  \label{eq:basic_natural_trans_CcG_Gf_op}
 \end{alignat}
\end{ex}

\begin{ex}{}{FuncHomFG}
 For two functors $F \colon \cE \to \cC$ and $G \colon \cD \to \cC$,
 let $\braket{F,G}$ be the functor from $\cE^\op \times \cD$ to $\cC^\op \times \cC$
 that maps each object $\braket{e,d}$ and each morphism $\braket{f,g}$ in $\cE^\op \times \cD$
 to the object $\braket{Fe,Gd}$ and the morphism $\braket{Ff,Gg}$ in $\cC^\op \times \cC$,
 respectively.
 Then, we can consider the functor
 \begin{alignat}{1}
  \cC(F\Endash,G\Enndash) &\coloneqq \cC(\Endash,\Enndash)
  \b \braket{F,G} \colon \cE^\op \times \cD \to \Set,
 \end{alignat}
 which is represented by
 \begin{alignat}{1}
  \InsertPDF{basic_natural_trans_CFG.pdf} \raisebox{1em}{.}
  \label{eq:basic_natural_trans_CFG}
 \end{alignat}
 The action on morphisms of the functor $\cC(F\Endash,G\Enndash)$ is
 \begin{alignat}{4}
  \mor (\cE^\op \times \cD) &\ni \braket{f,g}
  &&\quad\xmapsto{\braket{F,G}}\quad
  \braket{Ff,Gg}
  &&\quad\xmapsto{\cC(\Endash,\Enndash)}\quad
  &Gg \c \Endash \c Ff &\in \mor \Set,
 \end{alignat}
 where the map $Gg \c \Endash \c Ff$ is represented by
 \begin{alignat}{1}
  \InsertPDF{basic_natural_trans_CFGfg.pdf} \raisebox{1em}{.}
  \label{eq:basic_natural_trans_CFGfg}
 \end{alignat}
 Note that replacing the block $f$ on the right-hand side with the wire ``$\Endash$'' and the block $g$
 with the wire ``$\Enndash$'' yields the right-hand side of Eq.~\eqref{eq:basic_natural_trans_CFG}.
 In particular, we write $\cC(F\Endash,\id_\cC\Enndash)$ as $\cC(F\Endash,\Enndash)$,
 and $\cC(\id_\cC\Endash,G\Enndash)$ as $\cC(\Endash,G\Enndash)$.
\end{ex}

\begin{ex}{}{FuncCcGEeF}
 Consider a natural transformation from the set-valued functor $\cC(c,G\Endash)$ to
 the set-valued functor $\cC(e,F\Endash)$, where $c \in \cC$, $~e \in \cE$, $~G \colon \cD \to \cC$,
 and $F \colon \cD \to \cE$.
 As a slightly complex example, when we arbitrarily choose a morphism $f \in \cE(e,Hc)$ and
 a natural transformation $\alpha \colon H \b G \nto F$
 (where $H \colon \cC \to \cE$ is also arbitrary),
 we consider the map (i.e., a morphism in $\Set$) given by
 \begin{alignat}{1}
  \lefteqn{ \tau_x \colon \cC(c,Gx) \ni h \mapsto
  \alpha_x (H \b h) f \in \cE(e,Fx) } \nonumber \\
  &\diagram\qquad
  \InsertMidPDF{basic_natural_trans_CcG_nat.pdf}
  \label{eq:basic_natural_trans_CcG_nat}
 \end{alignat}
 for each $x \in \cD$.
 According to the rule mentioned earlier, the right-hand side of
 Eq.~\eqref{eq:basic_natural_trans_CcG_nat} represents a map.
 Specifically, when we insert a morphism $h \in \cC(c,Gx)$ into the dotted box,
 the entire diagram becomes a morphism $\alpha_x (H \b h) f \in \cE(e,Fx)$, so by Rule~3,
 it represents the map $\cC(c,Gx) \ni h \mapsto \alpha_x (H \b h) f \in \cE(e,Fx)$.
 It follows that $\tau \coloneqq \{ \tau_x \}_{x \in \cD}$ is a natural transformation
 from $\cC(c,G\Endash)$ to $\cE(e,F\Endash)$ since we have for each $g \in \cD(x,y)$,
 \begin{alignat}{1}
  \footnoteinset{-2.01}{-1.85}{\eqref{eq:sliding}}{%
  \InsertPDF{basic_natural_trans_CcG_nat_proof.pdf}} \raisebox{1em}{.}
  \label{eq:basic_natural_trans_CcG_nat_proof}
 \end{alignat}
 The natural transformation $\tau$ can be represented by
 \begin{alignat}{1}
  \InsertPDF{report_basic_natural_trans_CcG_nat2.pdf} \raisebox{1em}{,}
  \label{eq:basic_natural_trans_CcG_nat2}
 \end{alignat}
 where the area enclosed by the auxiliary line is grouped and represented as
 a block $\tau$ in the shape of ``$\raisebox{-.1em}{\InsertPDFtext{report_text_repr_def.pdf}}$''.
 Note that if you replace the wire ``$\Endash$'' in the central expression of
 Eq.~\eqref{eq:basic_natural_trans_CcG_nat2} with the wire $x$, it becomes equal to
 the right-hand side of Eq.~\eqref{eq:basic_natural_trans_CcG_nat}.
 Therefore, by Rule~4, it is understood that the central expression of
 Eq.~\eqref{eq:basic_natural_trans_CcG_nat2} represents a natural transformation, namely $\tau$,
 whose action on objects is $\ob \cD \ni x \mapsto \tau_x \in \mor \Set$
 (recall that a natural transformation is uniquely determined by its action on objects).
 Several natural transformations of this form appear throughout this paper.
 Note that this diagram can be considered to have the diagram on the right-hand side of
 Eq.~\eqref{eq:basic_natural_trans_CcG} embedded in it.
\end{ex}

\subsubsection{Preliminaries: diagrams for morphisms} \label{subsubsec:repr_yoneda_preliminary2}

In this paper, we identify any element $x \in X$ of any class $X$ with the map
$\{ * \} \ni * \mapsto x \in X$, where $\{ * \}$ is a singleton set.
Then, when $X$ is a set, we have $\Set(\{*\},X) = X$, or equivalently,
\begin{alignat}{1}
 X = \{ \{ * \} \ni * \mapsto x \in X \}_{x \in X}
 &\qquad\diagram\qquad \InsertMidPDF{basic_Set_object.pdf}.
 \label{eq:basic_Set_object}
\end{alignat}
We represent $\{ * \}$ with a gray dotted wire, as shown in this diagram.

It is worth mentioning that the identity morphism $\id_x$ can be expressed as
\begin{alignat}{1}
 \id_x &\qquad\diagram\qquad
 \InsertMidPDF{repr_yoneda_init_id.pdf},
 \label{eq:repr_yoneda_init_id}
\end{alignat}
where the black circles (dots) represent $\id_x$.
In the right-hand side of this diagram, $\id_x$ is represented as a map
(i.e., a morphism in $\Set$), $\{ * \} \ni * \mapsto \id_x \in \yoneda{x}(x) = \cC(x,x)$.
Here, as mentioned earlier, we identify $\id_x$ with the map $* \mapsto \id_x$.
A good intuitive interpretation of the diagram in Eq.~\eqref{eq:repr_yoneda_init_id} might be
to think of the right-hand side as ``the wire $x$ on the left-hand side, bent at the position of
the black circle to change into the wire $\yoneda{x}$''.

In the following, we explain that for any chosen morphism $a \in \cC(x,c)$
(with any $x,c \in \cC$), the following equation holds, and that the equation and diagram correspond.
\begin{alignat}{1}
 \lefteqn{ a = \yoneda{x}(a) \c \id_x = a = \yoneda{a}{}_c \c \id_c } \nonumber \\
 &\diagram\qquad
 \InsertMidPDF{repr_yoneda_init_summary.pdf},
 \label{eq:repr_yoneda_init_summary}
\end{alignat}
where, in this diagram, the morphism $a$ is represented as a round block for visual clarity.
(note that there is no semantic difference between a square block and a round block).
The left and right black circles represent $\id_x$ and $\id_c$, respectively.
In this section, we will often represent identity morphisms as black circles.

First, we derive the first equality of Eq.~\eqref{eq:repr_yoneda_init_summary}
using the following equation.
\begin{alignat}{1}
 a = a \c \id_x = \yoneda{x}(a) \c \id_x &\qquad\diagram\qquad
 \InsertMidPDF{repr_yoneda_init_summary0.pdf}.
 \label{eq:repr_yoneda_init_summary0}
\end{alignat}
The first equality is obvious since $\id_x$ is an identity morphism.
Also, the second equality follows from $\yoneda{x}(a) = a \c \Endash$,
which gives $\yoneda{x}(a) \c \id_x = a \c \id_x = a$.
Note that the right-hand side of the diagram in Eq.~\eqref{eq:repr_yoneda_init_summary0},
which represents $\yoneda{x}(a) \c \id_x$, may intuitively
understand if you think of it as the map (i.e., the morphism in $\Set$)
\begin{alignat}{1}
 \{ * \} \ni * \quad\xmapsto{* \mapsto \id_x}\quad \id_x \quad\xmapsto{\yoneda{x}(a)}\quad
 a \in \yoneda{x}(c) = \cC(x,c),
\end{alignat}
which is equal to the map $* \mapsto a$, and is identified with $a$.

Next, we explain the second equality of Eq.~\eqref{eq:repr_yoneda_init_summary}.
The equation $\yoneda{x}(a) \c \id_x = a$ has already been derived.
In the third expression in Eq.~\eqref{eq:repr_yoneda_init_summary}, similar to $\id_x$
in Eq.~\eqref{eq:repr_yoneda_init_id}, $a$ is represented as the map
$\{ * \} \ni * \mapsto a \in \yoneda{x}(c) = \cC(x,c)$.

Finally, we derive the last equality of Eq.~\eqref{eq:repr_yoneda_init_summary}.
From Eq.~\eqref{eq:basic_nat_precomp}, the map $\yoneda{a}{}_c$ is equal to $\Endash \c a$,
we have $\yoneda{a}{}_c \c \id_c = \id_c \c a = a$.
The right-hand side of the diagram can be considered as representing the map
\begin{alignat}{1}
 \{ * \} \ni * \quad\xmapsto{* \mapsto \id_c}\quad \id_c \quad\xmapsto{\yoneda{a}{}_c}\quad
 a \in \yoneda{x}(c) = \cC(x,c).
\end{alignat}

We give an intuitive interpretation of the diagram in Eq.~\eqref{eq:repr_yoneda_init_summary}.
Let us think of the wire $\yoneda{c}$ and the wire $c$ (as well as the wire $\yoneda{x}$
and the wire $x$), and the block $\yoneda{a}$ and the block $a$, as being similar things.
Then, it can be interpreted that the only main difference among the three diagrams other than
the left-hand side is the position of block $a$.
Intuitively, the functor $\yoneda{c}$ and the natural transformation $\yoneda{a}$ can be
considered as different ways of representing the object $c$ and the morphism $a$ in $\cC$, respectively.
Also, as already mentioned, the first equality can be thought of as ``bending the wire $x$
at the position of the black circle to change it into the wire $\yoneda{x}$''.

We also have for each $c,d \in \cC$,
\begin{alignat}{1}
 \yoneda{c}(\id_d) = \yoneda{c}(\Endash) \c \id_c &\qquad\diagram\qquad
 \InsertMidPDF{repr_yoneda_hom.pdf},
 \label{eq:repr_yoneda_hom}
\end{alignat}
where $\yoneda{c}(\Endash)$ is the map that maps each $a \in \cC(c,d)$ to
$\yoneda{c}(a) = a \c \Endash$.
The functor $\yoneda{c}$ maps each identity morphism to the identity morphism in $\Set$,
so $\yoneda{c}(\id_d)$ is equal to $\id_{\cC(c,d)}$.
Also, from the definition of $\yoneda{c}(\Endash)$, the map $\yoneda{c}(\Endash) \c \id_c$
maps each $a \in \cC(c,d)$ to $\yoneda{c}(a) \c \id_c = a \c \id_c = a$,
so it is equal to $\id_{\cC(c,d)}$.
Note that this can be understood by considering Eq.~\eqref{eq:repr_yoneda_init_summary0} with $x,c$
replaced by $c,d$.
Therefore, Eq.~\eqref{eq:repr_yoneda_hom} holds.

Let us consider any set-valued functor $X \colon \cC \to \Set$ and natural transformation
$\tau \colon \yoneda{c} \nto X$.
Each component $\tau_d$ of $\tau$ is a map from the hom-set $\cC(c,d)$ to the set $Xd$
and can be represented by
\begin{alignat}{1}
 \footnoteinset{-1.75}{0.3}{\eqref{eq:repr_yoneda_hom}}{%
 \footnoteinset{1.30}{0.3}{\eqref{eq:sliding}}{%
 \InsertPDF{repr_yoneda_hom_tau.pdf}}} \raisebox{1em}{.}
 \label{eq:repr_yoneda_hom_tau}
\end{alignat}
$\tau$ is represented as
\begin{alignat}{1}
 \footnoteinset{-2.41}{0.3}{\eqref{eq:repr_yoneda_hom_tau}}{%
 \InsertPDF{repr_yoneda_hom_tau_nat.pdf}} \raisebox{1em}{.}
 \label{eq:repr_yoneda_hom_tau_nat}
\end{alignat}
Recalling the rules of dotted box notation, one can see that $\tau$ can be represented as
this right-hand side.

\subsubsection{The Yoneda lemma} \label{subsubsec:repr_yoneda_lemma}

\myparagraph{Yoneda map}

Let us define a map $\alpha_{X,c}$, which is called the \termdef{Yoneda map}, as
\pagetarget{term:yoneda_map}{}%
\begin{alignat}{1}
 \alpha_{X,c} \colon \Func{\cC}{\Set}(\yoneda{c},X) \ni \tau \mapsto \tau_c(\id_c) \in Xc
 &\qquad\diagram\qquad
 \InsertMidPDF{repr_yoneda_alpha_map.pdf}, \nonumber \\
 \label{eq:repr_yoneda_alpha_map}
\end{alignat}
where the dotted circle represents the hom-set $\Func{\cC}{\Set}(\yoneda{c},X)$.
From this definition and Eq.~\eqref{eq:repr_yoneda_hom_tau_nat}, we have
\begin{alignat}{2}
 \footnoteinset{-0.39}{0.3}{\eqref{eq:repr_yoneda_hom_tau_nat}}{%
 \InsertMidPDF{repr_yoneda.pdf}}
 &\qquad\xmapsto{\alpha_{X,c}}\qquad
 \InsertMidPDF{repr_yoneda_alpha.pdf}.
 \label{eq:repr_yoneda_alpha}
\end{alignat}
Both sides of the formula to the right-hand side of `the symbol `$\xmapsto{\alpha_{X,c}}$''
represent $\alpha_{X,c}(\tau)$.
Note that in the dotted box notation of $\tau$ (i.e., the right-hand side of the formula
to the left-hand side of the symbol ``$\xmapsto{\alpha_{X,c}}$'' in Eq.~\eqref{eq:repr_yoneda_alpha}),
replacing the wire ``$\Endash$'' with the wire $c$ represents $\tau_c$,
and inserting $\id_c$ into the dotted box
(which yields the rightmost diagram in Eq.~\eqref{eq:repr_yoneda_alpha}) represents $\tau_c(\id_c)$,
i.e., $\alpha_{X,c}(\tau)$.
We also define the map $\beta_{X,c} \colon Xc \to \Func{\cC}{\Set}(\yoneda{c},X)$ as
\begin{alignat}{2}
 \InsertMidPDF{repr_yoneda_x.pdf}
 &\qquad\xmapsto{\beta_{X,c}}\qquad
 \InsertMidPDF{repr_yoneda_beta.pdf},
 \label{eq:repr_yoneda_beta}
\end{alignat}
whose mathematical notation is
\begin{alignat}{1}
 \beta_{X,c}(x) &\coloneqq \{ \cC(c,d) \ni f \mapsto (Xf)(x) \in Xd \}_{d \in \cC}.
\end{alignat}
For each $x \in Xc$, $\beta_{X,c}(x)$ is a natural transformation from $\yoneda{c}$ to $X$
(i.e., an element of $\Func{\cC}{\Set}(\yoneda{c},X)$) since we have for each $f \in \cC(a,b)$%
\footnote{One can also see that $\beta_{X,c}(x)$ is a special case of the natural transformation $\tau$,
represented in Eq.~\eqref{eq:basic_natural_trans_CcG_nat2}, specifically when
$\cD = \cC$, $~\cE = \Set$, $~F = H = X$, $~G = \id_\cC$, $~e = \{*\}$, $~\alpha = \id_X$,
and $f = x$, and thus is a natural transformation from $\cC(c,\Endash) = \yoneda{c}$ to
$\Set(\{ * \},X\Endash) = X$.
The reason why $\Set(\{ * \},X\Endash) = X$ holds is that the action on morphisms of
the functor $\Set(\{ * \},X\Endash)$ maps each $f \in \mor \cD$ to the map $* \mapsto Xf$,
and since $* \mapsto Xf$ is identified with $Xf$,
this is equivalent to the action on morphisms of $X$.},
\begin{alignat}{1}
 \InsertPDF{repr_yoneda_beta_nat.pdf} \raisebox{1em}{.}
 \label{eq:repr_yoneda_beta_nat}
\end{alignat}

Comparing Eqs.~\eqref{eq:repr_yoneda_alpha} and \eqref{eq:repr_yoneda_beta},
it appears that the maps $\alpha_{X,c}$ and $\beta_{X,c}$ are inverses of each other.
In fact, comparing the two blocks $\tau$ and $\alpha_{X,c}(\tau)$ in Eq.~\eqref{eq:repr_yoneda_alpha},
intuitively, the Yoneda map $\alpha_{X,c}$ simply converts the wire $\yoneda{c}$
attached to the bottom of $\tau$ to the wire $c$.
In contrast, $\beta_{X,c}$ converts the wire $c$ attached to the top of $x$
to the wire $\yoneda{c}$.
From another perspective, $\alpha_{X,c}$ removes the dotted box representing $\cC(c,\Endash)$,
and conversely, $\beta_{X,c}$ adds this dotted box.
As can be inferred from this, $\beta_{X,c}$ is the inverse map of $\alpha_{X,c}$,
which will be proven in Theorem~\ref{thm:Yoneda}.
With these maps, intuitively, the lower wire $\yoneda{c}$ and the upper wire $c$ can be freely
exchanged by $\alpha_{X,c}$ and $\beta_{X,c}$.
Note that if we regard the wire $c$ and the wire $\yoneda{c}$ as the same kind of thing,
intuitively, then $\alpha_{X,c}(\tau)$ is something like $\tau$ and
$\beta_{X,c}(x)$ is something like $x$.

\myparagraph{The Yoneda lemma}

\begin{thm}{the Yoneda lemma}{Yoneda}
 The following two properties hold:
 \begin{enumerate}
  \item For each $X \colon \cC \to \Set$ and $c \in \cC$, the Yoneda map $\alpha_{X,c}$
        is invertible.
  \item For each $\sigma \in \Func{\cC}{\Set}(X,Y)$ and $p \in \cC(c,d)$
        (where $X,Y \colon \cC \to \Set$ and $c,d \in \cC$ are arbitrary),
        \begin{alignat}{1}
         (\sigma \b p) \c \alpha_{X,c} = \alpha_{Y,d} \c (\sigma \c \Endash \c \yoneda{p})
         &\qquad\diagram\qquad
         \InsertMidPDF{repr_yoneda_alpha_nat_Xc2.pdf}
         \label{eq:repr_yoneda_alpha_nat_Xc2}
        \end{alignat}
        holds.
 \end{enumerate}
\end{thm}
The areas enclosed by the left and right auxiliary lines in Eq.~\eqref{eq:repr_yoneda_alpha_nat_Xc2}
are $\sigma \b p$ and $\sigma \c \Endash \c \yoneda{p}$, respectively.
In this diagram, the Yoneda map is represented as in the right-hand side of
Eq.~\eqref{eq:repr_yoneda_alpha_map}.
Note that the right-hand side of Eq.~\eqref{eq:repr_yoneda_alpha_nat_Xc2}
represents $\alpha_{Y,d} \c (\sigma \c \Endash \c \yoneda{p})$
by inserting the diagram representing $\sigma \c \Endash \c \yoneda{p}$
(i.e., the area enclosed by the auxiliary line) into the dotted circle of the diagram representing
$\alpha_{Y,d}$ (i.e., the right-hand side of Eq.~\eqref{eq:repr_yoneda_alpha_map}).
\begin{proof}
 (1): It is sufficient to show that the map $\beta_{X,c}$ defined in Eq.~\eqref{eq:repr_yoneda_beta}
 is the inverse of $\alpha_{X,c}$.
 We have for each $\tau \in \Func{\cC}{\Set}(\yoneda{c},X)$,
 \begin{alignat}{1}
  \InsertMidPDF{repr_yoneda_beta_alpha0.pdf}
  \quad\stackrel{\eqref{eq:repr_yoneda_alpha}}{\xmapsto{\alpha_{X,c}}}\quad
  \InsertMidPDF{repr_yoneda_beta_alpha1.pdf}
  \quad\stackrel{\eqref{eq:repr_yoneda_beta}}{\xmapsto{\beta_{X,c}}}\quad
  \footnoteinset{0.47}{0.3}{\eqref{eq:repr_yoneda_hom_tau_nat}}{%
  \InsertMidPDF{repr_yoneda_beta_alpha2.pdf}},
  \label{eq:repr_yoneda_beta_alpha}
 \end{alignat}
 and for each $x \in Xc$,
 \begin{alignat}{1}
  \InsertMidPDF{repr_yoneda_x.pdf}
  \quad\stackrel{\eqref{eq:repr_yoneda_beta}}{\xmapsto{\beta_{X,c}}}\quad
  \InsertMidPDF{repr_yoneda_alpha_beta1.pdf}
  \quad\stackrel{\eqref{eq:repr_yoneda_alpha}}{\xmapsto{\alpha_{X,c}}}\quad
  \InsertMidPDF{repr_yoneda_alpha_beta2.pdf},
  \label{eq:repr_yoneda_alpha_beta}
 \end{alignat}
 where we use the fact that the map ``$\xmapsto{\alpha_{X,c}}$'' in
 Eq.~\eqref{eq:repr_yoneda_alpha_beta} is equivalent to
 the operation of replacing the wire ``$\Endash$'' with the wire $c$ and substituting $\id_c$
 into the dotted box.
 Therefore, $\beta_{X,c}$ is the inverse of $\alpha_{X,c}$.

 (2): We have
 \begin{alignat}{1}
  \footnoteinset{-1.26}{0.3}{\eqref{eq:sliding}}{%
  \footnoteinset{1.65}{0.3}{\eqref{eq:repr_yoneda_init_summary}}{%
  \InsertPDF{repr_yoneda_alpha_nat_Xc2_proof.pdf}}} \raisebox{1em}{.}
  \label{eq:repr_yoneda_alpha_nat_Xc2_proof}
 \end{alignat}
\end{proof}

In what follows, we denote $\beta_{X,c}$ as $\alpha_{X,c}^{-1}$.
For any $x \in Xc$ (where $c$ is also arbitrary),
\begin{alignat}{1}
 x = \alpha_{X,c}(\alpha^{-1}_{X,c}(x))
 &\qquad\diagram\qquad
 \InsertMidPDF{repr_yoneda_beta2.pdf}
 \label{eq:repr_yoneda_beta2}
\end{alignat}
obviously holds.
Note that $a = \alpha_{\yoneda{x},c}(\yoneda{a})$ holds from
the last equality in Eq.~\eqref{eq:repr_yoneda_init_summary},
which yields
\begin{alignat}{1}
 \yoneda{a} = \alpha^{-1}_{\yoneda{x},c}(a).
 \label{eq:repr_yoneda_a_alpha_inv}
\end{alignat}

Equation~\eqref{eq:repr_yoneda_beta2} can be seen as a generalization of the last equality
in Eq.~\eqref{eq:repr_yoneda_init_summary}.

The Yoneda map $\alpha_{X,c}$ and its inverse $\alpha_{X,c}^{-1}$ can be expressed as
\begin{alignat}{1}
 \InsertMidPDF{repr_yoneda_cong_alpha3.pdf} &\quad\xmapsto{\alpha_{X,c}}\quad
 \InsertMidPDF{repr_yoneda_cong_alpha2.pdf}, \qquad
 \InsertMidPDF{repr_yoneda_cong_beta1.pdf} \quad\xmapsto{\alpha_{X,c}^{-1}}\quad
 \InsertMidPDF{repr_yoneda_cong_beta3.pdf},
 \label{eq:repr_yoneda_cong2}
\end{alignat}
where the blue block in the shape of
``$\raisebox{-.1em}{\InsertPDFtext{report_text_repr_yoneda_cong2.pdf}}$''
represents $\{ \alpha_{X,c}^{-1} \}_{X \in \Func{\cC}{\Set}}$%
\footnote{This blue block is the natural transformation obtained by substituting
$c = \{ * \}$, $~G = \ev_c$, $~e = \yoneda{c}$, and $F = \id_{\Func{\cC}{\Set}}$
into the right-hand side of Eq.~\eqref{eq:basic_natural_trans_CcG_nat2}.
As can be seen from Corollary~\ref{cor:YonedaNat}, this is a natural isomorphism from the functor
$\Set(\{ * \}, \Endash c) = \ev_c$ to the functor $\Func{\cC}{\Set}(\yoneda{c},\Endash)$.}.
The diagram to the right-hand side of the symbol ``$\xmapsto{\alpha_{X,c}^{-1}}$'' represents
the natural transformation
$\alpha_{X,c}^{-1}(x) \colon \yoneda{c} \nto X$.
Intuitively, the Yoneda map $\alpha_{X,c}$ simply converts the wire $\yoneda{c}$
into the wire $c$, and $\alpha_{X,c}^{-1}$ converts the wire $c$ into the wire $\yoneda{c}$.

\myparagraph{Naturality of the Yoneda lemma}

If we can express Eq.~\eqref{eq:repr_yoneda_alpha_nat_Xc2} as
\begin{alignat}{1}
 G(\sigma,p) \c \alpha_{X,c} &= \alpha_{Y,d} \c F(\sigma,p)
 \label{eq:repr_yoneda_nat_FG}
\end{alignat}
using two appropriate bifunctors $F,G \colon \Func{\cC}{\Set} \times \cC \to \Set$,
then this equation can be considered to represent naturality of
$\alpha \coloneqq \{ \alpha_{X,c} \}_{X \in \Func{\cC}{\Set}, c \in \cC}$,
and thus $\alpha$ is a natural transformation from $F$ to $G$.
In order to express it in such a way, it is sufficient to satisfy
\begin{alignat}{1}
 G(\sigma,p) &= \sigma \b p, \quad F(\sigma,p) = \sigma \c \Endash \c \yoneda{p}
 \qquad (\forall \braket{\sigma,p} \in \mor (\Func{\cC}{\Set} \times \cC))
 \label{eq:repr_yoneda_FG}
\end{alignat}
and this condition determines the action on morphisms of $G$ and $F$.
Since the action on morphisms of $G$ is $\braket{\sigma,p} \mapsto \sigma \b p$,
$G$ is the evaluation functor $\ev \colon \Func{\cC}{\Set} \times \cC \to \Set$.
From Example~\ref{ex:FuncEval2}, we have
\begin{alignat}{1}
 \ev(\sigma,p) = \sigma \b p &\qquad\diagram\qquad
 \InsertMidPDF{repr_yoneda_alpha_nat_ev.pdf}.
 \label{eq:repr_yoneda_alpha_nat_ev}
\end{alignat}
For $F$, let us consider the following bifunctor:
\begin{alignat}{1}
 \lefteqn{ \Func{\cC}{\Set}(\yoneda{\Enndash},\Endash)
 \coloneqq \Func{\cC}{\Set}(\Enndash,\Endash) \b
 \braket{\id_{\Func{\cC}{\Set}},\yoneda{\Endash}} } \nonumber \\
 &\diagram\qquad
 \InsertMidPDF{repr_yoneda_nat_bifunc.pdf},
 \label{eq:repr_yoneda_nat_bifunc}
\end{alignat}
where $\Func{\cC}{\Set}(\Enndash,\Endash)$ represents the hom-functor
$(\Func{\cC}{\Set})^\op(\Endash,\Enndash)$.
Also, $\braket{\id_{\Func{\cC}{\Set}},\yoneda{\Endash}}$ is a pair of the identity functor
$\id_{\Func{\cC}{\Set}}$ and the functor $\yoneda{\Endash} \colon \cC \to (\Func{\cC}{\Set})^\op$
(refer to the beginning of Example~\ref{ex:FuncHomFG}),
and the functor $\yoneda{\Endash}$ maps each object $c$ and each morphism $p$ in $\cC$
to $\yoneda{c}$ and $\yoneda{p}$, respectively.
Since the action on morphisms of $\Func{\cC}{\Set}(\yoneda{\Enndash},\Endash)$ is
\begin{alignat}{4}
 \mor (\Func{\cC}{\Set} \times \cC) &\ni \braket{\sigma,p}
 &&\quad\xmapsto{\braket{\id_{\Func{\cC}{\Set}},\yoneda{\Endash}}}\quad
 \braket{\sigma,\yoneda{p}}
 &&\quad\xmapsto{\Func{\cC}{\Set}(\Enndash,\Endash)}\quad
 &\sigma \c \Endash \c \yoneda{p} &\in \mor \Set,
 \label{eq:repr_yoneda_nat_bifunc_action}
\end{alignat}
$F = \Func{\cC}{\Set}(\yoneda{\Enndash},\Endash)$ holds from Eq.~\eqref{eq:repr_yoneda_FG}.
We have
\begin{alignat}{1}
 \lefteqn{ (\Func{\cC}{\Set}(\yoneda{\Enndash},\Endash))(\sigma,p)
 = \sigma \c \Endash \c \yoneda{p} }
 \nonumber \\
 &\diagram\qquad
 \InsertMidPDF{repr_yoneda_alpha_nat_Set.pdf}.
 \label{eq:repr_yoneda_alpha_nat_Set}
\end{alignat}

From the above discussion, we obtain the following corollary:
\begin{cor}{}{YonedaNat}
 Property~(2) of Theorem~\ref{thm:Yoneda} is equivalent to
 $\alpha \coloneqq \{ \alpha_{X,c} \}_{X \in \Func{\cC}{\Set}, c \in \cC}$
 being a natural transformation from the bifunctor
 $\Func{\cC}{\Set}(\yoneda{\Enndash},\Endash) \colon \Func{\cC}{\Set} \times \cC \to \Set$
 to the evaluation functor $\ev \colon \Func{\cC}{\Set} \times \cC \to \Set$.
 Or equivalently, for each $\sigma \in \Func{\cC}{\Set}(X,Y)$ and
 $p \in \cC(c,d)$ (where $X,Y \in \Func{\cC}{\Set}$ and $c,d \in \cC$ are arbitrary),
 \begin{alignat}{1}
  \lefteqn{ \ev(\sigma,p) \c \alpha_{X,c} = \alpha_{Y,d} \c
  (\Func{\cC}{\Set}(\yoneda{\Enndash},\Endash))(\sigma,p) } \nonumber \\
  &\diagram\qquad
  \InsertMidPDF{repr_yoneda_alpha_nat_Xc.pdf}
  \label{eq:repr_yoneda_alpha_nat_Xc}
 \end{alignat}
 holds.
\end{cor}
\begin{proof}
 The left- and right-hand side of Eq.~\eqref{eq:repr_yoneda_alpha_nat_Xc} can be expressed as
 \begin{alignat}{1}
  \footnoteinsets{0.91}{0.3}{\eqref{eq:repr_yoneda_alpha_nat_ev}}{\eqref{eq:repr_yoneda_alpha_map}}{%
  \InsertPDF{repr_yoneda_alpha_nat_Xc_left.pdf}}
  \label{eq:repr_yoneda_alpha_nat_Xc_left}
 \end{alignat}
 and
 \begin{alignat}{1}
  \footnoteinsets{1.11}{0.3}{\eqref{eq:repr_yoneda_alpha_map}}{\eqref{eq:repr_yoneda_alpha_nat_Set}}{%
  \InsertPDF{repr_yoneda_alpha_nat_Xc_right.pdf}} \raisebox{1em}{,}
  \label{eq:repr_yoneda_alpha_nat_Xc_right}
 \end{alignat}
 respectively.
 Therefore, Eq.~\eqref{eq:repr_yoneda_alpha_nat_Xc} and Eq.~\eqref{eq:repr_yoneda_alpha_nat_Xc2}
 are equivalent.
\end{proof}

Recall that the Yoneda map $\alpha_{X,c}$ can be interpreted as a
``transformation from the wire $\yoneda{c}$ to the wire $c$''.
Therefore, Eq.~\eqref{eq:repr_yoneda_alpha_nat_Xc2} can be interpreted as
``applying $\sigma$ and $p$ after transforming the wire $\yoneda{c}$ to the wire $c$''
and ``transforming the wire $\yoneda{c}$ to the wire $c$ after applying $\sigma$ and $\yoneda{p}$''
are equivalent.
In Eq.~\eqref{eq:repr_yoneda_nat_FG}, the map that applies $\sigma$ and $p$ can be considered
to be represented by $G(\sigma,p)$, and the map that applies $\sigma$ and $\yoneda{p}$ can be
considered to be represented by $F(\sigma,p)$.

Each component $\alpha_{X,c}$ of the natural transformation $\alpha$ is a morphism
from $\Func{\cC}{\Set}(\yoneda{c},X)$ to $Xc$,
which follows from $(\Func{\cC}{\Set}(\yoneda{\Enndash},\Endash))(X,c)
= \Func{\cC}{\Set}(\yoneda{c},X)$ and $\ev(X,c) = Xc$.
Also, we can see that Theorem~\ref{thm:Yoneda} is equivalent to $\alpha$ being
a natural isomorphism from $\Func{\cC}{\Set}(\yoneda{\Enndash},\Endash)$ to $\ev$,
or equivalently, the isomorphism $\alpha_{X,c} \colon \Func{\cC}{\Set}(\yoneda{c},X) \cong Xc$
being natural in $X$ and $c$.
This natural isomorphism can be depicted by
\begin{alignat}{1}
 \Func{\cC}{\Set}(\yoneda{\Enndash},\Endash) \cong \ev
 &\qquad\diagram\qquad
 \InsertMidPDF{repr_yoneda_cong.pdf}.
 \label{eq:repr_yoneda_cong}
\end{alignat}
Also, from Eq.~\eqref{eq:repr_yoneda_alpha_map}, one can depict $\alpha$ by%
\footnote{One might feel that this right-hand side appears to violate Rule~5 of the dotted box notation.
If this is a concern, then a slight modification to this rule could be considered.}
\begin{alignat}{1}
 \alpha &\qquad\diagram\qquad
 \InsertMidPDF{repr_yoneda_alpha_map2.pdf}.
 \label{eq:repr_yoneda_alpha_map2}
\end{alignat}

\myparagraph{Dual to the Yoneda lemma}

We briefly state the diagram of $\yonedaop{x}$.
Considering the ``upside-down'' version of Eq.~\eqref{eq:repr_yoneda_init_summary},
we can see that an element $a$ of $\yonedaop{x}(c) = \cC^\op(x,c) = \cC(c,x)$ is represented by
\begin{alignat}{1}
 \lefteqn{ a = \yonedaop{x}(a) \c \id_x = a = \yonedaop{a}{}_c \c \id_c } \nonumber \\
 &\diagram\qquad
 \InsertMidPDF{repr_yoneda_init_summary_op.pdf},
 \label{eq:repr_yoneda_init_summary_op}
\end{alignat}
where the left and right black circles in the diagram represent $\id_x$ and $\id_c$, respectively.
Recall that we represent the objects and morphisms in $\cC^\op$ as the ``upside-down'' of
the objects and morphisms in $\cC$ (see Eq.~\eqref{eq:basic_contra_circ}).
For the natural transformation $\yonedaop{a} \colon \yonedaop{c} \nto \yonedaop{x}$,
the reader refers to Eq.~\eqref{eq:basic_nat_postcomp}.
A natural transformation from $\yonedaop{c}$ to a presheaf $X' \colon \cC^\op \to \Set$
can also be represented in the same way as Eq.~\eqref{eq:repr_yoneda_hom_tau_nat}.

For any category $\cC$, denote the functor category $\Func{\cC^\op}{\Set}$ as
the \termdef{category of presheaves} and write it as $\hat{\cC}$.
Considering the duals of Theorem~\ref{thm:Yoneda} and Corollary~\ref{cor:YonedaNat}
yields the following corollary.
\begin{cor}{the Yoneda lemma}{Yonedaop}
 For any object $c \in \cC^\op$ and presheaf $X \in \hat{\cC}$,
 define a map $\alpha_{X,c} \colon \hat{\cC}(\yonedaop{c},X) \to Xc$,
 which is also called the \termdef{Yoneda map}, as
 \begin{alignat}{1}
  \alpha_{X,c}(\tau) \coloneqq \tau_c(\id_c)
  &\qquad\diagram\qquad
  \InsertMidPDF{repr_yonedaop_alpha.pdf}.
  \label{eq:repr_yonedaop_alpha}
 \end{alignat}
 Then, the following two properties hold:
 \begin{enumerate}
  \item For each $X \in \hat{\cC}$ and $c \in \cC^\op$, $\alpha_{X,c}$ is invertible.
  \item For each $\sigma \in \hat{\cC}(X,Y)$ and $p \in \cC^\op(c,d)$
        (where $X,Y \in \hat{\cC}$ and $c,d \in \cC^\op$ are arbitrary),
        \begin{alignat}{1}
         (\sigma \b p) \c \alpha_{X,c} &= \alpha_{Y,d} \c (\sigma \c \Endash \c \yonedaop{p})
        \end{alignat}
        holds.
        In other words, $\{ \alpha_{X,c} \}_{X \in \hat{\cC}, c \in \cC^\op}$ is
        a natural transformation from the bifunctor
        $\hat{\cC}(\yonedaop{\Enndash},\Endash) \colon \hat{\cC} \times \cC^\op \to \Set$
        to the evaluation functor $\ev \colon \hat{\cC} \times \cC^\op \to \Set$.
 \end{enumerate}
\end{cor}
Similar to Eq.~\eqref{eq:repr_yoneda_cong}, this natural isomorphism can be depicted by
\begin{alignat}{1}
 \hat{\cC}(\yonedaop{\Enndash},\Endash) \cong \ev
 &\qquad\diagram\qquad
 \InsertMidPDF{repr_yoneda_cong_op.pdf}.
 \label{eq:repr_yoneda_cong_op}
\end{alignat}

\subsubsection{Yoneda embeddings} \label{subsubsec:repr_yoneda_embedding}

Let us define a functor $\yonedaop{\cC}$ from a category $\cC$ to
the category of presheaves $\hat{\cC}$ as follows:
\begin{itemize}
 \item It maps each object $c$ in $\cC$ to the functor $\yonedaop{c} = \cC(\Endash,c)$.
 \item It maps each morphism $f \colon c \to d$ in $\cC$ to the natural transformation
       $\yonedaop{f} \colon \yonedaop{c} \nto \yonedaop{d}$.
\end{itemize}
This functor is called the \termdef{Yoneda embedding}.
We have for each morphism $f \in \cC(c,d)$,
\begin{alignat}{1}
 \footnoteinset{1.13}{0.3}{\eqref{eq:repr_yoneda_a_alpha_inv}}{%
 \InsertPDF{repr_yoneda_embedding_f.pdf}} \raisebox{1em}{.}
 \label{eq:repr_yoneda_embedding_f}
\end{alignat}

\begin{cor}{}{YonedaEmbedding}
 For any category $\cC$, the Yoneda embedding $\yonedaop{\cC}$ is fully faithful.
\end{cor}
\begin{proof}
 For each $c,d \in \cC$, the Yoneda embedding provides the map
 \begin{alignat}{1}
  \yonedaop{\cC} &\colon \cC(c,d) \ni f \mapsto
  \yonedaop{f} = \alpha_{\yonedaop{d},c}^{-1}(f) \in \hat{\cC}(\yonedaop{c},\yonedaop{d}).
 \end{alignat}
 This map is invertible and its inverse is $\alpha_{\yonedaop{d},c}$.
 Therefore, the Yoneda embedding is fully faithful.
\end{proof}

From Corollary~\ref{cor:YonedaEmbedding}, we can derive the following lemma.
\begin{lemma}{}{YonedaopCongFunc}
 For two functors $F,G \colon \cC \to \cD$, we have
 \begin{alignat}{1}
  F \cong G &\quad\Leftrightarrow\quad
  \text{there exists an isomorphism $\cD(d,Fc) \cong \cD(d,Gc)$ that is natural in $c$ and $d$}
  \nonumber \\
  &\quad\Leftrightarrow\quad
  \text{there exists an isomorphism $\cD(Fc,d) \cong \cD(Gc,d)$ that is natural in $c$ and $d$}.
  \nonumber \\
  \label{eq:YonedaopCongFunc}
 \end{alignat}
\end{lemma}
\begin{proof}
 It is sufficient to show the first ``$\Leftrightarrow$'' in Eq.~\eqref{eq:YonedaopCongFunc}
 since its dual proves the second ``$\Leftrightarrow$''.
 Let $R$ denote the standard functor from $\Func{\cC}{\hat{\cD}} = \Functwo{\cC}{\cD^\op}{\Set}$
 to $\Func{\cD^\op \times \cC}{\Set}$.
 From Corollary~\ref{cor:YonedaEmbedding}, the Yoneda embedding $\yonedaop{\cD}$ is fully faithful,
 and thus, from Lemma~\ref{lemma:FullFaithfulPostcomp},
 $\yonedaop{\cD} \b \Endash \colon \Func{\cC}{\cD} \to \Func{\cC}{\hat{\cD}}$ is also
 fully faithful.
 For two functors $F,G \colon \cC \to \cD$, we have
 \begin{alignat}{2}
  F \cong G \quad\Leftrightarrow\quad \yonedaop{\cD} \b F \cong \yonedaop{\cD} \b G
  &\quad\Leftrightarrow\quad &R (\yonedaop{\cD} \b F) &\cong R (\yonedaop{\cD} \b G) \nonumber \\
  &\quad\Leftrightarrow\quad &\cD(\Endash,F\Enndash) &\cong \cD(\Endash,G\Enndash).
  \label{eq:repr_yonedaop_cong_func_proof}
 \end{alignat}
 The first and second ``$\Leftrightarrow$'' follow from the fact that $\yonedaop{\cD} \b \Endash$ and $R$
 are fully faithful and the fact that for any functor $H \colon \cE \to \cF$ and $e,e' \in \cE$
 (where $\cE$ and $\cF$ are arbitrary), $e \cong e'$ is equivalent to $He \cong He'$.
 The last ``$\Leftrightarrow$'' follows from
 $R (\yonedaop{\cD} \b F) = \cD(\Endash,F\Enndash)$ and
 $R (\yonedaop{\cD} \b G) = \cD(\Endash,G\Enndash)$%
 \footnote{Proof: The functor $R (\yonedaop{\cD} \b F) \in \Func{\cD^\op \times \cC}{\Set}$
 maps each morphism $\braket{f,g} \in \mor (\cD^\op \times \cC)$ to
 $((\yonedaop{\cD} \b F)g) \b f \in \mor \Set$.
 Also, $((\yonedaop{\cD} \b F)g) \b f = \yonedaop{(Fg)} \b f = Fg \c \Endash \c f$ holds.
 Therefore, the action on morphisms of $R (\yonedaop{\cD} \b F)$ is
 $\mor (\cD^\op \times \cC) \ni \braket{f,g} \mapsto Fg \c \Endash \c f \in \mor \Set$,
 which is equal to the action on morphisms of the functor $\cD(\Endash,F\Enndash)$.
 Thus, we have $R (\yonedaop{\cD} \b F) = \cD(\Endash,F\Enndash)$.
 Similarly, we have $R (\yonedaop{\cD} \b G) = \cD(\Endash,G\Enndash)$.}.
 Therefore, the first ``$\Leftrightarrow$'' in Eq.~\eqref{eq:YonedaopCongFunc} holds.
\end{proof}

Since the Yoneda embedding $\yonedaop{\cD} \b \Endash \colon \Func{\cC}{\cD}
\to \Func{\cC}{\hat{\cD}}$ is fully faithful, for any natural transformation
$\theta \colon \cD(\Endash,F\Enndash) \nto \cD(\Endash,G\Enndash)$, we can find
a unique $\sigma \colon F \nto G$ satisfying
\begin{alignat}{1}
 \footnoteinset{0.00}{0.3}{\eqref{eq:basic_full_faithful}}{%
 \InsertPDF{repr_yoneda_func_iso_tau.pdf}} \raisebox{1em}{.}
 \label{eq:repr_yoneda_func_iso_tau}
\end{alignat}
This may be easier to understand if expressed as
\begin{alignat}{1}
 \InsertPDF{repr_yoneda_dFc_dGc.pdf} \raisebox{1em}{,}
 \label{eq:repr_yoneda_dFc_dGc}
\end{alignat}
where we represent $\theta$ as the block in the shape of
``$\raisebox{-.1em}{\InsertPDFtext{report_text_repr_dFc_dGc.pdf}}$''.
The left-hand side can be interpreted as the map that sends a morphism $g \in \cD(d,Fc)$
(where $d \in \cD$ and $c \in \cC$ are arbitrary) to the morphism $\theta_{d,c}(g) \in \cD(d,Gc)$
(by substituting $d$ and $c$ for $\Endash$ and $\Enndash$ respectively).
The right-hand side can be interpreted as the map that sends a morphism $g$ to
$\sigma_c g \in \cD(d,Gc)$.

\subsection{Universal properties} \label{subsec:repr_cG}

\subsubsection{Comma category $c \comma G$} \label{subsubsec:repr_cG_cG}

Let us arbitrarily choose an object $c \in \cC$ and a functor $G \colon \cD \to \cC$.
The category defined as follows is called a \termdef{comma category} and is denoted by $c \comma G$:
\begin{itemize}
 \item Its each object is a pair, $\braket{x,a}$, of an object $x \in \cD$
       and a morphism $a \in \cC(c,Gx)$.
 \item Its each morphism from an object $\braket{x,a}$ to an object $\braket{y,a'}$ is
       a morphism $f \in \cD(x,y)$ in $\cD$ satisfying%
       \footnote{Since the domain and the codomain of each morphism in $c \comma G$ must be
       objects in $c \comma G$, strictly speaking morphisms in $c \comma G$
       should be distinguished from morphisms in $\cD$.
       However, since it will always be clear from context which meaning is intended,
       we often represent morphisms in $c \comma G$ as morphisms in $\cD$.
       The same applies to comma categories $G \comma c$ and categories of elements mentioned later.}
       \begin{alignat}{1}
        a' = Gf \c a &\qquad\diagram\qquad
        \InsertMidPDF{repr_cG.pdf}.
        \label{eq:repr_cG}
       \end{alignat}
 \item The composite of its morphisms is the composite of morphisms in $\cD$,
       and its identity morphism is the identity morphism in $\cD$.
\end{itemize}

As mentioned earlier, we identify elements $x$ of a set $X$ with maps
$\{ * \} \ni * \mapsto x \in X$.
This identification leads to the following lemma.
\begin{lemma}{}{CommacGX}
 For any categories $\cC$ and $\cD$, object $c \in \cC$, and functor $G \colon \cD \to \cC$,
 the comma category $c \comma G$ is equal to the comma category $\{ * \} \comma \cC(c,G\Endash)$.
\end{lemma}
\begin{proof}
 An object in $\{ * \} \comma \cC(c,G\Endash)$ is a pair, $\braket{x,* \mapsto a}$, of
 an object $x \in \cD$ and a map $\{ * \} \ni * \mapsto a \in \cC(c,Gx)$.
 Since we identify $* \mapsto a$ with $a$, this is equal to the object $\braket{x,a}$ in $c \comma G$.
 This can also be seen from the following equation (note that $\cC(c,G\Endash) = \yoneda{c} \b G$):
 \begin{alignat}{1}
  (* \mapsto a) = a &\qquad\diagram\qquad
  \footnoteinset{0.19}{0.3}{\eqref{eq:repr_yoneda_init_summary}}{%
  \InsertMidPDF{repr_cG_X.pdf}}.
  \label{eq:repr_cG_X}
 \end{alignat}
 Also, a morphism $f \colon \braket{x,a} \to \braket{y,a'}$ in $\{ * \} \comma \cC(c,G\Endash)$
 is a morphism $f \in \cD(x,y)$ satisfying $a' = Gf \c a$, and thus is a morphism in $c \comma G$.
 Conversely, objects and morphisms in $c \comma G$ are contained in $\{ * \} \comma \cC(c,G\Endash)$.
 Composites of morphisms and identity morphisms are also the same.
 Therefore, $c \comma G = \{ * \} \comma \cC(c,G\Endash)$ holds.
\end{proof}

\subsubsection{Universal morphisms} \label{subsubsec:repr_cG_univ}

\begin{define}{universal morphisms}{Univ}
 An initial object%
 \footnote{An \termdef{initial object} in $\cC$ is an object $a$ in $\cC$ such that
 for any $x \in \cC$, there exists a unique morphism with domain $a$ and codomain $x$.}
 $\braket{u,\eta}$ in a comma category $c \comma G$ is
 called a \termdef{universal morphism} from $c$ to $G$.
 In other words, $\braket{u,\eta}$ with $u \in \cD$ and $\eta \in \cC(c,Gu)$ is called
 a universal morphism from $c$ to $G$
 if, for any morphism $a \in \cC(c,Gx)$ with any $x \in \cD$,
 there exists a unique morphism $\ol{a} \in \cD(u,x)$ satisfying
 \begin{alignat}{1}
  a = G \ol{a} \c \eta &\qquad\diagram\qquad
  \InsertMidPDF{repr_univ.pdf},
  \tag{univ}
  \label{eq:univ}
 \end{alignat}%
 where the black circle represents $\eta$.
\end{define}

Sometimes $\eta$ is simply called a universal morphism.
In this paper, when we mention the universal property of $\braket{u,\eta}$,
we are often referring to Eq.~\eqref{eq:univ}.
It follows from Eq.~\eqref{eq:univ} that $a$ and $\ol{a}$ correspond one-to-one.
In particular, $a$ is $\eta$ if and only if the corresponding morphism $\ol{a}$ is $\id_u$.

\begin{ex}{}{}
 Let us restate the second equality of Eq.~\eqref{eq:repr_yoneda_init_summary}:
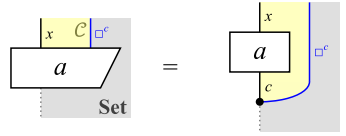
 \begin{alignat}{1}
  \InsertPDF{repr_yoneda_init2.pdf} \raisebox{1em}{.}
  \label{eq:repr_yoneda_init2}
 \end{alignat}
 Comparing it with Eq.~\eqref{eq:univ}, it can be seen that $\braket{x,a}$ is an object
 in the comma category $\{ * \} \comma \yoneda{c}$.
 In this example, $\ol{a}$ in Eq.~\eqref{eq:univ} is $a$ itself.
 The black circle represents the object $\braket{c,\id_c}$ in $\{ * \} \comma \yoneda{c}$,
 which can be understood as an initial object,
 i.e., a universal morphism from $\{ * \}$ to $\yoneda{c}$.
\end{ex}
\begin{ex}{the Yoneda lemma}{ReprYonedaUniv}
 For any object $c \in \cC$ and set-valued functor $X \colon \cC \to \Set$,
 there is a one-to-one correspondence between $x \in Xc$ and
 $\ol{x} \coloneqq \alpha_{X,c}^{-1}(x) \in \Func{\cC}{\Set}(\yoneda{c},X)$
 (recall Eq.~\eqref{eq:repr_yoneda_cong} or Eq.~\eqref{eq:repr_yoneda_cong2}).
 This correspondence is expressed by
 \begin{alignat}{1}
  \InsertMidPDF{repr_yoneda_init3.pdf}
  &\qquad\myLeftrightarrow{equivalent}\qquad
  \InsertMidPDF{repr_yoneda_init4.pdf},
  \nonumber \\
  \label{eq:repr_yoneda_init3}
 \end{alignat}
 where the black circle represents the identity morphism $\id_c$.
 Equation~\eqref{eq:repr_yoneda_init3} expresses the same equation in two ways using
 $\ev_c = \Endash \b c$ (recall Example~\ref{ex:FuncEval}).
 The right-hand side of the symbol ``$\Leftrightarrow$'' corresponds to Eq.~\eqref{eq:repr_yoneda_alpha}.
 It follows from Eq.~\eqref{eq:repr_yoneda_init3} that the black circle represents
 the object $\braket{\yoneda{c},\id_c}$ in the comma category $\{ * \} \comma \ev_c$,
 which is an initial object, i.e., a universal morphism from $\{ * \}$ to $\ev_c$.
\end{ex}

\begin{ex}{}{ReprUnivFF}
 Equation~\eqref{eq:basic_full_faithful} means that for each $a \in \cC$, $\braket{a,\id_{Fa}}$
 is a universal morphism from $Fa$ to $F$
 (please substitute $\braket{u,\eta}$ in Eq.~\eqref{eq:univ} with $\braket{a,\id_{Fa}}$).
 Thus, it follows that a functor $F \colon \cC \to \cD$ is fully faithful if and only if
 $\braket{a,\id_{Fa}}$ is a universal morphism from $Fa$ to $F$ for each $a \in \cC$.
\end{ex}

A universal morphism from $c$ to $G$ is an initial object in a comma category $c \comma G$,
and thus is essentially unique if it exists%
\footnote{''Essentially'' means ``up to isomorphism''.
We here used the fact that an initial object in any category is essentially unique.}.
More specifically, when $\braket{u,\eta}$ is a universal morphism from $c$ to $G$,
for any universal morphism $\braket{u',\eta'}$ from $c$ to $G$,
there exists a unique isomorphism $\psi \in \cD(u,u')$ such that
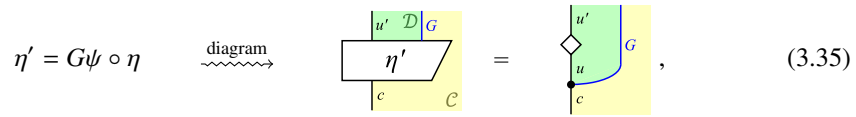
\begin{alignat}{1}
 \eta' = G\psi \c \eta
 &\qquad\diagram\qquad
 \InsertMidPDF{repr_cG_init_cong.pdf},
 \label{eq:repr_cG_init_cong}
\end{alignat}
where the black circle represents $\eta$, and the diamond block represents $\psi$
(note that in this paper, an isomorphism is often represented by a diamond block).
Therefore, objects $u$ and $u'$ in $\cD$ are isomorphic.
In particular, if $\eta$ is an isomorphism, then so is $\eta'$.
Conversely, for a universal morphism $\braket{u,\eta}$, $\braket{u',\eta'}$ represented
in the form of Eq.~\eqref{eq:repr_cG_init_cong} is also a universal morphism from $c$ to $G$.

For a set-valued functor $X \colon \cC \to \Set$, the \termdef{category of elements} $\el(X)$
can be defined as the comma category $\{ * \} \comma X$.
From Lemma~\ref{lemma:CommacGX}, we have $\el(\cC(c,G\Endash)) = c \comma G$.

\subsubsection{Representable functors} \label{subsubsec:repr_repr_repr}

For a set-valued functor $X \colon \cC \to \Set$,
if there exists an object $u \in \cC$ and a natural isomorphism $\sigma \colon \yoneda{u} \ntocong X$,
then we call $X$ \termdef{representable} and the pair $\braket{u,\sigma}$
a \termdef{representation} of $X$.
$\yoneda{u}$ itself is obviously representable.
Each component $\sigma_c$ of $\sigma$ (where $c \in \cC$) is an invertible map
from $\yoneda{u}(c) = \cC(u,c)$ to $Xc$.

\begin{thm}{}{ReprCG}
 For any categories $\cC$ and $\cD$, object $c \in \cC$, and functor $G \colon \cD \to \cC$,
 we have
 \begin{alignat}{1}
  \text{a universal morphism from $c$ to $G$ exists}
  &\quad\Leftrightarrow\quad \text{$\cC(c,G\Endash)$ is representable},
  \label{eq:repr_CG_repr}
 \end{alignat}
 or more specifically, we have that for any $u \in \cD$ and
 $\sigma \colon \yoneda{u} \nto \cC(c,G\Endash)$,
 \begin{alignat}{1}
  \lefteqn{ \text{$\braket{u,\alpha_{X,u}(\sigma)}$ is
  a universal morphism from $c$ to $G$} }
  \nonumber \\
  &\quad\Leftrightarrow\quad
  \text{$\braket{u,\sigma}$ is a representation for $\cC(c,G\Endash)$},
  \label{eq:repr_univ_usigma}
 \end{alignat}
 where $X \coloneqq \cC(c,G\Endash)$.
\end{thm}
\begin{proof}
 From Lemma~\ref{lemma:CommacGX}, we have $c \comma G = \{ * \} \comma X$,
 and being a universal morphism from $c$ to $G$ is equivalent to being an initial object in
 $c \comma G = \{ * \} \comma X$.
 Thus, Eq.~\eqref{eq:repr_univ_usigma} can be rewritten as
 \begin{alignat}{1}
  \text{$\braket{u,\alpha_{X,u}(\sigma)}$ is an initial object in $\{ * \} \comma X$}
  &\quad\Leftrightarrow\quad
  \text{$\braket{u,\sigma}$ is a representation for $X$}.
  \label{eq:repr_univ_usigma_another}
 \end{alignat}
 Also, if Eq.~\eqref{eq:repr_univ_usigma_another} holds, then the ``$\Leftarrow$'' direction in
 Eq.~\eqref{eq:repr_CG_repr} obviously holds.
 The ``$\Rightarrow$'' direction in Eq.~\eqref{eq:repr_CG_repr} holds since the Yoneda map $\alpha_{X,u}$
 is invertible by the Yoneda lemma (see Theorem~\ref{thm:Yoneda})%
 \footnote{Indeed, if $\{ * \} \comma X$ has an initial object $\braket{u,\eta}$
 (where $u \in \cD$ and $\eta \in Xu$), then
 since $\braket{u,\eta}$ is equal to $\braket{u,\alpha_{X,u}(\sigma)}$ with
 $\sigma \coloneqq \alpha_{X,u}^{-1}(\eta)$,
 $\braket{u,\sigma}$ is a representation for $X$ by Eq.~\eqref{eq:repr_univ_usigma_another}.}.
 In what follows, we show Eq.~\eqref{eq:repr_univ_usigma_another}.

 Arbitrarily choose an object $u \in \cD$ and a natural transformation
 $\sigma \colon \yoneda{u} \nto X$.
 Then, since $\alpha_{X,u}(\sigma)$ is in $Xu$, $\braket{u,\alpha_{X,u}(\sigma)}$ is
 an object in $\{ * \} \comma X$.
 $\braket{u,\alpha_{X,u}(\sigma)}$ is an initial object of $\{ * \} \comma X$
 if and only if for any $\braket{x,a} \in \{ * \} \comma X$
 (i.e., for any $x \in \cD$ and $a \in Xx$), there exists a unique $\ol{a} \in \cD(u,x)$
 satisfying
 \begin{alignat}{1}
  a = X \ol{a} \c \alpha_{X,u}(\sigma) &\qquad\diagram\qquad
  \footnoteinset{-0.06}{0.3}{\eqref{eq:univ}}{%
  \InsertMidPDF{repr_repr_univ.pdf}}.
  \label{eq:repr_repr_univ}
 \end{alignat}
 Note that the area enclosed by the auxiliary line represents the initial object
 $\braket{u,\alpha_{X,u}(\sigma)}$.
 The right-hand side of Eq.~\eqref{eq:repr_repr_univ} can be rewritten as
 \begin{alignat}{1}
  \footnoteinset{1.52}{0.3}{\eqref{eq:repr_yoneda_init_summary}}{%
  \InsertPDF{repr_repr_univ2.pdf}} \raisebox{1em}{.}
  \label{eq:repr_repr_univ2}
 \end{alignat}
 Therefore, $\braket{u,\alpha_{X,u}(\sigma)}$ is an initial object of $\{ * \} \comma X$
 if and only if for any $\braket{x,a} \in \{ * \} \comma X$, there exists
 a unique $\ol{a} \in \cD(u,x)$ satisfying
 \begin{alignat}{1}
  a = \sigma_x(\ol{a}) &\qquad\diagram\qquad
  \footnoteinsets{-0.01}{0.3}{\eqref{eq:repr_repr_univ}}{\eqref{eq:repr_repr_univ2}}{%
  \InsertMidPDF{repr_repr_univ3.pdf}}.
  \label{eq:repr_repr_univ3}
 \end{alignat}
 Furthermore, this is equivalent to each map $\sigma_x \colon \cD(u,x) \to Xx$ being invertible
 (i.e., an isomorphism in $\Set$)%
 \footnote{Proof: Arbitrarily choose $x \in \cD$. If for each $a \in Xx$, there exists a unique
 $\ol{a}$ satisfying Eq.~\eqref{eq:repr_repr_univ3}, then the map
 $\rho_x \colon Xx \ni a \mapsto \ol{a} \in \cD(u,x)$ is the inverse of $\sigma_x$ since
 \begin{alignat}{2}
  \sigma_x(\rho_x(a)) &= a ~(\forall a \in Xx), \qquad
  \rho_x(\sigma_x(b)) = b ~(\forall b \in \cD(u,x))
 \end{alignat}
 holds.
 Therefore, $\sigma_x$ is invertible.
 Conversely, if $\sigma_x$ is invertible, then $\ol{a}$ satisfying Eq.~\eqref{eq:repr_repr_univ3}
 (i.e., $a = \sigma_x(\ol{a})$) is uniquely determined by $\ol{a} = \sigma_x^{-1}(a)$.},
 and thus equivalent to $\sigma$ being a natural isomorphism.
 Since $\sigma$ being a natural isomorphism is equivalent to $\braket{u,\sigma}$ being
 a representation for $X$, Eq.~\eqref{eq:repr_univ_usigma_another} holds.
\end{proof}

Substituting $\eta \coloneqq \alpha_{X,u}(\sigma)$ into Eq.~\eqref{eq:repr_univ_usigma} yields
\begin{alignat}{1}
 \lefteqn{ \text{$\braket{u,\alpha_{X,u}^{-1}(\eta)}$ is a representation for $\cC(c,G\Endash)$} }
 \nonumber \\
 &\qquad\Leftrightarrow\qquad \text{$\braket{u,\eta}$ is
 a universal morphism from $c$ to $G$}.
 \label{eq:repr_univ_usigma2}
\end{alignat}
These equations indicate that two seemingly different concepts, the representation
$\braket{u,\sigma}$ of $\cC(c,G\Endash)$ and the universal morphism $\braket{u,\eta}$
from $c$ to $G$, can be interconverted via the Yoneda map $\alpha_{X,u}$ and its inverse.
The relationship between $\sigma$ and $\eta$ is represented by
\begin{alignat}{2}
 \eta = \alpha_{X,u}(\sigma) &\qquad\diagram\qquad
 \InsertMidPDF{repr_repr_eta_sigma1.pdf}, \nonumber \\
 \sigma = \alpha^{-1}_{X,u}(\eta) &\qquad\diagram\qquad
 \InsertMidPDF{repr_repr_eta_sigma2.pdf}.
 \label{eq:repr_repr_eta_sigma}
\end{alignat}
It can be seen that the black circle representing $\eta$ is connected to three wires: $G,c,u$,
while the block representing $\sigma$ is connected to three wires: $G,\yoneda{c},\yoneda{u}$.
Intuitively, it can be said that $\sigma$ corresponds to replacing the two wires $c$ and $u$ in $\eta$
with $\yoneda{c}$ and $\yoneda{u}$, respectively.

As a special case of this theorem, considering the case where $c$ is $\{ * \} \in \Set$ and
$G$ is $X \colon \cC \to \Set$ yields
\begin{alignat}{1}
 \text{a universal morphism from $\{ * \}$ to $X$ exists}
 &\qquad\Leftrightarrow\qquad \text{$X$ is representable},
 \label{eq:adj_1X_repr}
\end{alignat}
where we used the equality $\Set(\{ * \},X\Endash) = X$.
We can consider the case of a presheaf $X \colon \cC^\op \to \Set$
by simply replacing $\cC$ with $\cC^\op$;
in this case, Eq.~\eqref{eq:adj_1X_repr} still holds.

\subsubsection{Basic diagrams for universal morphisms} \label{subsubsec:repr_repr_diagram}

Assuming that there exists a universal morphism from $c \in \cC$ to $G \colon \cD \to \cC$,
we show basic diagrams for this morphism.
From Theorem~\ref{thm:ReprCG}, there exists such a universal morphism if and only if
the set-valued functor $\cC(c,G\Endash)$ is representable.
Let $\braket{u,\sigma}$ be a representation for $\cC(c,G\Endash)$; then,
$\sigma \colon \yoneda{u} \ntocong \cC(c,G\Endash)$ holds.
Using a diagram like Eq.~\eqref{eq:basic_natural_trans_CcG},
this isomorphism is expressed by
\begin{alignat}{1}
 \cD(u,\Endash) \cong \cC(c,G\Endash) &\qquad\diagram\qquad
 \InsertMidPDF{repr_cong.pdf}.
 \label{eq:repr_cong}
\end{alignat}
Let $\eta \coloneqq \alpha_{X,u}(\sigma)$; then, from Eq.~\eqref{eq:repr_univ_usigma},
$\braket{u,\eta}$ is a universal morphism from $c$ to $G$.
For each $x \in \cD$, $\sigma_x$ maps a morphism $\ol{a} \in \cD(u,x)$ to $a \in \cC(c,Gx)$ represented by
\begin{alignat}{2}
 \InsertMidPDF{repr_def_inv.pdf} &\qquad\xmapsto{\sigma_x}\qquad
 \InsertMidPDF{repr_def_inv2.pdf},
 \label{eq:repr_natural_trans_reprdef_inv}
\end{alignat}
where the black circle represents the universal morphism $\eta$.
The right-hand side of the symbol ``$\xmapsto{\sigma_x}$'' corresponds to Eq.~\eqref{eq:univ}.
Conversely, each $a \in \cC(c,Gx)$ is mapped to $\ol{a} \in \cD(u,x)$ by $\sigma^{-1}_x$,
which is represented by
\begin{alignat}{2}
 \InsertMidPDF{repr_def.pdf}
 &\qquad\xmapsto{\sigma^{-1}_x}\qquad
 \InsertMidPDF{repr_def2.pdf},
 \label{eq:repr_natural_trans_reprdef}
\end{alignat}
where we represent the natural transformation $\sigma^{-1}$ with the blue block
in the shape of ``$\raisebox{-.1em}{\InsertPDFtext{report_text_repr_def.pdf}}$''.
The morphism $a$ inside this block is an input to $\sigma^{-1}$.

For any $a \in \cC(c,Gx)$ and $b \in \cD(u,x)$,
we have $\sigma_x(\sigma_x^{-1}(a)) = a$ and $\sigma_x^{-1}(\sigma_x(b)) = b$, i.e.,
\begin{alignat}{1}
 \InsertPDF{repr_twice.pdf} \raisebox{1em}{.}
 \label{eq:repr_twice}
\end{alignat}
Thus, we have for any $b,b' \in \cD(u,x)$,
\begin{alignat}{1}
 \InsertMidPDF{repr_univ_eq1.pdf}
 &\qquad\Rightarrow\qquad
 \InsertMidPDF{repr_univ_eq2.pdf}.
 \label{eq:repr_univ_eq}
\end{alignat}
Also, for each $f \in \cD(x,y)$, from
\begin{alignat}{1}
 \InsertPDF{repr_natural_trans_inv2.pdf} \raisebox{1em}{,}
 \label{eq:repr_natural_trans_inv2}
\end{alignat}
we have
\begin{alignat}{1}
 \footnoteinset{-2.71}{0.3}{\eqref{eq:repr_natural_trans_reprdef}}{%
 \footnoteinset{-0.46}{0.3}{\eqref{eq:repr_natural_trans_reprdef}}{%
 \footnoteinset{2.71}{0.3}{\eqref{eq:repr_natural_trans_inv2}}{%
 \InsertPDF{repr_natural_trans2.pdf}}}} \raisebox{1em}{.}
 \label{eq:repr_natural_trans2}
\end{alignat}
Equation~\eqref{eq:repr_natural_trans2} is nothing but naturality of $\sigma^{-1}$,
i.e., $(f \c \Endash) \c \sigma^{-1}_x = \sigma^{-1}_y \c (Gf \c \Endash)$.

\begin{ex}{the Yoneda lemma}{}
 In Eq.~\eqref{eq:repr_yoneda_cong2}, we introduced a blue block in the shape of
 ``$\raisebox{-.1em}{\InsertPDFtext{report_text_repr_yoneda_cong2.pdf}}$''.
 It can be considered as a special case of the blue block introduced
 in Eq.~\eqref{eq:repr_natural_trans_reprdef}.
 Indeed, from Example~\ref{ex:ReprYonedaUniv}, $\braket{\yoneda{c},\id_c}$ is
 a universal morphism from $\{ * \}$ to $\ev_c$.
 Considering Eq.~\eqref{eq:repr_natural_trans_reprdef} for this universal morphism
 (by substituting $\{ * \}$ for $c$ and $\ev_c$ for $G$)
 yields the right-hand side of the comma in Eq.~\eqref{eq:repr_yoneda_cong2}.
\end{ex}

Adjoints, (co)limits, and Kan extensions, which will be discussed in Sections
\ref{sec:adj}, \ref{sec:limit}, and \ref{sec:Kan}, can be seen as
special cases of universal morphisms.
The relationships between these concepts and universal morphisms are summarized in Table~\ref{tab:repr_examples}.
This table also shows the diagrams mainly used in this paper.
As will be shown in Subsubsection~\ref{subsubsec:Kan_property_adj}, adjoints and limits are
also special cases of Kan extensions.
\begin{table}[hbt]
 \centering
 \caption{Relationships among universal morphisms, adjoints, (co)limits, and Kan extensions
 (the equations within the brackets correspond to each other).}
 \label{tab:repr_examples}
 \begin{tabular}{c|l|c}
  \bhline{1pt}
  universal morphisms & \multirow{2}{*}{
      \begin{minipage}{18mm}
       \centering\InsertMidPDF{repr_univ_general.pdf}
      \end{minipage}}
      & a universal morphism $\braket{u,\eta}$ \\
  & & from $c \in \cC$ to $G \colon \cD \to \cC$ \\
  & & [Eq.~\eqref{eq:univ}, ~Eq.~\eqref{eq:repr_cong}] \\ \hline
  left adjoints & \multirow{2}{*}{
      \begin{minipage}{18mm}
       \centering\InsertMidPDF{repr_univ_adj.pdf}
      \end{minipage}}
      & a universal morphism $\braket{u_c,\eta_c}$ \\
  (Theorem~\ref{thm:AdjNasUniv})
  & & from $c \in \cC$ to $G \colon \cD \to \cC$ (for each $c$) \\
  & & ($F \colon \cC \to \cD$ with $Fc = u_c$ is a left adjoint to $G$) \\
  & & [Eq.~\eqref{eq:adj_conj_inv}, ~Eq.~\eqref{eq:adj_def_nat2}] \\ \hline
  colimits & \multirow{2}{*}{
      \begin{minipage}{18mm}
       \centering\InsertMidPDF{repr_univ_colim.pdf}
      \end{minipage}}
      & a universal morphism $\braket{d,\eta}$ \\
  (Subsubsection~\ref{subsubsec:limit_def_colimit})
  & & from $D \in \Func{\cJ}{\cC}$ to
          $\Delta_\cJ = \Endash \b {!} \colon \cC \to \Func{\cJ}{\cC}$ \\
  & & (i.e., a colimit of $D$) \\
  & & [Eq.~\eqref{eq:limit_repr_terminal_op}, ~Eq.~\eqref{eq:limit_repr_cong_op}] \\ \hline
  left Kan extensions & \multirow{2}{*}{
      \begin{minipage}{18mm}
       \centering\InsertMidPDF{repr_univ_Kan.pdf}
      \end{minipage}}
      &  a universal morphism $\braket{L,\eta}$ \\
  (Definition~\ref{def:Kan})
  & & from $F \in \Func{\cC}{\cE}$ to
          $\Endash \b K \colon \Func{\cD}{\cE} \to \Func{\cC}{\cE}$ \\
  & & (i.e., a left Kan extension of $F$ along $K$) \\
  & & [Eq.~\eqref{eq:Kan_left_universal}, ~Eq.~\eqref{eq:Kan_left_cong}] \\
  \bhline{1pt}
 \end{tabular}
\end{table}

\subsubsection{Comma category $G \comma c$} \label{subsubsec:repr_cG_Gc}

We can consider an upside-down version of the comma category $c \comma G$:
\begin{itemize}
 \item Its each object is a pair, $\braket{x,a}$, of an object $x \in \cD$ and
       a morphism $a \in \cC(Gx,c)$.
 \item Its each morphism from an object $\braket{x,a}$ to an object $\braket{y,a'}$
       is a morphism $f \in \cD(x,y)$ in $\cD$ satisfying
       \begin{alignat}{1}
        a' \c Gf = a &\qquad\diagram\qquad
        \InsertMidPDF{repr_Gc.pdf}.
        \label{eq:repr_Gc}
       \end{alignat}
\item The composite of its morphisms is the composite of morphisms in $\cD$,
      and its identity morphism is the identity morphism in $\cD$.
\end{itemize}
This category is denoted by $G \comma c$ and is also called a comma category.
Note that $G \comma c$ is not exactly the opposite of $c \comma G$
(please compare Eq.~\eqref{eq:repr_cG} and Eq.~\eqref{eq:repr_Gc}).

\begin{define}{universal morphisms}{UnivOp}
 A terminal object%
 \footnote{A \termdef{terminal object} in $\cC$ is an object $a$ in $\cC$ such that
 for any $x \in \cC$, there exists a unique morphism with domain $x$ and codomain $a$.}
 $\braket{u,\varepsilon}$ (or simply $\varepsilon$)
 in a comma category $G \comma c$ is called a universal morphism from $G$ to $c$.
 In other words, $\braket{u,\varepsilon}$ with $u \in \cD$ and $\varepsilon \in \cC(Gu,c)$
 is called a universal morphism from $G$ to $c$
 if, for any morphism $a \in \cC(Gx,c)$ with any $x \in \cD$,
 there exists a unique morphism $\ol{a} \in \cD(x,u)$ satisfying
\begin{alignat}{1}
 a = \varepsilon \c G \ol{a} &\qquad\diagram\qquad
 \InsertMidPDF{repr_univ_op.pdf},
 \label{eq:univ_op}
\end{alignat}%
 where the black circle represents $\varepsilon$.
\end{define}

Since terminal objects are essentially unique, universal morphisms from $G$ to $c$ are
also essentially unique.
It follows from the dual of Lemma~\ref{lemma:CommacGX} that $G \comma c$ is equal to
the comma category $(\cC(G\Endash,c))^\op \comma \{ * \}$
(where $(\cC(G\Endash,c))^\op \colon \cD \to \Set^\op$ and $\{ * \} \in \Set^\op$).
Also, as the dual of Theorem~\ref{thm:ReprCG}, we have
\begin{alignat}{1}
 \text{a universal morphism from $G$ to $c$ exists}
 &\qquad\Leftrightarrow\qquad \text{$\cC(G\Endash,c)$ is representable}.
 \label{eq:adj_CG_repr_op}
\end{alignat}

By vertically flipping diagrams for a universal morphism from $c$ to $G$,
we obtain corresponding diagrams for a universal morphism from $G$ to $c$.
For example, if $\braket{u,\sigma}$ is a representation for the functor $\cC(G\Endash,c)$,
then any object $\braket{x,a}$ in $G \comma c$ is mapped to the following morphism
$\ol{a} \in \cD(x,u)$ by $\sigma^{-1}_x$:
\begin{alignat}{2}
 \InsertMidPDF{repr_def_op.pdf}
 &\qquad\xmapsto{\sigma^{-1}_x}\qquad
 \InsertMidPDF{repr_def2_op.pdf},
 \label{eq:repr_def_op}
\end{alignat}
which corresponds to the upside-down version of Eq.~\eqref{eq:repr_natural_trans_reprdef}.

The category of elements of a set-valued functor $Y \colon \cC^\op \to \Set$, $\el(Y)$, is
defined as the comma category $Y^\op \comma \{ * \}$ with $\{ * \} \in \Set^\op$.
Then, we have $\el(\cC(G\Endash,c)) = G \comma c$.
Some basic properties of comma categories will be discussed in
Subsubsection~\ref{subsubsec:Kan_pointwise_comma}.

\subsection{Basic properties of universal morphisms} \label{subsec:repr_univ} 

We show some basic properties of universal morphisms as the following three lemmas,
which will be used in the following sections.

\begin{lemma}{}{ReprUnivCirc}
 Consider any object $c \in \cC$ and functors $G \colon \cD \to \cC$ and $G' \colon \cD' \to \cD$.
 Assume that there exists a universal morphism $\braket{u,\eta}$ from $c$ to $G$;
 then, the following properties hold:
 \begin{enumerate}
  \item If there exists a universal morphism $\braket{u',\eta'}$ from $u$ to $G'$,
        then $\braket{u',v}$ with $v \coloneqq (G \b \eta') \c \eta$
        is a universal morphism from $c$ to $G \b G'$.
  \item If there exists a universal morphism $\braket{u',v}$ from $c$ to $G \b G'$,
        then there exists a universal morphism $\braket{u',\eta'}$ from $u$ to $G'$
        satisfying $v = (G \b \eta') \c \eta$.
 \end{enumerate}
\end{lemma}
Note that the relationships between the universal morphisms $\braket{u,\eta}$, $\braket{u',\eta'}$,
and $\braket{u',v}$ can be represented by
\begin{alignat}{1}
 \InsertPDF{repr_univ_circ_u.pdf} \raisebox{1em}{,}
 \label{eq:repr_univ_circ_u}
\end{alignat}
where the black circle represents $\eta$.
Property~(1) asserts that $v \coloneqq (G \b \eta') \c \eta$, which is something corresponding to
the ``vertical composite of universal morphisms $\eta$ and $\eta'$'', is a universal morphism.
Conversely, Property~(2) asserts that if such a universal morphism $v$ exists, then $v$ can be
decomposed into two universal morphisms $\eta$ and $\eta'$.
\begin{proof}
 (1): For any object $\braket{x,a}$ in $c \comma (G \b G')$, there exist
 unique morphisms $\ol{a} \in \cD(u,G'x)$ and $\ol{a'} \in \cD'(u',x)$
 satisfying
 \begin{alignat}{1}
  \footnoteinset{-1.77}{0.3}{\eqref{eq:univ}}{%
  \footnoteinset{1.53}{0.3}{\eqref{eq:univ}}{%
  \InsertPDF{repr_univ_circ.pdf}}} \raisebox{1em}{,}
  \label{eq:repr_univ_circ}
 \end{alignat}
 where the black circles represent $\eta$, and the morphism
 enclosed by the auxiliary line is $v$.
 Therefore, $\braket{u',v}$ is a universal morphism from $c$ to $G \b G'$.

 (2): By the universal property of $\eta$, there exists a unique $\eta' \in \cD(u,G'u')$
 satisfying Eq.~\eqref{eq:repr_univ_circ_u} (i.e., $v = (G \b \eta') \c \eta$).
 For any object $\braket{x,a}$ in $u \comma G'$, there exists a unique morphism $\ol{a} \in \cD'(u',x)$
 satisfying
 \begin{alignat}{1}
  \footnoteinset{-1.64}{0.35}{\eqref{eq:univ}}{%
  \footnoteinset{1.80}{0.35}{\eqref{eq:repr_univ_circ_u}}{%
  \InsertPDF{repr_univ_decomp.pdf}}} \raisebox{1em}{,}
  \label{eq:repr_univ_decomp}
 \end{alignat}
 where the first equality follows from the universal property of $v$ for the morphism
 enclosed by the auxiliary line.
 Therefore, by the universal property of $\eta$ (recall Eq.~\eqref{eq:repr_univ_eq}),
 $\braket{u',\eta'}$ is a universal morphism from $u$ to $G'$.
\end{proof}

\begin{lemma}{}{ReprIsoUniv}
 Consider any object $c \in \cC$ and functor $G \colon \cD \to \cC$.
 Arbitrarily choose any category $\cC'$ isomorphic to $\cC$ and invertible functor $P \colon \cC \to \cC'$.
 Then, $\braket{u,\eta}$ is a universal morphism from $c$ to $G$ if and only if
 $\braket{u,P \b \eta}$ is a universal morphism from $Pc$ to $P \b G$.
 Dually, $\braket{u',\varepsilon}$ is a universal morphism from $G$ to $c$ if and only if
 $\braket{u',P \b \varepsilon}$ is a universal morphism from $P \b G$ to $Pc$.
\end{lemma}
\begin{proof}
 We show that $\braket{u,\eta}$ is a universal morphism from $c$ to $G$ if and only if
 $\braket{u,P \b \eta}$ is a universal morphism from $Pc$ to $P \b G$;
 the remainder of the proof is obvious from its dual%
 \footnote{Properties related to duality can be shown simply by considering their duals,
 so we usually omit their proofs.}.
 It suffices to show the ``only if'' part since the ``if'' part is obvious
 by considering $P^{-1}$ instead of $P$.
 Let $\braket{u,\eta}$ be a universal morphism from $c$ to $G$.
 Also, arbitrarily choose any $x \in \cD$ and $a' \in \cC'(Pc,(P \b G)x)$ and let
 $a \coloneqq P^{-1} \bullet a'$; then,
 there exists a unique $\ol{a}$ satisfying
 \begin{alignat}{1}
  \footnoteinset{1.85}{0.3}{\eqref{eq:univ}}{%
  \InsertPDF{repr_iso_univ.pdf}} \raisebox{1em}{,}
  \label{eq:repr_iso_univ}
 \end{alignat}
 where the black circle represents $\eta$.
 The natural transformation enclosed by the auxiliary line is $P \b \eta$.
 Therefore, $\braket{u,P \b \eta}$ is a universal morphism from $Pc$ to $P \b G$.
\end{proof}

\begin{lemma}{}{ReprFuncH}
 Consider any functors $H \colon \cJ \to \cC$ and $G \colon \cD \to \cC$.
 Assume that for each $j \in \cJ$, there exists a universal morphism
 $\braket{u_j,\eta_j}$ from $Hj$ to $G$.
 Then, the following two properties hold:
 \begin{enumerate}
  \item There exists a unique functor $F \colon \cJ \to \cD$ that maps each object $j \in \cJ$
        to $Fj = u_j$ and such that $\eta \coloneqq \{ \eta_j \}_{j \in \cJ}$ is
        a natural transformation from $H$ to $G \b F$.
        Note that $\eta$ can be represented by
        \begin{alignat}{1}
         \InsertPDF{repr_funcH_univ.pdf} \raisebox{1em}{,}
         \label{eq:repr_funcH_univ}
        \end{alignat}
        where the blue circle represents $\eta$ and the black ellipse represents $\eta_j$.
  \item $\braket{F,\eta}$ is a universal morphism from the object $H \in \Func{\cJ}{\cC}$ to
        the functor $G \b \Endash \colon \Func{\cJ}{\cD} \to \Func{\cJ}{\cC}$
        (i.e., an initial object in $H \comma (G \b \Endash)$).
 \end{enumerate}
\end{lemma}
\begin{proof}
 (1): First, assuming that there exists a functor $F$ satisfying the conditions,
 we show that $F$ is uniquely determined.
 By assumption, we have for any $f \in \cJ(j,j')$,
 \begin{alignat}{1}
  \footnoteinset{-1.51}{0.3}{\eqref{eq:nat}}{%
  \footnoteinset{1.53}{0.3}{\eqref{eq:univ}}{%
  \InsertPDF{repr_funcH.pdf}}} \raisebox{1em}{,}
  \label{eq:repr_funcH}
 \end{alignat}
 where the black ellipses represent $\eta_j$ and $\eta_{j'}$ and $\ol{f}$ is the morphism
 uniquely determined by the universal property of $\eta_j$.
 Thus, we obtain $Ff = \ol{f}$ from Eq.~\eqref{eq:repr_univ_eq}.
 Therefore, if a functor $F$ exists, then it is uniquely determined
 as one that maps each morphism $f$ in $\cJ$ to $\ol{f}$.

 Next, we show that $F$ defined as above is a functor.
 We have for any $f \in \cJ(j,j')$ and $g \in \cJ(j',j'')$,
 \begin{alignat}{1}
  \footnoteinset{-3.04}{0.3}{\eqref{eq:nat}}{%
  \footnoteinset{0.00}{0.3}{\eqref{eq:univ}}{%
  \footnoteinset{3.04}{0.3}{\eqref{eq:univ}}{%
  \InsertPDF{repr_funcH_nat.pdf}}}} \raisebox{1em}{.}
  \label{eq:repr_funcH_nat}
 \end{alignat}
 Therefore, it follows from Eq.~\eqref{eq:repr_univ_eq} that $F(gf) = Fg \c Ff$ holds.
 Also, $F \id_j = \id_{Fj}$ obviously holds, and thus $F$ is a functor.
 It can be easily seen that $F$ satisfies the conditions.

 (2): It suffices to show that for any natural transformation $\alpha \colon H \nto G \b X$
 with any $X \colon \cJ \to \cD$,
 there exists a unique natural transformation $\ol{\alpha} \colon F \nto X$ satisfying
 \begin{alignat}{1}
  \footnoteinset{-0.14}{0.3}{\eqref{eq:univ}}{%
  \InsertPDF{repr_funcH_alpha_univ.pdf}} \raisebox{1em}{,}
  \label{eq:repr_funcH_alpha_univ}
 \end{alignat}
 where the blue circle represents $\eta$.
 Since the first and second equalities hold, it suffices to show the last equation,
 which is equivalent to $\ol{\alpha} \coloneqq \{ \ol{\alpha_j} \}_{j \in \cJ}$
 being a natural transformation.
 We have
 \begin{alignat}{1}
  \footnoteinset{-3.30}{0.3}{\eqref{eq:repr_funcH_alpha_univ}}{%
  \footnoteinset{0.00}{0.3}{\eqref{eq:sliding}}{%
  \footnoteinsets{3.31}{0.3}{\eqref{eq:repr_funcH_alpha_univ}}{\eqref{eq:nat}}{%
  \InsertPDF{repr_funcH_alpha_nat.pdf}}}} \raisebox{1em}{.}
  \label{eq:repr_funcH_alpha_nat}
 \end{alignat}
 Therefore,
 \begin{alignat}{1}
  \InsertPDF{repr_funcH_alpha_nat2.pdf}
  \label{eq:repr_funcH_alpha_nat2}
 \end{alignat}
 holds from Eq.~\eqref{eq:repr_univ_eq}, and thus $\ol{\alpha}$ is a natural transformation.
\end{proof}


\section{Adjunctions} \label{sec:adj}

Roughly speaking, an adjunction is a concept that describes the relationship between
two functors that are weakly inverse to each other.
It is well known that adjunctions can be intuitively understood through
the use of string diagrams.
Adjunctions are also closely related to monads.
String diagrams for adjunctions and monads are discussed in detail in Ref.~\cite{Mar-2014}.
In this paper, we briefly describe some basic properties of adjunctions in a self-contained manner.
Monads are not discussed in this paper, except for a brief coverage
in Subsubsection~\ref{subsubsec:Kan_property_monad}.

\subsection{Definition of adjunctions} \label{subsec:adj_def}

\begin{define}{adjunctions}{Adjunction}
 Let us consider a pair of two functors $F \colon \cC \to \cD$ and $G \colon \cD \to \cC$.
 When there is a natural isomorphism
 $\varphi \colon \cD(F\Endash,\Enndash) \ntocong \cC(\Endash,G\Enndash)$,
 we call $\braket{F,G,\varphi}$ an \termdef{adjunction}.
 Also, $F$ is called a \termdef{left adjoint} to $G$, and $G$ is called
 a \termdef{right adjoint} to $F$.
 The symbol ``$\dashv$'' is used to denote adjunctions: $F \dashv G$ means that $F$ is a left adjoint
 to $G$.
\end{define}

$\varphi = \{ \varphi_{c,d} \}_{\braket{c,d} \in \cC \times \cD}$ is
a collection of morphisms (i.e., maps) in $\Set$,
whose each component $\varphi_{c,d}$ is a bijective map from $\cD(Fc,d)$ to $\cC(c,Gd)$.
Naturality of $\varphi$ is equivalent to satisfying
\begin{alignat}{1}
 \InsertPDF{adj_def_psi.pdf}
 \label{eq:adj_def_psi}
\end{alignat}
for any $g \in \cC^\op(c,c') = \cC(c',c)$ and $h \in \cD(d,d')$,
where we used the expression for direct products (recall Subsection~\ref{subsec:category_prod}).
This equation is equivalent to satisfying
\begin{alignat}{2}
 \InsertPDF{adj_phi_nat.pdf}
 \label{eq:adj_phi_nat}
\end{alignat}
for each $f \in \cD(Fc,d)$ (see Eq.~\eqref{eq:basic_functor_Hom}).
Using the expression of Eq.~\eqref{eq:basic_natural_trans_CcG},
the isomorphism $\cD(F\Endash,\Enndash) \cong \cC(\Endash,G\Enndash)$ can be
represented by
\begin{alignat}{1}
 \InsertPDF{adj_def_nat2.pdf} \raisebox{1em}{.}
 \label{eq:adj_def_nat2}
\end{alignat}

\subsection{Necessary and sufficient conditions for adjunctions} \label{subsec:adj_nas}

Necessary and sufficient conditions for $F \dashv G$ are shown by two theorems
(Theorems~\ref{thm:AdjNasSnake} and \ref{thm:AdjNasUniv}).

\subsubsection{Preliminaries} \label{subsubsec:adj_nas_preliminary}

\begin{lemma}{}{AdjUniv}
 Consider an adjunction $\braket{F,G,\varphi}$ with $F \colon \cC \to \cD$.
 Also, let $\eta_c \coloneqq \varphi_{c,Fc}(\id_{Fc})$ and
 $\varepsilon_d \coloneqq \varphi^{-1}_{Gd,d}(\id_{Gd})$.
 Then, the following properties hold:
 \begin{enumerate}
  \item For each $c \in \cC$, $\braket{Fc,\eta_c}$ is
        a universal morphism from $c$ to $G$.
  \item For each $d \in \cD$, $\braket{Gd,\varepsilon_d}$ is
        a universal morphism from $F$ to $d$.
 \end{enumerate}
\end{lemma}
\begin{proof}
 (1): For each $c \in \cC$, $\{ \varphi_{c,d} \}_{d \in \cD}$ is a natural isomorphism
 from $\cD(Fc,\Endash)$ to $\cC(c,G\Endash)$.
 Thus, $\braket{Fc, \{ \varphi_{c,d} \}_{d \in \cD}}$ is a representation for $\cC(c,G\Endash)$.
 Therefore, from Eq.~\eqref{eq:repr_univ_usigma}, $\braket{Fc, \eta_c}$ is a universal morphism
 from $c$ to $G$.

 (2): The dual of (1) proves (2).
\end{proof}

For each $c \in \cC$ and $d \in \cD$, 
$f \in \cD(Fc,d)$ and $\ol{f} \coloneqq \varphi_{c,d}(f) \in \cC(c,Gd)$
correspond one-to-one.
Using the diagram in Eq.~\eqref{eq:repr_natural_trans_reprdef_inv} to represent the universal morphism
$\braket{Fc,\eta_c}$, we obtain
\begin{alignat}{2}
 \InsertMidPDF{adj_conj_inv.pdf}
 &\qquad\xmapsto{\varphi_{c,d}}\qquad
 \InsertMidPDF{adj_conj_inv2.pdf},
 \label{eq:adj_conj_inv}
\end{alignat}
where the area enclosed by the gray dotted line represents the morphism $\eta_c$.
The right-hand side of the symbol ``$\xmapsto{\varphi_{c,d}}$''
corresponds to Eq.~\eqref{eq:univ}.
Similarly, using the universal property of $\braket{Gd, \varepsilon_d}$, we obtain
\begin{alignat}{2}
 \InsertMidPDF{adj_conj.pdf}
 &\qquad\xmapsto{\varphi_{c,d}^{-1}}\qquad
 \InsertMidPDF{adj_conj2.pdf},
 \label{eq:adj_conj}
\end{alignat}
where the area enclosed by the gray dotted line represents the morphism $\varepsilon_d$.
The right-hand side of the symbol ``$\xmapsto{\varphi_{c,d}^{-1}}$''
corresponds to Eq.~\eqref{eq:univ_op}.
Here, the morphisms $\eta_c$ and $\varepsilon_d$ are represented by special-shaped blocks;
it will soon become clear that these representations help with intuitive understanding
and that the gray dotted lines can be considered simply as auxiliary lines.
From Eq.~\eqref{eq:adj_conj_inv}, Eq.~\eqref{eq:adj_phi_nat} can be expressed by
\begin{alignat}{2}
 \InsertMidPDF{adj_phi_nat_elegant.pdf} \qquad (\forall f \in \cD(Fc,d)).
 \label{eq:adj_phi_nat_elegant}
\end{alignat}
Similarly, using the notation of Eq.~\eqref{eq:adj_conj}, naturality of
$\varphi^{-1} = \{ \varphi^{-1}_{c,d} \}_{c \in \cC, d \in \cD}$ is represented
by the fulfillment of
\begin{alignat}{2}
 \InsertMidPDF{adj_phi_nat_elegant_inv.pdf} \qquad (\forall \ol{f} \in \cC(c,Gd))
 \label{eq:adj_phi_nat_elegant_inv}
\end{alignat}
for any $g \in \cC^\op(c,c') = \cC(c',c)$ and $h \in \cD(d,d')$.
 
\subsubsection{Necessary and sufficient conditions using unit and counit}

The first necessary and sufficient condition uses two natural transformations
$\eta$ and $\varepsilon$.

\begin{thm}{}{AdjNasSnake}
 Consider a pair of two functors $F \colon \cC \to \cD$ and $G \colon \cD \to \cC$.
 The following are equivalent:
 \begin{enumerate}
  \item $F \dashv G$ holds.
  \item There exist two natural transformations $\eta \colon \id_\cC \nto G \b F$ and
        $\varepsilon \colon F \b G \nto \id_\cD$ satisfying
        \begin{alignat}{1}
         \lefteqn{ (G \b \varepsilon) \c (\eta \b G) = \id_G,
         ~(\varepsilon \b F) \c (F \b \eta) = \id_F } \nonumber \\
         &\diagram\qquad
         \InsertMidPDF{adj_zigzag.pdf},
         \tag{zigzag}
         \label{eq:zigzag}
        \end{alignat}%
        where  $\eta$ and $\varepsilon$ are represented by the following semicircular blocks:
        \begin{alignat}{1}
         \InsertPDF{adj_eta_epsilon.pdf} \raisebox{1em}{.}
         \label{eq:adj_eta_epsilon}
        \end{alignat}
 \end{enumerate}
\end{thm}
$\eta$ is called the \termdef{unit} of the adjunction $F \dashv G$,
and $\varepsilon$ is called its \termdef{counit}.
Intuitively, $\varepsilon$ can be interpreted as the ``natural transformation obtained by rotating
$\eta$ by 180 degrees''.
Equation~\eqref{eq:zigzag} is often called the \termdef{zigzag} (or triangle) \termdef{identities}.
\begin{proof}
 $(1) \Rightarrow (2)$:
 For each $d \in \cD$, by setting $c = Gd$ and $\ol{f} = \id_{Gd}$
 in Eqs.~\eqref{eq:adj_conj_inv} and \eqref{eq:adj_conj}, we obtain
 \begin{alignat}{1}
  \footnoteinsets{0.63}{0.3}{\eqref{eq:adj_conj}}{\eqref{eq:adj_conj_inv}}{%
  \InsertPDF{adj_zigzag_proof1.pdf}} \raisebox{1em}{.}
  \label{eq:adj_zigzag_proof1}
 \end{alignat}
 Similarly, for each $c \in \cC$, by setting $d = Fc$ and $f = \id_{Fc}$, we obtain
 \begin{alignat}{1}
  \footnoteinsets{0.53}{0.3}{\eqref{eq:adj_conj_inv}}{\eqref{eq:adj_conj}}{%
  \InsertPDF{adj_zigzag_proof2.pdf}} \raisebox{1em}{.}
  \label{eq:adj_zigzag_proof2}
 \end{alignat}
 It remains to show that $\eta \coloneqq \{ \eta_c \}_{c \in \cC}$ and
 $\varepsilon \coloneqq \{ \varepsilon_d \}_{d \in \cD}$ are natural transformations,
 as Eq.~\eqref{eq:zigzag} obviously holds from Eqs.~\eqref{eq:adj_zigzag_proof1} and
 \eqref{eq:adj_zigzag_proof2} if they are.
 Substituting $d = d' = Fc$ and $f = h = \id_{Fc}$ into Eq.~\eqref{eq:adj_phi_nat_elegant}
 yields the naturality condition for $\eta$
 (i.e., $\eta_c \c g = (G \b F)g \c \eta_{c'}$), thus $\eta$ is a natural transformation.
 Similarly, substituting $c = c' = Gd$ and $\ol{f} = g = \id_{Gd}$ into
 Eq.~\eqref{eq:adj_phi_nat_elegant_inv} yields the naturality condition for $\varepsilon$
 (i.e., $h \c \varepsilon_d = \varepsilon_{d'} \c (F \b G)h$),
 thus $\varepsilon$ is also a natural transformation.

 $(2) \Rightarrow (1)$:
 Defining a map $\varphi_{c,d}$ by Eq.~\eqref{eq:adj_conj_inv},
 it follows from Eq.~\eqref{eq:zigzag} that its inverse $\varphi^{-1}_{c,d}$ is
 given by Eq.~\eqref{eq:adj_conj}.
 Also, Eq.~\eqref{eq:adj_phi_nat_elegant} holds from naturality of $\eta$.
 Therefore, $\varphi$ is a natural isomorphism, and thus $F \dashv G$ holds.
\end{proof}

It can be seen that $\eta_c$ in Eq.~\eqref{eq:adj_conj_inv} and $\varepsilon_d$ in Eq.~\eqref{eq:adj_conj}
are nothing more than representations using Eq.~\eqref{eq:adj_eta_epsilon}.
Since $\eta$ and $\varepsilon$ are natural transformations, there is no problem
with removing the gray dotted lines from the above diagrams.
In what follows, these gray dotted lines are omitted.

Instead of calling $\braket{F,G,\varphi}$ an adjunction,
it is also possible to call $\braket{F,G,\eta,\varepsilon}$ an adjunction.
It follows from Eqs.~\eqref{eq:adj_conj_inv} and \eqref{eq:adj_conj} that
$\varphi$ and $\eta$, as well as $\varphi$ and $\varepsilon$, can be interconverted.

\subsubsection{Necessary and sufficient conditions using universal morphisms}

The second necessary and sufficient condition uses universal morphisms.

\begin{thm}{}{AdjNasUniv}
 For any functor $G \colon \cD \to \cC$, the following are equivalent:
 \begin{enumerate}
  \item $G$ has a left adjoint.
  \item There exists a universal morphism from each $c \in \cC$ to $G$.
 \end{enumerate}
 Also, if Condition~(1) holds, then for a left adjoint $F$ to $G$ and its unit $\eta$,
 $\braket{Fc,\eta_c}$ is a universal morphism from $c$ to $G$.
 On the other hand, if Condition~(2) holds, then for a universal morphism $\braket{d_c,\eta_c}$
 from $c$ to $G$, there exists a unique left adjoint $F$ to $G$
 such that $Fc = d_c$ holds and $\{ \eta_c \}_{c \in \cC}$ is its unit.
\end{thm}
By considering the dual of this theorem, it can be immediately shown that
a functor $F \colon \cC \to \cD$ has a right adjoint if and only if
there exists a universal morphism from $F$ to each $d \in \cD$.
\begin{proof}
 $(1) \Rightarrow (2)$: This is obvious from Lemma~\ref{lemma:AdjUniv}.

 $(2) \Rightarrow (1)$:
 For each $c \in \cC$, let $\braket{d_c,\eta_c}$ be a universal morphism from $c$ to $G$.
 Considering the case where $H$ is the identity functor in Lemma~\ref{lemma:ReprFuncH},
 there exists a unique functor $F$ that maps each $c \in \cC$ to
 $Fc = d_c$ and such that $\eta \coloneqq \{ \eta_c \}_{c \in \cC}$
 is a natural transformation from $\id_\cC$ to $G \b F$.
 Choose the functor $F$ in this way.

 For each $c \in \cC$ and $d \in \cD$, consider the map $\varphi_{c,d}$
 given by Eq.~\eqref{eq:adj_conj_inv}.
 Since $\braket{Fc,\eta_c}$ is a universal morphism from $c$ to $G$,
 it follows that $\varphi_{c,d}$ is invertible
 (see Eq.~\eqref{eq:repr_univ_usigma2} or Eq.~\eqref{eq:repr_repr_eta_sigma}).
 Also, it follows from Eq.~\eqref{eq:adj_phi_nat_elegant} that
 $\varphi \coloneqq \{ \varphi_{c,d} \}_{c \in \cC, d \in \cD}$ is a natural isomorphism.
 Therefore, $\braket{F,G,\varphi}$ is an adjunction.
 Furthermore, since $\eta_c = \varphi_{c,Fc}(\id_{Fc})$, $\eta$ is its unit.
\end{proof}

\subsection{Basic properties of adjunctions} \label{subsec:adj_property}

\subsubsection{Uniqueness of adjunctions}

Consider an adjunction $F \dashv G$.
We show that any left adjoint to $G$ is isomorphic to $F$,
which means that a left adjoint is essentially unique (see Proposition~\ref{pro:AdjUniqueNat}).
Conversely, any functor isomorphic to $F$ is a left adjoint to $G$
(see Proposition~\ref{pro:AdjUnique2}).
Dually, the same applies to right adjoints.

\begin{proposition}{}{AdjUniqueNat}
 Let $\braket{F,G,\eta,\varepsilon}$ and $\braket{F',G,\eta',\varepsilon'}$ be
 adjunctions, where $F,F' \colon \cC \to \cD$ and $G \colon \cD \to \cC$.
 Then, there exists a unique natural isomorphism $\alpha \colon F \ntocong F'$ satisfying
 \begin{alignat}{1}
  \InsertPDF{adj_unique_eta.pdf} \raisebox{1em}{,}
  \label{eq:adj_unique_eta}
 \end{alignat}
 where the area enclosed by the auxiliary line is $\eta$.
 Note that the isomorphism $F \cong F'$ implies that a left adjoint to $G$ is essentially unique
 if it exists.
 Also,
 \begin{alignat}{1}
  \InsertPDF{adj_unique_epsilon.pdf}
  \label{eq:adj_unique_epsilon}
 \end{alignat}
 holds, where the area enclosed by the auxiliary line is $\varepsilon$.
\end{proposition}
\begin{proof}
 If a natural isomorphism  $\alpha$ satisfying Eq.~\eqref{eq:adj_unique_eta} exists,
 then from
 \begin{alignat}{1}
  \footnoteinset{-3.50}{0.3}{\eqref{eq:zigzag}}{%
  \footnoteinset{-0.06}{0.3}{\eqref{eq:adj_unique_eta}}{%
  \footnoteinset{3.38}{0.3}{\eqref{eq:zigzag}}{%
  \InsertPDF{adj_unique_epsilon_alpha.pdf}}}} \raisebox{1em}{,}
  \label{eq:adj_unique_epsilon_alpha}
 \end{alignat}
 Eq.~\eqref{eq:adj_unique_epsilon} holds.
 Thus, it suffices to show that there exists a unique natural isomorphism $\alpha$
 satisfying Eq.~\eqref{eq:adj_unique_eta}.
 If there exists a natural transformation $\alpha$ satisfying Eq.~\eqref{eq:adj_unique_eta},
 then we have
 \begin{alignat}{1}
  \footnoteinset{-1.91}{0.3}{\eqref{eq:zigzag}}{%
  \footnoteinset{1.13}{0.3}{\eqref{eq:adj_unique_eta}}{%
  \InsertPDF{adj_unique_eta2.pdf}}} \raisebox{1em}{.}
  \label{eq:adj_unique_eta2}
 \end{alignat}
 Conversely, when $\alpha$ is defined by Eq.~\eqref{eq:adj_unique_eta2}, 
 Eq.~\eqref{eq:adj_unique_eta} obviously holds.
 Thus, there exists a unique natural transformation $\alpha$
 satisfying Eq.~\eqref{eq:adj_unique_eta}.
 Also, let
 \begin{alignat}{1}
  \InsertPDF{adj_unique_eta_alpha_inv.pdf} \raisebox{1em}{;}
  \label{eq:adj_unique_eta_alpha_inv}
 \end{alignat}
 then, from
 \begin{alignat}{1}
  \footnoteinsets{-2.91}{1.2}{\eqref{eq:adj_unique_eta2}}{\eqref{eq:adj_unique_eta_alpha_inv}}{%
  \footnoteinset{-2.91}{-1.15}{\eqref{eq:zigzag}}{%
  \footnoteinset{-0.01}{-1.15}{\eqref{eq:zigzag}}{%
  \InsertPDF{adj_unique_eta_alpha_inv_proof1.pdf}}}}
  \label{eq:adj_unique_eta_alpha_inv_proof1}
 \end{alignat}
 and
 \begin{alignat}{1}
  \footnoteinsets{-2.91}{1.35}{\eqref{eq:adj_unique_eta2}}{\eqref{eq:adj_unique_eta_alpha_inv}}{%
  \footnoteinset{-2.91}{-1.15}{\eqref{eq:zigzag}}{%
  \footnoteinset{-0.01}{-1.15}{\eqref{eq:zigzag}}{%
  \InsertPDF{adj_unique_eta_alpha_inv_proof2.pdf}}}} \raisebox{1em}{,}
  \label{eq:adj_unique_eta_alpha_inv_proof2}
 \end{alignat}
 we have $\beta = \alpha^{-1}$.
 Therefore, $\alpha$ is an isomorphism.
\end{proof}

\begin{proposition}{}{AdjUnique2}
 Consider an adjunction $F \dashv G$.
 Any functor isomorphic to $F$ is a left adjoint to $G$.
\end{proposition}
\begin{proof}
 Let $\alpha$ be a natural isomorphism from $F$ to a functor $F'$.
 Also, let
 \begin{alignat}{1}
  \InsertPDF{adj_cong_eta_epsilon.pdf} \raisebox{1em}{,}
  \label{eq:adj_cong_eta_epsilon}
 \end{alignat}
 where the natural transformations enclosed by the auxiliary lines on the left and right are
 respectively the unit and counit of the adjunction $F \dashv G$.
 Then, from
 \begin{alignat}{1}
  \footnoteinset{3.43}{0.3}{\eqref{eq:zigzag}}{%
  \InsertPDF{adj_cong_proof.pdf}}
  \label{eq:adj_cong_proof}
 \end{alignat}
 and
 \begin{alignat}{1}
  \footnoteinset{2.52}{0.3}{\eqref{eq:zigzag}}{%
  \InsertPDF{adj_cong_proof2.pdf}} \raisebox{1em}{,}
  \label{eq:adj_cong_proof2}
 \end{alignat}
 $\eta'$ and $\varepsilon'$ satisfy Eq.~\eqref{eq:zigzag}.
 Thus, from Theorem~\ref{thm:AdjNasSnake}, $F' \dashv G$ holds.
\end{proof}

\subsubsection{Properties of equivalences of categories} \label{subsubsec:adj_property_simeq}

An equivalence of categories, $\cC \simeq \cD$, is the existence of functors
$F \colon \cC \to \cD$ and $G \colon \cD \to \cC$ such that
$\id_\cC \cong G \b F$ and $F \b G \cong \id_\cD$.
We show that each equivalence of categories yields an adjunction.
\begin{proposition}{}{AdjCong2Adj}
 Let us consider two functors $F \colon \cC \to \cD$ and $G \colon \cD \to \cC$.
 If there exist natural isomorphisms $\eta \colon \id_\cC \ntocong G \b F$ and
 $\tau \colon F \b G \ntocong \id_\cD$, which indicates $\cC \simeq \cD$,
 then there exists an adjunction $F \dashv G$ whose unit is $\eta$.
\end{proposition}
\begin{proof}
 Let
 \begin{alignat}{1}
  \InsertPDF{adj_simeq_epsilon.pdf} \raisebox{1em}{;}
  \label{eq:adj_simeq_epsilon}
 \end{alignat}
 then, we have
 \begin{alignat}{1}
  \footnoteinset{-2.91}{0.3}{\eqref{eq:adj_simeq_epsilon}}{%
  \footnoteinset{1.05}{0.3}{\eqref{eq:sliding}}{%
  \InsertPDF{adj_simeq_zigzag2.pdf}}} \raisebox{1em}{,}
  \label{eq:adj_simeq_zigzag2}
 \end{alignat}
 where the natural transformation enclosed by the auxiliary line on the left is $\eta$.
 Also, let
 \begin{alignat}{1}
  \InsertPDF{adj_simeq_zigzag1.pdf} \raisebox{1em}{;}
  \label{eq:adj_simeq_zigzag1}
 \end{alignat}
 then, $\alpha$ is a natural isomorphism since it is composed of only natural isomorphisms.
 From
 \begin{alignat}{1}
  \footnoteinset{-2.92}{-1.0}{\eqref{eq:adj_simeq_zigzag2}}{%
  \InsertPDF{adj_simeq_zigzag1b.pdf}} \raisebox{1em}{,}
  \label{eq:adj_simeq_zigzag1b}
 \end{alignat}
 we have $\alpha\alpha = \alpha$.
 Applying $\alpha^{-1}$ to both sides yields $\alpha = \id_G$.
 Therefore, Eq.~\eqref{eq:zigzag} holds, and thus, from Theorem~\ref{thm:AdjNasSnake},
 there exists an adjunction $F \dashv G$ whose unit is $\eta$.
\end{proof}

\begin{proposition}{}{AdjFullFaithful}
 Consider an adjunction $F \dashv G$.
 The following are equivalent:
 \begin{enumerate}
  \item $G \b F \cong \id_\cC$ holds.
  \item $F$ is fully faithful.
  \item The unit $\eta$ is an isomorphism.
 \end{enumerate}
 Dually, the following are equivalent:
 \begin{enumerate}
  \item $F \b G \cong \id_\cD$ holds.
  \item $G$ is fully faithful.
  \item The counit $\varepsilon$ is an isomorphism.
 \end{enumerate}
\end{proposition}
\begin{proof}
 $(1) \Rightarrow (2)$:
 Arbitrarily choose a natural isomorphism $\phi \colon G \b F \ntocong \id_\cC$
 and two objects $a,b \in \cC$.
 It suffices to show that for each $f \in \cD(Fa,Fb)$, there exists a unique $\ol{f} \in \cC(a,b)$
 satisfying (see Eq.~\eqref{eq:basic_full_faithful})
 \begin{alignat}{1}
  \InsertPDF{adj_simeq_full_faithful1.pdf} \raisebox{1em}{.}
  \label{eq:adj_simeq_full_faithful1}
 \end{alignat}
 If such $\ol{f}$ exists, then
 \begin{alignat}{1}
  \InsertPDF{adj_simeq_full_faithful2.pdf}
  \label{eq:adj_simeq_full_faithful2}
 \end{alignat}
 uniquely determines $\ol{f}$,
 where the white and blue diamond blocks represent $\phi$ and $\phi^{-1}$, respectively.
 The first equality follows from $\phi \phi^{-1} = \id_{\id_\cC}$.
 Also, such $\ol{f}$ satisfies
 \begin{alignat}{1}
  \footnoteinset{-4.09}{2.35}{\eqref{eq:adj_simeq_full_faithful2}}{%
  \footnoteinset{0.00}{2.35}{\eqref{eq:zigzag}}{%
  \footnoteinset{3.85}{-0.95}{\eqref{eq:zigzag}}{%
  \InsertPDF{adj_simeq_full_faithful3.pdf}}}} \raisebox{1em}{,}
  \label{eq:adj_simeq_full_faithful3}
 \end{alignat}
 and thus Eq.~\eqref{eq:adj_simeq_full_faithful1} holds.
 We use the fact that the areas enclosed by auxiliary lines represent $\phi^{-1} \phi = \id_{G \b F}$.
 Therefore, there exists a unique $\ol{f}$ satisfying Eq.~\eqref{eq:adj_simeq_full_faithful1}.

 $(2) \Rightarrow (3)$:
 It suffices to show that for any $c \in \cC$, $\eta_c$ is an isomorphism.
 Let $\braket{F,G,\varphi}$ be an adjunction.
 For each $a \in \cC$, consider the map
 \begin{alignat}{1}
  \rho \coloneqq \varphi_{a,c} \c \hat{F} \colon \cC(a,c) \to \cC(a,G \b Fc),
 \end{alignat}
 where $\hat{F}$ is the map $\cC(a,c) \ni f \mapsto Ff \in \cD(Fa,Fc)$.
 $\rho$ maps each $f \in \cC(a,c)$ to $\eta_c f \in \cC(a,G \b Fc)$ since
 \begin{alignat}{1}
  \InsertMidPDF{adj_simeq_eta_f.pdf}
  &\quad\xmapsto{\hat{F}}\quad
  \InsertMidPDF{adj_simeq_eta_f_F.pdf}
  \quad\xmapsto{\varphi_{a,c}}\quad
  \InsertMidPDF{adj_simeq_eta_f_GF.pdf}
  \label{eq:adj_simeq_eta_f}
 \end{alignat}
 holds.
 $\hat{F}$ is invertible since $F$ is fully faithful, and
 $\varphi_{a,c}$ is invertible since $\braket{F,G,\varphi}$ is an adjunction;
 thus, $\rho$ is also invertible.
 Consider the case of $a = G \b F c$ and consider a morphism $h \in \cC(G \b Fc,c)$
 that satisfies
 \begin{alignat}{1}
  \rho(h) = \id_{G \b Fc}
  &\qquad\diagram\qquad
  \InsertMidPDF{adj_simeq_eta_h.pdf}.
  \label{eq:adj_simeq_eta_h}
 \end{alignat}
 Note that since $\rho$ is invertible, such $h$ is uniquely determined by $h = \rho^{-1}(\id_{G \b Fc})$.
 We show that $h$ is the inverse of $\eta_c$.
 Since the left-hand side, $\rho(h)$, of Eq.~\eqref{eq:adj_simeq_eta_h} is equal to $\eta_c h$,
 $\eta_c h = \id_{G \b F c}$ holds.
 Therefore, it suffices to show that $h \eta_c = \id_c$.
 We have
 \begin{alignat}{1}
  \footnoteinset{-3.19}{0.35}{\eqref{eq:zigzag}}{%
  \footnoteinset{1.18}{0.35}{\eqref{eq:adj_simeq_eta_h}}{%
  \footnoteinset{4.35}{0.35}{\eqref{eq:zigzag}}{%
  \InsertPDF{adj_simeq_eta_hF.pdf}}}} \raisebox{1em}{,}
  \label{eq:adj_simeq_eta_hF}
 \end{alignat}
 i.e., $F \b h \eta_c = F \b \id_c$.
 Therefore, since $F$ is fully faithful, $h \eta_c = \id_c$ holds
 (see Eq.~\eqref{eq:basic_full_faithful}).

 $(3) \Rightarrow (1)$:
 It is clear from $\eta \colon G \b F \ntocong \id_\cC$.
\end{proof}

\subsubsection{Properties of bifunctors}

\begin{proposition}{}{AdjBifunc}
 Let us consider a bifunctor $F \colon \cC \times \cE \to \cD$.
 For each $e \in \cE$, if the functor $F(\Endash,e)$ has a right adjoint $G_e$,
 then there exists a unique bifunctor $G \colon \cE^\op \times \cD \to \cC$ satisfying
 the following conditions:
 \begin{enumerate}
  \item $G(e,\Endash) = G_e$ holds for each $e \in \cE$.
  \item The isomorphism
        \begin{alignat}{1}
         \varphi_{c,e,d} \colon \cD(F(c,e),d) \cong \cC(c,G(e,d)),
        \end{alignat}
        which is natural in $c,d$ and given by the adjunction $F(\Endash,e) \dashv G(e,\Endash)$,
        is also natural in $e$.
 \end{enumerate}
\end{proposition}
\begin{proof}
 Assuming the existence of such a bifunctor $G$, let us consider properties that $G$ must satisfy.
 By Condition~(2), the isomorphism
 $\varphi_{c,e,d} \colon \cD(F(c,e),d) \cong \cC(c,G(e,d))$ is natural in $c$, $e$, and $d$.
 For each $e$, since $G(e,\Endash) = G_e$ is a right adjoint to $F(\Endash,e)$,
 this adjunction has a unit $\eta^e \coloneqq \{ \varphi_{c,e,F(c,e)}(\id_{F(c,e)}) \}_{c \in \cC}$
 and a counit $\varepsilon^e \coloneqq \{ \varphi^{-1}_{G(e,d),e,d}(\id_{G(e,d)}) \}_{d \in \cD}$.
 Since $\varphi_{c,e,d}$ is natural in $e$ and $d$, this naturality is expressed as
 \begin{alignat}{1}
  \InsertPDF{adj_bifunc_proof_nat.pdf}
  \label{eq:adj_bifunc_proof_nat}
 \end{alignat}
 for any $f \in \cE^\op(e',e) = \cE(e,e')$ and $g \in \cD(d,d')$.
 This is equivalent to
 \begin{alignat}{1}
  \InsertPDF{adj_bifunc_proof1.pdf}
  \label{eq:adj_bifunc_proof1}
 \end{alignat}
 for each $h \in \cD(F(c,e'),d)$,
 where the areas enclosed by the auxiliary lines are $\eta^{e'}$ and $\eta^e$
 (see Eq.~\eqref{eq:adj_conj_inv}).
 Substituting $c = G(e',d)$ and
 \begin{alignat}{1}
  \InsertPDF{adj_bifunc_proof_h.pdf}
  \label{eq:adj_bifunc_proof_h}
 \end{alignat}
 (where the area enclosed by the auxiliary line is $\varepsilon^{e'}$) into this equation
 and using Eq.~\eqref{eq:zigzag} yields
 \begin{alignat}{1}
  \InsertPDF{adj_bifunc_proof2.pdf} \raisebox{1em}{.}
  \label{eq:adj_bifunc_proof2}
 \end{alignat}
 This equation uniquely determines $G(f,g)$, so the bifunctor $G$ is unique.
 Conversely, if we define $G(f,g)$ as in Eq.~\eqref{eq:adj_bifunc_proof2},
 then $G(\Endash,\Enndash)$ is a bifunctor satisfying the conditions.
 Indeed, it follows that $G(\Endash,\Enndash)$ is a bifunctor since
 \begin{alignat}{1}
  \footnoteinset{-2.52}{1.2}{\eqref{eq:adj_bifunc_proof2}}{%
  \InsertPDF{adj_bifunc_proof3.pdf}}
  \label{eq:adj_bifunc_proof3}
 \end{alignat}
 holds for each $f \in \cE(e,e')$, $f' \in \cE(e',e'')$, $g \in \cD(d,d')$,
 and $g' \in \cD(d',d'')$, and also $G(\id_e,\id_d) = \id_{G(e,d)}$ holds.
\end{proof}


\section{Limits and colimits} \label{sec:limit}

A (co)limit is a concept that can be used to handle various specific concepts such as
products, equalizers, and pullbacks in a unified manner.
Instead of dealing with these specific (co)limits, this section shows how to represent (co)limits
and explains their basic properties using string diagrams.

\subsection{Definition of limits and colimits} \label{subsec:limit_def}

\subsubsection{Cones}

\begin{define}{cones}{Cone}
 For an object $c \in \cC$ and a functor $D \colon \cJ \to \cC$,
 a natural transformation from $\Delta_\cJ c \colon \cJ \to \cC$ to $D$ is called
 a \termdef{cone} from $c$ to $D$, or simply a cone to $D$.
\end{define}
As shown in Example~\ref{ex:FuncDelta}, the functor $\Delta_\cJ c$ is the same as $c \b {!}$,
where $\Delta_\cJ = \Endash \b {!} \colon \cC \to \Func{\cJ}{\cC}$ is a diagonal functor and 
${!}$ is the only functor from $\cJ$ to $\cOne$.
$\Delta_\cJ c$ maps all objects in $\cJ$ to $c$ and all morphisms in $\cJ$ to $\id_c$.
The collection of all cones from $c$ to $D$ (i.e., $\Func{\cJ}{\cC}(\Delta_\cJ c,D)$) is denoted by
$\Cone(c,D)$.
In this paper, when considering cones, it is often assumed that its domain $\cJ$ is small
(and sometimes this assumption is made implicitly).

A cone $\alpha \coloneqq \{ \alpha_j \}_{j \in \cJ} \in \Cone(c,D)$ is represented by
\begin{alignat}{1}
 \footnoteinset{-3.59}{0.35}{\eqref{eq:basic_functor_delta}}{%
 \footnoteinset{3.02}{0.35}{\eqref{eq:basic_functor_terminal}}{%
 \InsertPDF{limit_cone.pdf}}} \raisebox{1em}{,}
 \label{eq:limit_cone}
\end{alignat}
where the functor ${!}$ is represented by a gray dotted wire.
Note that we often represent a cone, which is a natural transformation, using a trapezoidal block.
Naturality of $\alpha$ says that for any morphism $h \in \cJ(i,j)$ in $\cJ$,
\begin{alignat}{1}
 \InsertPDF{limit_nat.pdf}
 \label{eq:limit_nat}
\end{alignat}
holds.

The comma category $\Cone_D \coloneqq \Delta_\cJ \comma D$ is often called
the \termdef{category of cones} to $D$.
Let us specify $\Cone_D$ (recall Subsubsection~\ref{subsubsec:repr_cG_Gc}):

\begin{itemize}
 \item Its each object is a pair, $\braket{c,\alpha}$, of $c \in \cC$ and $\alpha \in \Cone(c,D)$.
 \item Its each morphism with domain $\braket{c,\alpha}$ and codomain $\braket{c',\alpha'}$
       is a morphism $f \in \cC(c,c')$ in $\cC$ such that
       \begin{alignat}{1}
        \InsertMidPDF{limit_cone_D1.pdf}
        \qquad\myLeftrightarrow{equivalent}\qquad
        \InsertMidPDF{limit_cone_D2.pdf}.
        \label{eq:limit_cone_D}
       \end{alignat}
 \item The composite of its morphisms is the composite of morphisms in $\cC$,
       and its identity morphism is the identity morphism in $\cC$.
\end{itemize}
Roughly speaking, $\Cone_D$ is a category whose objects are cones to $D$.
Since
\begin{alignat}{1}
 \Cone_D &= \Delta_\cJ \comma D = \el(\Func{\cJ}{\cC}(\Delta_\cJ\Endash,D))
 = \el(\Cone(\Endash,D)),
\end{alignat}
$\Cone_D$ is equal to the category of elements of $\Cone(\Endash,D)$.

\subsubsection{Limits}

\begin{define}{limits}{Limit}
 A universal morphism from the diagonal functor $\Delta_\cJ \colon \cC \to \Func{\cJ}{\cC}$
 to a functor $D \in \Func{\cJ}{\cC}$ is called a \termdef{limit} of $D$.
 Rewriting it using Eq.~\eqref{eq:univ_op}, $\braket{d,\kappa}$ is called a limit of $D$
 if, for any cone $\alpha \in \Cone(c,D)$ (where $c$ is arbitrary),
 there exists a unique morphism $\ol{\alpha} \in \cC(c,d)$ satisfying
 \begin{alignat}{1}
  \InsertMidPDF{limit_repr_terminal.pdf}
  \qquad\myLeftrightarrow{equivalent}\qquad
  \InsertMidPDF{limit_repr_terminal2.pdf},
  \label{eq:limit_repr_terminal}
 \end{alignat}
 where the blue circles denote $\kappa$.
\end{define}
In another expression, a limit of $D$ is a terminal object in $\Cone_D$.

Applying the evaluation functor $\ev_j$ to Eq.~\eqref{eq:limit_repr_terminal} yields
each component $\alpha_j$ of a cone $\alpha$,
so Eq.~\eqref{eq:limit_repr_terminal} is equivalent to the existence of a unique
$\ol{\alpha} \in \cC(c,d)$ satisfying
\begin{alignat}{2}
 \InsertMidPDF{limit_repr_leg.pdf}
 &\qquad(\forall j \in \cJ),
 \label{eq:limit_repr_leg}
\end{alignat}
where the black circle is the component $\kappa_j$ of $\kappa$.
Note that one can differentiate between the cone $\kappa$ and its component $\kappa_j$
in a string diagram by examining the wires connected to it.
Equation~\eqref{eq:limit_repr_terminal} can be regarded as an equation for obtaining
the corresponding $\alpha$ from $\ol{\alpha}$.
Conversely, to obtain the corresponding $\ol{\alpha}$ from $\alpha$,
we can use a diagram introduced in Subsubsection~\ref{subsubsec:repr_repr_diagram}
and express it as follows (where we show two expressions as in Eq.~\eqref{eq:limit_repr_terminal}):
\begin{alignat}{1}
 \InsertPDF{limit_repr_terminal_repr.pdf} \raisebox{1em}{.}
 \label{eq:limit_repr_terminal_repr}
\end{alignat}

When $\braket{d,\kappa}$ is a limit of $D$, we call $d$ its \termdef{limit object}
and $\kappa$ its \termdef{limit cone}.
We often write $d$ as $\lim D$.
It should be understood that $\lim D$ represents one of the limit objects of $D$.
In what follows, we sometimes represent $\kappa$ as the right-hand side of
\begin{alignat}{1}
 \InsertPDF{limit_repr_u.pdf} \raisebox{1em}{,}
 \label{eq:limit_repr_u}
\end{alignat}
where the dotted wire labeled ``$\lim$'' represents something like a map that maps
a functor $D \in \Func{\cJ}{\cC}$ to its limit object $\lim D \in \cC$.
However, it should be noted that if $D$ does not have a limit,
then $\lim D$ should be considered undefined;
in this sense, this dotted wire is not strictly regarded as a map.

\subsubsection{Cocones and colimits} \label{subsubsec:limit_def_colimit}

The duals of cones and limits are called cocones and colimits, respectively.
Specifically, for an object $c \in \cC$ and a functor $D \colon \cJ \to \cC$,
a natural transformation from $D$ to $\Delta_\cJ c$ is called
a \termdef{cocone} from $D$ to $c$ or simply a cocone from $D$.
We write $\Cocone(D,c)$ for the collection of all cocones from $D$ to $c$
(i.e., $\Func{\cJ}{\cC}(D,\Delta_\cJ c)$).
The comma category $\Cocone_D \coloneqq D \comma \Delta_\cJ$ is called the \termdef{category of cocones}
from $D$.
Also, a universal morphism from $D$ to $\Delta_\cJ$
(i.e., an initial object in $\Cocone_D$) is called a \termdef{colimit} of $D$.

A cocone $\alpha \coloneqq \{ \alpha_j \}_{j \in \cJ} \in \Cocone(D,c)$ is represented
by the following diagram as the dual of Eq.~\eqref{eq:limit_cone}:
\begin{alignat}{1}
 \footnoteinset{-3.59}{0.35}{\eqref{eq:basic_functor_delta}}{%
 \footnoteinset{3.02}{0.35}{\eqref{eq:basic_functor_terminal}}{%
 \InsertPDF{limit_cone_op.pdf}}} \raisebox{1em}{.}
 \label{eq:limit_cone_op}
\end{alignat}
When $D$ has a colimit $\braket{d,\kappa}$, for any cocone $\alpha \in \Cocone(D,c)$,
there exists a unique morphism $\ol{\alpha} \in \cC(d,c)$ satisfying
\begin{alignat}{1}
 \InsertMidPDF{limit_repr_terminal_op.pdf}
 \qquad\myLeftrightarrow{equivalent}\qquad
 \InsertMidPDF{limit_repr_terminal_op2.pdf},
 \nonumber \\
 \label{eq:limit_repr_terminal_op}
\end{alignat}
where the blue circles are $\kappa$.
Since each component $\alpha_j$ is obtained by applying the evaluation functor $\ev_j$ to $\alpha$,
Eq.~\eqref{eq:limit_repr_terminal_op} is equivalent to
the existence of a unique $\ol{\alpha} \in \cC(d,c)$ satisfying
\begin{alignat}{2}
 \InsertMidPDF{limit_repr_leg_op.pdf}
 &\qquad(\forall j \in \cJ),
 \label{eq:limit_repr_leg_op}
\end{alignat}
where the black circle denotes $\kappa_j$.
We also have a similar equation to Eq.~\eqref{eq:limit_repr_terminal_repr}.
When $\braket{d,\kappa}$ is a colimit of $D$, we call $d$ its \termdef{colimit object}
and $\kappa$ its \termdef{colimit cocone}.
$d$ is often written as $\colim D$.

\subsubsection{Properties immediately derived from definitions}

Since initial and terminal objects are essentially unique,
(co)limits are essentially unique if they exist.
Let us specifically state this fact for limits.
When $\braket{d,\kappa}$ is a limit of $D$, for any limit $\braket{d',\kappa'}$ of $D$,
there exists a unique isomorphism $\psi \in \cC(d',d)$ such that
$\kappa' = \kappa \c \Delta_\cJ \psi$, i.e.,
\begin{alignat}{1}
 \InsertMidPDF{limit_cong.pdf}
 \qquad\myLeftrightarrow{equivalent}\qquad
 \InsertMidPDF{limit_cong2.pdf},
 \label{eq:limit_cong}
\end{alignat}
where the blue circles represent $\kappa$ and the diamond block represents $\psi$
(see Eq.~\eqref{eq:repr_cG_init_cong}).
Conversely, any $\braket{d',\kappa'}$ represented in the form of Eq.~\eqref{eq:limit_cong} is
a limit of $D$.
The same applies to colimits.

From the discussion on representability in Subsubsection~\ref{subsubsec:repr_repr_repr},
we obtain the following lemma.
\begin{lemma}{}{LimitRepr}
 For any small category $\cJ$, a functor $D \colon \cJ \to \cC$ has a limit if and only if
 the presheaf $\Cone(\Endash,D) = \Func{\cJ}{\cC}(\Delta_\cJ \Endash, D)
 \colon \cC^\op \to \Set$ is representable, i.e., there exists $d \in \cC$ such that
 \begin{alignat}{1}
  \cC(\Endash,d) \cong \Cone(\Endash,D)
  &\qquad\diagram\qquad
  \InsertMidPDF{limit_repr_cong.pdf}.
  \label{eq:limit_repr_cong}
 \end{alignat}
 In this case, $d \cong \lim D$ holds.

 Dually, $D$ has a colimit if and only if
 the set-valued functor $\Cocone(D,\Endash) \coloneqq \Func{\cJ}{\cC}(D, \Delta_\cJ \Endash)
 \colon \cC \to \Set$ is representable, i.e., there exists $d \in \cC$ such that
 \begin{alignat}{1}
  \cC(d,\Endash) \cong \Cocone(D,\Endash)
  &\qquad\diagram\qquad
  \InsertMidPDF{limit_repr_cong_op.pdf}.
  \label{eq:limit_repr_cong_op}
 \end{alignat}
 In this case, $d \cong \colim D$ holds.
\end{lemma}
\begin{proof}
 $D$ has a limit if and only if the category of cones $\Cone_D = \Delta_\cJ \comma D$ has
 a terminal object.
 From Eq.~\eqref{eq:adj_CG_repr_op}, this is equivalent to $\Cone(\Endash,D)$ being representable.
 In particular, $d$ is a limit object of $D$ if and only if $\Cone_D$ has
 a terminal object of the form $\braket{d,\kappa}$, which is equivalent to
 $\Cone(\Endash,D) \cong \yonedaop{d} = \cC(\Endash,d)$.
\end{proof}

\begin{lemma}{}{LimitCone1cD}
 For any functor $D \colon \cJ \to \cC$, the following equation holds%
 \footnote{$\Cocone(\cC(D \Endash,c), \{*\})$
 is the collection of all cocones from $\cC(D \Endash,c) \colon \cJ \to \Set^\op$ to $\{*\} \in \Set^\op$.}:
 \begin{alignat}{1}
  \Cone(c,D) = \Cone(\{*\}, \cC(c,D \Endash)), \qquad
  \Cocone(D,c) = \Cocone(\cC(D \Endash,c), \{*\}).
  \label{eq:limit_Cone_1cD}
 \end{alignat}
\end{lemma}
\begin{proof}
 We show $\Cone(c,D) = \Cone(\{*\}, \cC(c,D \Endash))$; the other formula is dual.
 We have for any $\alpha \in \Cone(c,D)$,
 \begin{alignat}{1}
  \footnoteinset{-2.26}{-0.55}{\eqref{eq:repr_yoneda_init_summary}}{%
  \InsertPDF{limit_cone_1cD.pdf}} \raisebox{1em}{.}
  \label{eq:limit_cone_1cD}
 \end{alignat}
 Note that it follows from Eq.~\eqref{eq:repr_yoneda_init_summary} that the last equality holds,
 i.e., $\alpha$ is a cone from $\{*\}$ to $\yoneda{c} \b D = \cC(c,D \Endash)$.
 Therefore, we have $\Cone(c,D) \subseteq \Cone(\{*\}, \cC(c,D \Endash))$.
 Also, by tracing this equation backwards we obtain
 $\Cone(c,D) \supseteq \Cone(\{*\}, \cC(c,D \Endash))$.
\end{proof}
Note that, as a corresponding equation to Eq.~\eqref{eq:repr_yoneda_init_summary},
any $\alpha \in \Cone(c,D)$ can be represented as
\begin{alignat}{1}
 \InsertPDF{limit_cone_1cD2.pdf} \raisebox{1em}{.}
 \label{eq:limit_cone_1cD2}
\end{alignat}

\subsection{Basic properties of limits and colimits} \label{subsec:limit_property}

\subsubsection{Limits as right adjoints to diagonal functors} \label{subsubsec:limit_def_adj}

From the properties of limits and adjoints, the following proposition holds.
\begin{proposition}{}{LimitAdj}
 For a small category $\cJ$ and a category $\cC$, the following are equivalent:
 \begin{enumerate}
  \item Any functor from $\cJ$ to $\cC$ has a limit.
  \item The diagonal functor $\Delta_\cJ \colon \cC \to \Func{\cJ}{\cC}$ has a right adjoint.
 \end{enumerate}
 Furthermore, if they hold, then there exists a functor $\lim \colon \Func{\cJ}{\cC} \to \cC$
 that maps each functor $D \colon \cJ \to \cC$ to one of its limit objects,
 and $\lim$ is a right adjoint to the diagonal functor $\Delta_\cJ$.
 Moreover, for the counit $\varepsilon$ of the adjunction $\Delta_\cJ \dashv \lim$,
 $\varepsilon_D$ is a limit cone of $D$.
\end{proposition}
\begin{proof}
 $(1) \Rightarrow (2)$:
 Since each $D \colon \cJ \to \cC$ has a limit, i.e.,
 a universal morphism from $\Delta_\cJ$ to each $D$ exists,
 from the dual of Theorem~\ref{thm:AdjNasUniv}, $\Delta_\cJ$ has a right adjoint.

 $(2) \Rightarrow (1)$:
 Let $\lim$ be a right adjoint to $\Delta_\cJ$.
 Also, arbitrarily choose $D \colon \cJ \to \cC$.
 Then, from Lemma~\ref{lemma:AdjUniv},
 $\braket{\lim D, \varepsilon_D}$ is a universal morphism
 from $\Delta_\cJ$ to $D$, i.e., a limit of $D$.
\end{proof}

Dually, any functor from $\cJ$ to $\cC$ has a colimit if and only if
$\Delta_\cJ \colon \cC \to \Func{\cJ}{\cC}$ has a left adjoint.

Let us specify such a functor $\lim$.
For each $D \colon \cJ \to \cC$, choose one of its limits and denote it by
$\braket{\lim D, \varepsilon_D}$.
Also, for each natural transformation $\tau \in \Func{\cJ}{\cC}(D,D')$,
let $\lim \tau$ be $\ol{\tau} \in \cC(\lim D, \lim D')$ satisfying
\begin{alignat}{1}
 \InsertPDF{limit_adj_mor.pdf} \raisebox{1em}{,}
 \label{eq:limit_adj_mor}
\end{alignat}
where the blue circles are $\varepsilon_D$ and $\varepsilon_{D'}$.
Then, for each $\tau \in \Func{\cJ}{\cC}(D,D')$ and $\tau' \in \Func{\cJ}{\cC}(D',D'')$,
$\lim \tau'\tau = (\lim \tau')(\lim \tau)$ holds from
\begin{alignat}{1}
 \InsertPDF{limit_adj_mor_func.pdf}
 \label{eq:limit_adj_mor_func}
\end{alignat}
and the dual of Eq.~\eqref{eq:repr_univ_eq}.
Also, $\lim$ maps the identity morphism $\id_D$ to the identity morphism $\id_{\lim D}$,
and thus $\lim$ is a functor.
When a functor $\lim$ exists, using the expression of counits introduced
in Eq.~\eqref{eq:adj_eta_epsilon}, the limit cone $\varepsilon_D$ can be represented by
\begin{alignat}{1}
 \InsertPDF{limit_repr_u_adj.pdf} \raisebox{1em}{.}
 \label{eq:limit_repr_u_adj}
\end{alignat}
The counit $\varepsilon$ of this adjunction is the collection
$\{ \varepsilon_D \}_{D \in \Func{\cJ}{\cC}}$.
Equations~\eqref{eq:limit_repr_terminal} and \eqref{eq:limit_repr_terminal_repr} can be rewritten
as follows:
\begin{alignat}{1}
 \InsertPDF{limit_repr_adj.pdf} \raisebox{1em}{.}
 \label{eq:limit_repr_adj}
\end{alignat}

A category $\cC$ is called \termdef{complete} if any functor $D \colon \cJ \to \cC$
with any small category $\cJ$ has a limit and \termdef{cocomplete}
if any functor $D \colon \cJ \to \cC$ with any small category $\cJ$ has a colimit.

\subsubsection{Preservation of limits and colimits} \label{subsubsec:limit_property_preserve}

Let us consider two functors $G \colon \cD \to \cC$ and $D \colon \cJ \to \cD$
with a small category $\cJ$.
We say that $G$ \termdef{preserves} the limits of $D$ if
when $D$ has a limit $\braket{d,\kappa}$, $\braket{Gd, G \b \kappa}$ is a limit of $G \b D$,
i.e.,
\begin{alignat}{1}
 \InsertMidPDF{limit_preserve_def1.pdf} \text{is a limit cone of $D$}
 &\qquad\Rightarrow\qquad
 \InsertMidPDF{limit_preserve_def2.pdf} \text{is a limit cone of $G \b D$}
 \label{eq:limit_preserve_def}
\end{alignat}
holds.
Also, we call $G$ preserving any limit (or continuous) if it preserves
the limits of any functor $D \colon \cJ \to \cD$ with any small category $\cJ$.
The preservation of colimits is defined as their dual.

$G$ preserves the limits of $D$ if and only if
\begin{alignat}{1}
 \lefteqn{ \InsertMidPDF{limit_repr_terminal2.pdf} \quad (\forall c \in \cC, ~\alpha \in \Cone(c,D)) }
 \nonumber \\
 &\Rightarrow\qquad
 \InsertMidPDF{limit_preserve.pdf} \quad (\forall c \in \cC, ~\beta \in \Cone(c, G \b D))
 \label{eq:limit_preserve}
\end{alignat}
holds, where the blue circle is a cone, denoted by $\kappa$, to $D$.
In this case, $\braket{d,\kappa}$ is a limit of $D$.
Intuitively, it can be said that the equation will still preserve its structure
even if we insert the wire $G$ on the left side of the wire $D$.
Also, the relationship between limit cones of $D$ and $G \b D$ is represented by
\begin{alignat}{1}
 \InsertPDF{limit_preserve_terminal.pdf} \raisebox{1em}{,}
 \label{eq:limit_preserve_terminal}
\end{alignat}
where the blue circle represents $\kappa$ and the blue ellipse represents any limit cone of $G \b D$.
The diamond-shaped block represents an isomorphism, which means $G \b \lim D \cong \lim (G \b D)$.
Equation~\eqref{eq:limit_preserve_terminal} intuitively means that
connecting the wire $G$ to the blue circle can be identified with the blue ellipse.

\begin{thm}{Right adjoints preserve limits (RAPL)}{LimitAdjPreserve}
 Any right adjoint preserves any limit.
 Dually, any left adjoint preserves any colimit.
\end{thm}
\begin{proof}
 Assume that there exists a limit, denoted by $\braket{d,\kappa}$,
 of a functor $D \colon \cJ \to \cD$, where $\cJ$ is a small category.
 Also, arbitrarily choose $F \dashv G \colon \cC \to \cD$.
 We have for any cone $\alpha \in \Cone(c,G \b D)$ (where $c$ is arbitrary),
 \begin{alignat}{1}
  \footnoteinset{-3.45}{0.3}{\eqref{eq:adj_conj}}{%
  \footnoteinset{0.12}{0.3}{\eqref{eq:limit_repr_terminal}}{%
  \footnoteinset{3.42}{0.3}{\eqref{eq:adj_conj}}{%
  \InsertPDF{limit_preserve_adj.pdf}}}} \raisebox{1em}{,}
  \label{eq:limit_preserve_adj}
 \end{alignat}
 where the blue circle is $\kappa$.
 In the first equality, we set the transpose of $\alpha$ to $\ol{\alpha}$
 (note that $F \b \Endash \dashv G \b \Endash$ holds).
 The second equality follows from applying Eq.~\eqref{eq:limit_repr_terminal} to the cone $\ol{\alpha}$
 (we set $\ol{\alpha}$ on the right-hand side of Eq.~\eqref{eq:limit_repr_terminal} to $\ol{\beta}$).
 In the last equality, we set the transpose of $\ol{\beta}$ to $\beta$.
 Since $\beta$ is uniquely determined from $\alpha$, $\braket{Gd,G \b \kappa}$ is a limit of $G \b D$.
 Therefore, $G$ preserves any limit.
\end{proof}

\begin{thm}{}{LimitPresheaf}
 For each object $c \in \cC$, $\yoneda{c} \colon \cC \to \Set$ preserves any limit.
\end{thm}
\begin{proof}
 Assume that there exists a limit, denoted by $\braket{d,\kappa}$,
 of a functor $D \colon \cJ \to \cD$, where $\cJ$ is a small category.
 Then, it suffices to show that $\braket{\yoneda{c} \b d,\yoneda{c} \b \kappa}$ is a limit of
 $\yoneda{c} \b D = \cC(c,D \Endash)$.
 Arbitrarily choose a set $X$ and a cone $\alpha \in \Cone(X,\yoneda{c} \b D)$,
 and let $\alpha(x) \coloneqq \{ \alpha_j(x) \}_{j \in \cJ}$ for each $x \in X$.
 Also, let $\Gamma$ be the map that sends a collection of the form
 $\{ \{ f_{x,y} \}_{x \in X} \}_{y \in Y}$
 to the collection $\{ \{ f_{x,y} \}_{y \in Y} \}_{x \in X}$
 (by swapping the order of $\{ \Endash \}_{x \in X}$ and $\{ \Endash \}_{y \in Y}$).
 $\Gamma$ is obviously invertible and satisfies
 \begin{alignat}{1}
  \lefteqn{ \alpha = \{ \{ \alpha_j(x) \}_{x \in X} \}_{j \in \cJ}
  \quad\xmapsto{\Gamma}\quad
  \{ \alpha(x) \}_{x \in X} = \{ \{ \alpha_j(x) \}_{j \in \cJ} \}_{x \in X} } \nonumber \\
  &\diagram\qquad
  \InsertMidPDF{report_limit_preserve_yoneda_swap.pdf}.
  \label{eq:limit_preserve_yoneda_swap}
 \end{alignat}
 Note that since, for $F \coloneqq \yoneda{c} \b D \colon \cJ \to \Set$,
 \begin{alignat}{1}
  \alpha \in \Cone(X,F)
  &\quad\Leftrightarrow\quad
  \alpha_{\cod h} = Fh \c \alpha_{\dom h} \quad(\forall h \in \mor \cJ) \nonumber \\
  &\quad\Leftrightarrow\quad
  \alpha_{\cod h}(x) = Fh \c \alpha_{\dom h}(x) \quad(\forall h \in \mor \cJ,~ x\in X)
 \end{alignat}
 holds, each $\alpha(x)$ is a cone from $\{*\}$ to $F$.
 It follows that $\ol{\alpha} \coloneqq \{ \ol{\alpha(x)} \}_{x \in X}$ satisfying
 \begin{alignat}{1}
  \footnoteinset{1.13}{1.30}{\eqref{eq:limit_cone_1cD2}}{%
  \footnoteinset{-4.29}{-0.55}{\eqref{eq:limit_repr_terminal}}{%
  \footnoteinset{-0.72}{-0.55}{\eqref{eq:limit_cone_1cD2}}{%
  \InsertPDF{report_limit_preserve_yoneda.pdf}}}}
  \label{eq:limit_preserve_yoneda}
 \end{alignat}
 (where the blue circles are $\kappa$) is uniquely determined.
 Therefore, $\braket{\yoneda{c} \b d,\yoneda{c} \b \kappa}$ is a limit of
 $\yoneda{c} \b D$.
\end{proof}

\begin{cor}{}{LimitPresheaf2}
 $\yonedaop{c} \colon \cC^\op \to \Set$ preserves any limit.
 Alternatively, $\yonedaop{c} \colon \cC \to \Set^\op$ preserves any colimit.
\end{cor}
\begin{proof}
 Since $\yonedaop{c} = \cC(\Endash,c) = \cC^\op(c,\Endash)$,
 it is clear by substituting $\cC^\op$ for $\cC$ in Theorem~\ref{thm:LimitPresheaf}.
\end{proof}

\begin{proposition}{}{LimitPreserveCong}
 Consider a small category $\cJ$, categories $\cC$ and $\cD$, and a functor $G \colon \cD \to \cC$.
 If any functor from $\cJ$ to $\cC$ and any functor from $\cJ$ to $\cD$ have limits,
 and $G$ preserves any limit, then the isomorphism $G \b \lim D \cong \lim (G \b D)$
 is natural in $D \colon \cJ \to \cD$.
\end{proposition}
\begin{proof}
 By Proposition~\ref{pro:LimitAdj}, a right adjoint, $\lim$, to
 $\Delta_\cJ \colon \cC \to \Func{\cJ}{\cC}$ and a right adjoint, $\lim$, to
 $\Delta_\cJ \colon \cD \to \Func{\cJ}{\cD}$ exist.
 Also, since $G$ preserves any limit, we have for each $D \colon \cJ \to \cD$
 (see Eqs.~\eqref{eq:limit_repr_u_adj} and \eqref{eq:limit_preserve_terminal}),
 \begin{alignat}{1}
  \InsertPDF{limit_preserve_cong_phi.pdf} \raisebox{1em}{,}
  \label{eq:limit_preserve_cong_phi}
 \end{alignat}
 where the diamond block represents the isomorphism from $G \b \lim D$ to $\lim (G \b D)$,
 which we denote as $\varphi_D$.
 Note that the functor $G \b \Delta_\cJ \colon \cD \to \Func{\cJ}{\cC}$ and
 the functor $\Delta_\cJ \b G \colon \cD \to \Func{\cJ}{\cC}$ are equal
 since both have the action on morphisms
 $\mor \cD \ni f \mapsto \{ Gf \}_{j \in \cJ} \in \mor \Func{\cJ}{\cC}$.
 If $\varphi \coloneqq \{ \varphi_D \}_{D \in \Func{\cJ}{\cD}}$ is a natural transformation,
 then $\varphi$ is a natural isomorphism from $G \b \lim$ to $\lim \b (G \b \Endash)$.
 Thus, it remains to show naturality of $\varphi$.
 We have for any $\tau \colon D \nto D'$ (where $D,D' \colon \cJ \to \cD$ are also arbitrary),
 \begin{alignat}{1}
  \footnoteinset{-1.60}{1.55}{\eqref{eq:sliding}}{%
  \footnoteinset{1.58}{1.55}{\eqref{eq:limit_preserve_cong_phi}}{%
  \footnoteinset{-1.60}{-0.9}{\eqref{eq:sliding}}{%
  \footnoteinset{1.58}{-0.9}{\eqref{eq:limit_preserve_cong_phi}}{%
  \InsertPDF{limit_preserve_cong_phi_nat.pdf}}}}} \raisebox{1em}{.}
  \label{eq:limit_preserve_cong_phi_nat}
 \end{alignat}
 Therefore, due to the universality of $\lim (G \b D')$, we have
 \begin{alignat}{1}
  \InsertPDF{limit_preserve_cong_phi_nat2.pdf} \raisebox{1em}{,}
  \label{eq:limit_preserve_cong_phi_nat2}
 \end{alignat}
 and thus $\varphi$ is a natural transformation.
\end{proof}

\begin{proposition}{}{LimitCong}
 If a functor $D \colon \cJ \to \cC$ with a small category $\cJ$ has a limit, then
 there exists an isomorphism
 \begin{alignat}{1}
  \Cone(c,D) \cong \cC(c,\lim D) \cong \lim \cC(c,D \Endash)
  \label{eq:limit_cong_limit}
 \end{alignat}
 that is natural in $c \in \cC$.
 Dually, if $D$ has a colimit, then there exists an isomorphism
 \begin{alignat}{1}
  \Cocone(D,c) \cong \cC(\colim D,c) \cong \lim \cC(D \Endash,c)
  \label{eq:limit_cong_colimit}
 \end{alignat}
 that is natural in $c \in \cC$.
\end{proposition}
\begin{proof}
 Since $\Cone(\Endash,D) \cong \cC(\Endash,\lim D)$ holds from Eq.~\eqref{eq:limit_repr_cong},
 the left isomorphism in Eq.~\eqref{eq:limit_cong_limit} holds and is natural in $c \in \cC$.
 Also, since from Theorem~\ref{thm:LimitPresheaf} $\cC(c,\lim D)$ is a limit object of
 $\yoneda{c} \b D = \cC(c,D\Endash)$, the right isomorphism in Eq.~\eqref{eq:limit_cong_limit} holds.
 We show that the right isomorphism is natural in $c \in \cC$.

 The isomorphism $\lim \cC(c,D \Endash) \cong \cC(c,\lim D)$ is depicted by
 \begin{alignat}{1}
  \footnoteinset{0.00}{0.3}{\eqref{eq:limit_preserve_terminal}}{%
  \InsertPDF{limit_cong_preserve.pdf}} \raisebox{1em}{,}
  \label{eq:limit_cong_preserve}
 \end{alignat}
 where the blue ellipse and the blue circle represent limit cones of $\yoneda{c} \b D$ and $D$,
 respectively.
 The area enclosed by the auxiliary line represents the isomorphism, which we denote as $\psi_c$,
 from $\lim \cC(c,D\Endash)$ to $\cC(c,\lim D)$.
 Also, the solid wire $\lim$ represents the functor $\lim \colon \Func{\cJ}{\Set} \to \Set$,
 which exists since $\Set$ is complete.
 With no loss of generality we can assume that this functor $\lim$ is defined
 for each $f \in \cC(a,b)$ to satisfy
 \begin{alignat}{1}
  \InsertPDF{limit_cong_nat_proof2.pdf}
  \label{eq:limit_cong_nat_proof2}
 \end{alignat}
 (see Eq.~\eqref{eq:limit_adj_mor}).
 The isomorphism $\psi_c \colon \lim \cC(c,D \Endash) \cong \cC(c,\lim D)$ is natural
 in $c \in \cC$%
 \footnote{This means that $\{ \psi_c \}_{c \in \cC}$ is a natural isomorphism
 from the presheaf $\lim_{j \in \cJ} \cC(\Endash,Dj) \in \hat{\cC}$
 to the presheaf $\cC(\Endash,\lim D) \in \hat{\cC}$.}
 if and only if
 \begin{alignat}{1}
  \yoneda{f}{}_{\lim D} \c \psi_b = \psi_a \c \lim(\yoneda{f} \b D)
  &\qquad\diagram\qquad
  \InsertMidPDF{limit_cong_nat.pdf}
  \label{eq:limit_cong_nat}
 \end{alignat}
 holds for each $f \in \cC(a,b)$ (where $a,b \in \cC$ are arbitrary).
 Therefore, it suffices to show Eq.~\eqref{eq:limit_cong_nat}.
 From Eqs.~\eqref{eq:limit_cong_preserve} and \eqref{eq:limit_cong_nat_proof2}, we have
 \begin{alignat}{1}
  \footnoteinset{-1.66}{1.75}{\eqref{eq:sliding}}{%
  \footnoteinset{1.65}{1.75}{\eqref{eq:limit_cong_preserve}}{%
  \footnoteinset{-1.66}{-1.1}{\eqref{eq:limit_cong_nat_proof2}}{%
  \footnoteinset{1.65}{-1.1}{\eqref{eq:limit_cong_preserve}}{%
  \InsertPDF{limit_cong_nat_proof.pdf}}}}} \raisebox{1em}{.}
  \label{eq:limit_cong_nat_proof}
 \end{alignat}
 Let the blue circle be denoted as $\kappa$; then, since $\yoneda{a} \b \kappa$ is
 a limit cone of $\yoneda{a} \b D$, the two natural transformations
 enclosed by the auxiliary lines are equal due to the universality of $\yoneda{a} \b \kappa$.
 Therefore, Eq.~\eqref{eq:limit_cong_nat} holds.
\end{proof}

\subsubsection{Limits and colimits of functors to functor categories} \label{subsubsec:limit_property_func}

The following proposition asserts that a limit of a functor from a small category $\cJ$
to a functor category $\Func{\cI}{\cC}$ can be obtained from limits of
the functors $\ev_i \b D$ $~(i \in \cI)$.
\begin{proposition}{}{LimitBifunc}
 Consider a functor $D \colon \cJ \to \Func{\cI}{\cC}$ with a small category $\cJ$.
 Assume that, for each $i \in \cI$, $D_i \coloneqq \ev_i \b D = D(\Endash)(i)$ has
 a limit $\braket{l_i,\kappa_i}$.
 Let us express $\kappa_i$ as $\{ \kappa_{i,j} \}_{j \in \cJ}$
 and let $\kappa \coloneqq \{ \{ \kappa_{i,j} \}_{i \in \cI} \}_{j \in \cJ}$.
 Then, there exists a unique functor $L \in \Func{\cI}{\cC}$ such that
 $L(i) = l_i$ holds and that $\kappa$ is a cone from $L$ to $D$.
 Also, $\braket{L,\kappa}$ is a limit of $D$.
 Dually, the same holds for colimits.
\end{proposition}
\begin{proof}
 Let $P$ be the canonical functor from $\Functwo{\cJ}{\cI}{\cC}$
 to $\Functwo{\cI}{\cJ}{\cC}$%
 \footnote{$P$ can be defined as the map of
 each morphism $\{ \{ \alpha_{i,j} \}_{i \in \cI} \}_{j \in \cJ}$ in $\Functwo{\cJ}{\cI}{\cC}$
 to $\{ \{ \alpha_{i,j} \}_{j \in \cJ} \}_{i \in \cI}$.
 The canonical functor from $\Func{\cI \times \cJ}{\cC}$ to $\Functwo{\cJ}{\cI}{\cC}$
 can be defined in a similar manner.},
 and let $D^* \coloneqq P \b D$; then, we have $D^*(i) = D_i$.
 Note that the two functors $D \colon \cJ \to \Func{\cI}{\cC}$ and
 $D^* \colon \cI \to \Func{\cJ}{\cC}$ are related as if the order of $\cI$ and $\cJ$
 were swapped.
 We prove this proposition by replacing the discussion about $D$ with the discussion about $D^*$.

 Let $\kappa^* \coloneqq \{ \kappa_i \}_{i \in \cI}$;
 then, $\kappa = \{ \{ \kappa_{i,j} \}_{i \in \cI} \}_{j \in \cJ}$
 is a natural transformation from $\Delta_\cJ L \colon \cJ \to \Func{\cI}{\cC}$ to $D$
 (i.e., a cone from $L$ to $D$) if and only if
 $\kappa^* = \{ \{ \kappa_{i,j} \}_{j \in \cJ} \}_{i \in \cI}$ is
 a natural transformation from $\Delta_\cJ \b L \colon \cI \to \Func{\cJ}{\cC}$ to $D^*$.
 Indeed, if $\kappa^*$ is a natural transformation, then
 \begin{alignat}{1}
  P \b \kappa = \kappa^*
  &\qquad\diagram\qquad
  \InsertMidPDF{limit_bifunc_P.pdf}
  \label{eq:limit_bifunc_P}
 \end{alignat}
 holds, and thus $\kappa$ is a natural transformation.
 The converse also holds since $P^{-1} \b \kappa^* = \kappa$.
 
 It follows from Lemma~\ref{lemma:ReprIsoUniv} that $\braket{L,\kappa}$ is
 a universal morphism from $\Delta_\cJ$ to $D$ if and only if
 $\braket{L,\kappa^*}$ is a universal morphism from $\Delta_\cJ \b \Endash$ to $D^*$.
 Thus, it is sufficient to show that a functor $L$ satisfying the conditions is uniquely determined
 and that $\braket{L,\kappa^*}$ is a universal morphism from $\Delta_\cJ \b \Endash$
 to $D^*$.
 This can be shown by substituting $D^* \colon \cI \to \Func{\cJ}{\cC}$
 and $\Delta_\cJ \colon \cC \to \Func{\cJ}{\cC}$ for $H$ and $G$, respectively,
 in the dual of Lemma~\ref{lemma:ReprFuncH}.
 Indeed, if we let $L$ be $F$ of this lemma, then it follows that there exists
 a unique functor $L \in \Func{\cI}{\cC}$ such that
 $L(i) = l_i$ holds and that $\kappa^*$ is a natural transformation
 from $\Delta_\cJ \b L$ to $D^*$, i.e.,
 \begin{alignat}{1}
  \InsertPDF{limit_bifunc_limit.pdf} \raisebox{1em}{.}
  \label{eq:limit_bifunc_limit}
 \end{alignat}
 Also, this lemma says that $\braket{L,\kappa^*}$ is a universal morphism
 from $\Delta_\cJ \b \Endash$ to $D^*$.
\end{proof}

\begin{cor}{}{LimitBifuncFunc}
 If any functor from a small category $\cJ$ to a category $\cC$ has a limit,
 then any functor from $\cJ$ to $\Func{\cI}{\cC}$ also has a limit.
 The same holds for colimits.
\end{cor}
\begin{proof}
 From Proposition~\ref{pro:LimitBifunc}, for any functor $D \colon \cJ \to \Func{\cI}{\cC}$,
 since $D_i \coloneqq \ev_i \b D$ has a limit, so does $D$.
\end{proof}

\subsubsection{Interactions between limits}

Let us consider a bifunctor $D \colon \cI \times \cJ \to \cC$,
where $\cI$ and $\cJ$ are small categories.
Let $P$ be the canonical functor from $\Func{\cI \times \cJ}{\cC}$ to $\Functwo{\cJ}{\cI}{\cC}$.
Also, let one of the limit objects of $P \b D \colon \cJ \to \Func{\cI}{\cC}$
be written as $\lim_{j \in \cJ} D(\Endash,j) \in \Func{\cI}{\cC}$.
Note that this limit object is a functor from $\cI$ to $\cC$ and does not depend on $j$;
it may be helpful to think of $j$ as representing something like an internal variable.
Similarly, by swapping the roles of $\cI$ and $\cJ$, we can define
$\lim_{i \in \cI} D(i,\Endash) \in \Func{\cJ}{\cC}$.
Using the same notation, let $\lim_{j \in \cJ} Ej$ denote a limit object
of a functor $E \colon \cJ \to \cC$ (i.e., $\lim_{j \in \cJ} Ej \cong \lim E$ holds).
Then, a limit object of $D(i,\Endash)$ can be
expressed as $\lim_{j \in \cJ} D(i,j) \in \cC$ and a limit object of $D$ can be expressed
as $\lim_{\braket{i,j} \in \cI \times \cJ} D(i,j)$.
Applying Proposition~\ref{pro:LimitBifunc} to the functor $P \b D$,
if there exists a limit object $\lim_{j \in \cJ} D(i,j)$ for each $i \in \cI$,
then there exists a functor $L$ satisfying $L(i) = \lim_{j \in \cJ} D(i,j)$
and $L \cong \lim_{j \in \cJ} D(\Endash,j)$.
Furthermore, when $\lim_{j \in \cJ} D(\Endash,j)$ has a limit, we can write its limit object
as $\lim_{i \in \cI} \lim_{j \in \cJ} D(i,j)$.

The following proposition holds.
\begin{proposition}{}{LimitLimit}
 Let us consider a bifunctor $D \colon \cI \times \cJ \to \cC$, where $\cI$ and $\cJ$
 are small categories.
 Assume that for each $i \in \cI$, $D(i,\Endash)$ has a limit,
 in which case we can define a functor $\lim_{j \in \cJ} D(\Endash,j)$.
 If either $\lim_{j \in \cJ} D(\Endash,j)$ or $D$ has a limit,
 then the other also has a limit and we have
 \begin{alignat}{1}
  \lim_{i \in \cI} \lim_{j \in \cJ} D(i,j)
  &\cong \lim_{\braket{i,j} \in \cI \times \cJ} D(i,j).
  \label{eq:limit_limit}
 \end{alignat}
\end{proposition}
\begin{proof}
 Let $P$ be the canonical functor from $\Func{\cI \times \cJ}{\cC}$ to $\Functwo{\cJ}{\cI}{\cC}$.
 From Proposition~\ref{pro:LimitBifunc}, $P \b D$ has a limit,
 which we denote as $\braket{\lim_{j \in \cJ} D(\Endash,j),\kappa}$.
 From Lemma~\ref{lemma:ReprIsoUniv}, we have
 \begin{alignat}{1}
  \lefteqn{\text{$\braket{\lim_{\braket{i,j} \in \cI \times \cJ} D(i,j),v'}$
  is a limit of $D$ (i.e., a universal morphism from $\Delta_{\cI \times \cJ}$ to $D$)}}
  \nonumber \\
  &\quad\Leftrightarrow\quad
  \text{$\braket{\lim_{\braket{i,j} \in \cI \times \cJ} D(i,j),P \b v'}$
  is a universal morphism from $P \b \Delta_{\cI \times \cJ}$ to $P \b D$}
  \nonumber \\
  &\quad\Leftrightarrow\quad
  \text{$\braket{\lim_{\braket{i,j} \in \cI \times \cJ} D(i,j),P \b v'}$ is
  a universal morphism from $\Delta_\cJ \b \Delta_\cI$ to $P \b D$},
  \nonumber \\
  \label{eq:limit_limit_D_P}
 \end{alignat}
 where the last line follows from $P \b \Delta_{\cI \times \cJ} = \Delta_\cJ \b \Delta_\cI$.
 Note that $P \b v'$ is depicted by
 \begin{alignat}{1}
  \InsertPDF{limit_limit_D_decomp.pdf} \raisebox{1em}{.}
  \label{eq:limit_limit_D_decomp}
 \end{alignat}
 The proof can be easily obtained by applying the dual of Lemma~\ref{lemma:ReprUnivCirc}.

 First, consider the case where $\lim_{j \in \cJ} D(\Endash,j)$ has the limit
 $\braket{\lim_{i \in \cI} \lim_{j \in \cJ} D(i,j),\kappa'}$.
 Let
 \begin{alignat}{1}
  \InsertPDF{limit_limit_D_decomp1.pdf} \raisebox{1em}{,}
  \label{eq:limit_limit_D_decomp1}
 \end{alignat}
 where the blue circle represents $\kappa$; then, from the dual of Lemma~\ref{lemma:ReprUnivCirc},
 $\braket{\lim_{i \in \cI} \lim_{j \in \cJ} D(i,j),v}$ is a universal morphism
 from $\Delta_\cJ \b \Delta_\cI$ to $P \b D$.
 Therefore, from Eq.~\eqref{eq:limit_limit_D_P},
 $\braket{\lim_{i \in \cI} \lim_{j \in \cJ} D(i,j), P^{-1} \b v}$ is a limit of $D$,
 so Eq.~\eqref{eq:limit_limit} holds.

 Next, consider the case where $D$ has a limit.
 From Eq.~\eqref{eq:limit_limit_D_P}, there exists a universal morphism
 $\braket{\lim_{\braket{i,j} \in \cI \times \cJ} D(i,j),v}$
 from $\Delta_\cJ \b \Delta_\cI$ to $P \b D$.
 Thus, from the dual of Lemma~\ref{lemma:ReprUnivCirc} (refer to Eq.~\eqref{eq:limit_limit_D_decomp1}),
 $\lim_{j \in \cJ} D(\Endash,j)$ has a limit of the form
 $\braket{\lim_{\braket{i,j} \in \cI \times \cJ} D(i,j),\kappa'}$.
 Therefore, Eq.~\eqref{eq:limit_limit} holds.
\end{proof}

This proposition leads directly to the following corollary.
\begin{cor}{}{LimitLimitAdj}
 Consider any category $\cC$ and small categories $\cI$ and $\cJ$.
 Assume that any functor from $\cI$ to $\cC$ and functor from $\cJ$ to $\cC$ have limits.
 Then, a functor $D \colon \cI \times \cJ \to \cC$ has a limit
 and satisfies
 \begin{alignat}{1}
  \lim_{i \in \cI} \lim_{j \in \cJ} D(i,j)
  &\cong \lim_{\braket{i,j} \in \cI \times \cJ} D(i,j)
  \cong \lim_{j \in \cJ} \lim_{i \in \cI} D(i,j).
  \label{eq:limit_commutes_with_limit}
 \end{alignat}
\end{cor}
\begin{proof}
 The first isomorphism is obvious from Proposition~\ref{pro:LimitLimit}.
 Swapping $\cI$ and $\cJ$ yields the second isomorphism.
\end{proof}


\section{Kan extensions} \label{sec:Kan}

Kan extensions are concepts that can unify the treatment of adjoints and limits
from different perspectives.
It has been pointed out that ``The notion of Kan extensions subsumes all the other fundamental
concepts of category theory'' \cite{Mac-2013}.
String diagrams are also useful when investigating the properties of Kan extensions.
After discussing general Kan extensions, we explain
the properties of pointwise Kan extensions,
which are often recognized as important Kan extensions.

\subsection{Definition of Kan extensions} \label{subsec:Kan_def}

\begin{define}{Kan extensions}{Kan}
 Consider two functors $K \colon \cC \to \cD$ and $F \colon \cC \to \cE$.
 A universal morphism $\braket{L,\eta}$ from $F$ to
 $\Endash \b K \colon \Func{\cD}{\cE} \to \Func{\cC}{\cE}$
 is called a \termdef{left Kan extension} of $F$ along $K$.
 $\eta$ is also called its \termdef{unit}.
 In other words, $\braket{L,\eta}$ is called a left Kan extension of $F$ along $K$ if
 for any $H \colon \cD \to \cE$ and $\sigma \colon F \nto H \b K$,
 there exists a unique $\ol{\sigma} \colon L \nto H$ satisfying
 \begin{alignat}{1}
  \InsertPDF{Kan_left_universal.pdf} \raisebox{1em}{,}
  \label{eq:Kan_left_universal}
 \end{alignat}
 where the blue circle is the unit $\eta$.

 As the dual to a left Kan extension, a universal morphism
 $\braket{R,\varepsilon}$ from $\Endash \b K$ to $F$
 is called a \termdef{right Kan extension} of $F$ along $K$.
 $\varepsilon$ is also called its \termdef{counit}.
 In other words, $\braket{R,\varepsilon}$ is called a right Kan extension of $F$ along $K$ if
 for any $H \colon \cD \to \cE$ and $\sigma \colon H \b K \nto F$,
 there exists a unique $\ol{\sigma} \colon H \nto R$ satisfying
 \begin{alignat}{1}
  \InsertPDF{Kan_right_universal.pdf} \raisebox{1em}{,}
  \label{eq:Kan_right_universal}
 \end{alignat}
 where the blue circle is the counit $\varepsilon$.
\end{define}

We sometimes simply refer to $L$ as a left Kan extension and $R$ as a right Kan extension.
We often write $L$ as $\Lan_K F$ and $R$ as $\Ran_K F$.
Diagrams for right Kan extensions correspond to the ``upside-down'' version of
diagrams for left Kan extensions.
In what follows, we often only consider the properties of left Kan extensions,
but by flipping diagrams upside-down, we can immediately obtain the properties of
right Kan extensions.

Using a diagram introduced in Subsubsection~\ref{subsubsec:repr_repr_diagram},
the inverse map $\sigma \mapsto \ol{\sigma}$ of the map $\ol{\sigma} \mapsto \sigma$
of Eq.~\eqref{eq:Kan_left_universal} is represented by
\begin{alignat}{1}
 \InsertPDF{Kan_left_universal2.pdf} \raisebox{1em}{.}
 \label{eq:Kan_left_universal2}
\end{alignat}
The same applies to right Kan extensions.

Left Kan extensions are essentially unique if they exist, and so are right Kan extensions.
Here, we specifically discuss the case of left Kan extensions.
Let $\braket{L,\eta}$ be a left Kan extension of $F$ along $K$; then,
for any left Kan extension $\braket{L',\eta'}$ of $F$ along $K$, there exists a natural isomorphism
$\psi \colon L \ntocong L'$ satisfying
\begin{alignat}{1}
 \InsertPDF{Kan_left_unique.pdf} \raisebox{1em}{,}
 \label{eq:Kan_left_unique}
\end{alignat}
where the blue circle is $\eta$, and the diamond block is $\psi$
(see Eq.~\eqref{eq:repr_cG_init_cong}).
Conversely, any $\braket{L',\eta'}$ represented in the form of Eq.~\eqref{eq:Kan_left_unique} is
a left Kan extension of $F$ along $K$.
Thus, any functor isomorphic to $L$ is a left Kan extension.

\begin{ex}{(co)limits are Kan extensions}{KanLimit}
 Consider the case of $\cD = \cOne$, in which $K$ is only one functor ${!}$ from $\cC$ to $\cOne$.
 A left Kan extension of $F$ along ${!}$ is equal to a universal morphism from $F$ to
 $\Endash \b {!} = \Delta_\cC$ (recall Example~\ref{ex:FuncDelta}), i.e., a colimit of $F$.
 This Kan extension is represented by
 \begin{alignat}{1}
  \InsertPDF{Kan_colim.pdf} \raisebox{1em}{.}
  \label{eq:Kan_colim}
 \end{alignat}
 The left blue circle represents the unit of this Kan extension
 and the right blue circle represents a colimit cocone; they can be identified.
 Therefore, a colimit is a special case of a left Kan extension.
 Dually, a limit is a special case of a right Kan extension.
\end{ex}

\begin{ex}{the Yoneda lemma}{}
 From Example~\ref{ex:ReprYonedaUniv}, since $\braket{\yoneda{c},\id_c}$ is a universal morphism
 from $\{*\} \in \Set$ to $\ev_c = \Endash \b c \colon \Func{\cC}{\Set} \to \Set$,
 it is a left Kan extension of $\{*\}$ along $c$.
 The right-hand side of the symbol ``$\Leftrightarrow$'' in Eq.~\eqref{eq:repr_yoneda_init3}
 corresponds to Eq.~\eqref{eq:Kan_left_universal}.
\end{ex}

Assume that $\Func{\cC}{\cE}$ and $\Func{\cD}{\cE}$ are locally small.
By the properties of universal morphisms, there exists a left Kan extension
of $F \colon \cC \to \cE$ along $K \colon \cC \to \cD$ if and only if
the set-valued functor $\Func{\cC}{\cE}(F, \Endash \b K)$ is representable, i.e.,
there exists an $L \colon \cD \to \cE$ satisfying
\begin{alignat}{1}
 \Func{\cD}{\cE}(L, \Endash) \cong \Func{\cC}{\cE}(F, \Endash \b K)
 &\qquad\diagram\qquad
 \InsertMidPDF{Kan_left_cong.pdf},
 \label{eq:Kan_left_cong}
\end{alignat}
in which case $L \cong \Lan_K F$ holds.
Furthermore, by the properties of adjunctions, there exists a left Kan extension
of each $F \colon \cC \to \cE$ along $K$ if and only if there exists
a left adjoint to the functor $\Endash \b K \colon \Func{\cD}{\cE} \to \Func{\cC}{\cE}$.
If $\Endash \b K$ has a left adjoint, denoted by $\Lan_K$, then for its unit $\eta$,
$\eta_F$ is the unit of the left Kan extension $\Lan_K F$.
In this case, from Eq.~\eqref{eq:Kan_left_universal}, for any natural transformation
$\sigma \colon F \nto H \b K$ (where $H \colon \cD \to \cE$ is arbitrary),
there exists a unique $\ol{\sigma} \colon \Lan_K F \nto H$ satisfying
\begin{alignat}{1}
 \footnoteinset{-1.62}{0.3}{\eqref{eq:Kan_left_universal}}{%
 \InsertPDF{Kan_left_universal_adj.pdf}} \raisebox{1em}{.}
 \label{eq:Kan_left_universal_adj}
\end{alignat}

Dually, there exists a right Kan extension of $F$ along $K$ if and only if
the presheaf $\Func{\cC}{\cE}(\Endash \b K,F)$ is representable, i.e., there exists
an $R \colon \cD \to \cE$ satisfying
\begin{alignat}{1}
 \Func{\cD}{\cE}(\Endash,R) \cong \Func{\cC}{\cE}(\Endash \b K, F)
 &\qquad\diagram\qquad
 \InsertMidPDF{Kan_right_cong.pdf},
 \label{eq:Kan_right_cong}
\end{alignat}
in which case $R \cong \Ran_K F$ holds.
Furthermore, there exists a right Kan extension of each $F \colon \cC \to \cE$
along $K$ if and only if there exists
a right adjoint to the functor $\Endash \b K \colon \Func{\cD}{\cE} \to \Func{\cC}{\cE}$.
If $\Endash \b K$ has a right adjoint, denoted by $\Ran_K$, then for its counit $\varepsilon$,
$\varepsilon_F$ is the counit of the right Kan extension $\Ran_K F$.
In this case, from Eq.~\eqref{eq:Kan_right_universal}, for any natural transformation
$\sigma \colon H \b K \nto F$ (where $H \colon \cD \to \cE$ is arbitrary),
there exists a unique $\ol{\sigma} \colon H \nto \Ran_K F$ satisfying
\begin{alignat}{1}
 \footnoteinset{-1.62}{0.3}{\eqref{eq:Kan_right_universal}}{%
 \InsertPDF{Kan_right_universal_adj.pdf}} \raisebox{1em}{.}
 \label{eq:Kan_right_universal_adj}
\end{alignat}

\subsection{Pointwise Kan extensions} \label{subsec:Kan_pointwise}

\subsubsection{Preliminaries: Basic properties of comma categories} \label{subsubsec:Kan_pointwise_comma}

In this subsubsection, we provide an overview of the fundamental properties of comma categories,
which are often used when investigating pointwise left Kan extensions.
Specifically, we consider the comma category $K \comma d$ with $K \colon \cC \to \cD$
and $d \in \cD$.
Note that $K \comma d = \el(\cD(K\Endash,d))$ holds.

Let us call the following functor, denoted by $P_d \colon K \comma d \to \cC$,
the \termdef{forgetful functor} from $K \comma d$ to $\cC$:
\begin{itemize}
 \item It maps each object $\braket{c,p}$ in $K \comma d$ to $c$.
 \item It maps each morphism $h \colon \braket{c,p} \to \braket{c',p'}$ in $K \comma d$
       to the morphism $h \colon c \to c'$ in $\cC$.
\end{itemize}
It is readily confirmed that $P_d$ is a functor.
The behavior of this functor can be represented by
\begin{alignat}{1}
 \InsertPDF{Kan_left_pointwise_comma_Pd.pdf} \raisebox{1em}{.}
 \label{eq:Kan_left_pointwise_comma_Pd}
\end{alignat}
The collection of morphisms
\begin{alignat}{1}
 \lefteqn{ \theta^d = \{ \theta^d_{\braket{c,p}} \coloneqq p \}_{\braket{c,p} \in K \comma d} }
 \nonumber \\
 &\diagram\qquad \InsertMidPDF{Kan_left_pointwise_comma_cp.pdf}
 \label{eq:Kan_left_pointwise_comma_cp}
\end{alignat}
(where the blue circles represent $\theta^d$, and the dotted lines represent
the functor ${!} \colon K \comma d \to \cOne$) is a natural transformation
from $K \b P_d$ to $d \b {!}$ (i.e., a cocone from $K \b P_d$ to $d$).
Indeed, since we have for any morphism $h \colon \braket{c,p} \to \braket{c',p'}$
in $K \comma d$,
\begin{alignat}{1}
 \footnoteinset{-3.05}{0.3}{\eqref{eq:Kan_left_pointwise_comma_cp}}{%
 \footnoteinset{0.00}{0.3}{\eqref{eq:repr_Gc}}{%
 \footnoteinsets{3.05}{0.3}{\eqref{eq:Kan_left_pointwise_comma_cp}}{\eqref{eq:Kan_left_pointwise_comma_Pd}}{%
 \InsertPDF{Kan_left_pointwise_comma_ph.pdf}}}} \raisebox{1em}{,}
 \label{eq:Kan_left_pointwise_comma_ph}
\end{alignat}
$\theta^d$ satisfies naturality.
We call $\theta^d$ the \termdef{canonical cocone} for $K \comma d$.
From Eq.~\eqref{eq:Kan_left_pointwise_comma_cp},
\begin{alignat}{1}
 \theta^{Kc}_{\braket{c,\id_{Kc}}} = \id_{Kc}
 &\qquad\diagram\qquad \InsertMidPDF{Kan_left_pointwise_comma_c1.pdf}
 \label{eq:Kan_left_pointwise_comma_c1}
\end{alignat}
holds.

The map, denoted by $\Phi$, that maps $\{ \{ p \}_{p \in \cD(Kc,d)} \}_{c \in \cC}$
to $\{ p \}_{c \in \cC, p \in \cD(Kc,d)} = \theta^d$
(by removing nesting) can be represented by
\begin{alignat}{1}
 \InsertPDF{Kan_left_pointwise_comma_theta.pdf} \raisebox{1em}{.}
 \label{eq:Kan_left_pointwise_comma_theta}
\end{alignat}
Note that $c \in \cC$ and $p \in \cD(Kc,d)$ is equivalent to $\braket{c,p} \in K \comma d$.
$\Phi$ is obviously invertible.

For any morphism $f \colon d \to d'$ in $\cD$, there exists a functor, denoted by $f \c \Endash$,
from $K \comma d$ to $K \comma d'$ that
\begin{itemize}
 \item maps each object $\braket{c,p}$ in $K \comma d$ to $\braket{c,fp} \in K \comma d'$.
 \item maps each morphism $h \colon \braket{c,p} \to \braket{c',p'}$ in $K \comma d$
       to $h \colon \braket{c,fp} \to \braket{c',fp'}$ in $K \comma d'$.
       Note that both morphisms correspond to $h \in \cC(c,c')$ satisfying $p = p' \c Kh$.
\end{itemize}
One can easily verify that $f \c \Endash$ is a functor.
$\braket{c,p} \in K \comma d$ can be represented by
\begin{alignat}{1}
 \InsertPDF{Kan_left_pointwise_comma_cp_id.pdf}
 \label{eq:Kan_left_pointwise_comma_cp_id}
\end{alignat}
since $p = p \c \id_{Kc}$ holds.
Also, we have for any $f \in \cD(d,d')$ and $\braket{c,p} \in K \comma d$,
\begin{alignat}{1}
 \footnoteinset{0.20}{0.3}{\eqref{eq:Kan_left_pointwise_comma_cp}}{%
 \footnoteinset{3.11}{0.3}{\eqref{eq:Kan_left_pointwise_comma_cp}}{%
 \InsertPDF{Kan_left_pointwise_comma_fp_proof.pdf}}} \raisebox{1em}{,}
 \label{eq:Kan_left_pointwise_comma_fp_proof}
\end{alignat}
which gives
\begin{alignat}{1}
 \InsertPDF{Kan_left_pointwise_comma_fp.pdf} \raisebox{1em}{.}
 \label{eq:Kan_left_pointwise_comma_fp}
\end{alignat}

The comma category $d \comma K$, which is equal to $\el(\cD(d,K\Endash))$,
can be considered similarly.
The \termdef{forgetful functor}, denoted by $P_d \colon d \comma K \to \cC$, from $d \comma K$ to $\cC$
is defined as follows:
\begin{itemize}
    \item It maps each object $\braket{c,p}$ in $d \comma K$ to $c$.
    \item It maps each morphism $h$ in $d \comma K$ to the morphism $h$ in $\cC$.
\end{itemize}
In this case, the collection of morphisms
\begin{alignat}{1}
 \theta^d &\coloneqq \{ \theta^d_{\braket{c,p}} \coloneqq p \}_{\braket{c,p} \in d \comma K}
\end{alignat}
is a natural transformation from $d \b {!}$ to $K \b P_d$
(i.e., a cone from $d$ to $K \b P_d$).
We call $\theta^d$ the \termdef{canonical cone} for $d \comma K$.

\subsubsection{Pointwise Kan extensions} \label{subsubsec:Kan_pointwise_pointwise}

Assume that there exists a left Kan extension $\braket{L,\eta}$
of a functor $F \colon \cC \to \cE$ along a functor $K \colon \cC \to \cD$.
To explicitly determine $L$ and $\eta$, let us consider the composite
of $\eta$ and the canonical cocone $\theta^d$ for $K \comma d$
for each $d \in \cD$.
Specifically, we define $\mu^d$ as
\begin{alignat}{1}
 \mu^d \coloneqq (L \b \theta^d) (\eta \b P_d)
 &\qquad\diagram\qquad
 \InsertMidPDF{Kan_left_pointwise0.pdf},
 \label{eq:Kan_left_pointwise0}
\end{alignat}
where the left and right blue circles represent $\eta$ and $\theta^d$, respectively.
Since $\mu^d$ is simply a composite of the unit, $\eta$, of the left Kan extension and $\theta^d$,
it can be considered somewhat ``canonical''.
On the other hand, $\mu^d$ is a cocone from $F \b P_d$ to $Ld$.
If each $\mu^d$ is a universal morphism, i.e., a colimit of $F \b P_d$, then $Ld$ will be
essentially uniquely determined as a colimit object of $F \b P_d$.
Moreover, each component of the unit $\eta$ can be obtained by
\begin{alignat}{1}
 \footnoteinset{-2.24}{0.3}{\eqref{eq:Kan_left_pointwise_comma_c1}}{%
 \footnoteinset{1.18}{0.3}{\eqref{eq:Kan_left_pointwise0}}{%
 \InsertPDF{Kan_left_pointwise_eta.pdf}}} \raisebox{1em}{,}
 \label{eq:Kan_left_pointwise_eta}
\end{alignat}
and thus we obtain $\eta = \{ \mu^{Kc}_{\braket{c,1_{Kc}}} \}_{c \in \cC}$.

Unfortunately, the existence of a left Kan extension $\braket{L,\eta}$ does not
guarantee that $F \b P_d$ has a colimit for each $d \in \cD$.
However, conversely, if these colimits exist, then it can be shown that
\begin{itemize}
 \item There exists a left Kan extension $\braket{L,\eta}$ satisfying
       Eq.~\eqref{eq:Kan_left_pointwise0} (see Theorem~\ref{thm:KanPointwise}).
 \item For any left Kan extension $\braket{L,\eta}$,
       each $\mu^d$ defined by Eq.~\eqref{eq:Kan_left_pointwise0} is a colimit cocone of $F \bullet P_d$
       (see Corollary~\ref{cor:KanPointwiseNas0}).
\end{itemize}
Therefore, let us call a left Kan extension pointwise if $F \b P_d$ has a colimit for each $d \in \cD$.
Explicitly, it is defined as follows:
\begin{define}{pointwise Kan extensions}{KanPointwise}
 Consider two functors $K \colon \cC \to \cD$ and $F \colon \cC \to \cE$.
 A left Kan extension of $F$ along $K$ is called \termdef{pointwise}
 if $F \b P_d$ has a colimit for each $d \in \cD$.
 Dually, a right Kan extension of $F$ along $K$ is called \termdef{pointwise}
 if $F \b P_d$ has a limit for each $d \in \cD$.
\end{define}

If a pointwise left Kan extension $L$ exists, then any left Kan extension of $F$ along $K$
is pointwise (the same applies to right Kan extensions).
This is because whether or not $F \b P_d$ has a colimit does not depend on the choice of the left Kan extension.
As will gradually become clear, several useful properties hold if a Kan extension is pointwise.

If a left Kan extension $\braket{L,\eta}$ is pointwise, then we have
\begin{alignat}{1}
 \InsertPDF{Kan_left_pointwise.pdf} \raisebox{1em}{,}
 \label{eq:Kan_left_pointwise}
\end{alignat}
where the left and right blue circles are $\eta$ and $\theta^d$, respectively, and
the diamond block represents the isomorphism $Ld \cong \colim (F \b P_d)$.
The natural transformation enclosed by the auxiliary line is the colimit cocone $\mu^d$
of Eq.~\eqref{eq:Kan_left_pointwise0}.
Due to the universal property of $\mu^d$ (see Eq.~\eqref{eq:limit_repr_terminal_op}),
for any cocone $\alpha \in \Cocone(F \b P_d,e)$,
there exists a unique morphism $\ol{\alpha} \in \cE(Ld,e)$ satisfying
\begin{alignat}{1}
 \InsertPDF{Kan_left_pointwise_univ.pdf} \raisebox{1em}{.}
 \label{eq:Kan_left_pointwise_univ}
\end{alignat}

It follows from the dual of Eq.~\eqref{eq:Kan_left_pointwise} that a pointwise right Kan extension
$\braket{R,\varepsilon}$ satisfies
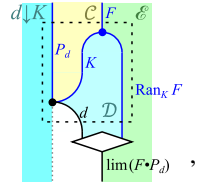
\begin{alignat}{1}
 \InsertPDF{Kan_right_pointwise.pdf} \raisebox{1em}{,}
 \label{eq:Kan_right_pointwise}
\end{alignat}
where the left and right blue circles are $\varepsilon$ and
the canonical cone $\theta^d$ for $d \comma K$, respectively.
The diamond block represents the isomorphism $Rd \cong \lim (F \b P_d)$ and
the natural transformation enclosed by the auxiliary line is a limit cone of $F \b P_d$.

\begin{thm}{}{KanPointwise}
 Consider two functors $K \colon \cC \to \cD$ and $F \colon \cC \to \cE$.
 Assume that for each $d \in \cD$, $F \b P_d$ has a colimit $\braket{e_d,\mu^d}$,
 where $P_d$ is the forgetful functor from $K \comma d$ to $\cC$.
 Then, the following two properties hold:
 \begin{enumerate}
  \item There exists a unique functor $L \colon \cD \to \cE$ that maps each object $d$ in $\cD$
        to $Ld = e_d$, and satisfies
        \begin{alignat}{1}
         \lefteqn{ \mu^{d'} \b (f \c \Endash) = (Lf \b {!}) \c \mu^d } \nonumber \\
         &\diagram\qquad
         \InsertMidPDF{Kan_left_pointwise_proof_Lf.pdf}
         \label{eq:Kan_left_pointwise_Lf}
        \end{alignat}
        for each morphism $f \colon d \to d'$ in $\cD$,
        where the left and right blue ellipses represent $\mu^{d'}$ and $\mu^d$, respectively.
  \item Let
        \begin{alignat}{1}
         \lefteqn{ \eta \coloneqq \{ \eta_c \coloneqq \mu^{Kc}_{\braket{c,\id_{Kc}}} \}_{c \in \cC} }
         \nonumber \\
         &\diagram\qquad \InsertMidPDF{Kan_left_pointwise_eta_mu.pdf},
         \nonumber \\
         \label{eq:Kan_left_pointwise_eta_mu}
        \end{alignat}
        where the blue circle is $\eta$ (see Eq.~\eqref{eq:Kan_left_pointwise_eta});
        then, $\braket{L,\eta}$ is a pointwise left Kan extension of $F$ along $K$.
 \end{enumerate}
\end{thm}
\begin{proof}
 (1): Since the left-hand side of Eq.~\eqref{eq:Kan_left_pointwise_Lf} is a cocone from $F \b P_d$
 and $\mu^d$ is a colimit cocone, $Lf$ satisfying Eq.~\eqref{eq:Kan_left_pointwise_Lf}
 is uniquely determined.
 It remains to show that $L$ is a functor.
 We have for any $f \in \cD(d,d')$ and $g \in \cD(d',d'')$,
 \begin{alignat}{1}
  \footnoteinset{-3.43}{0.3}{\eqref{eq:Kan_left_pointwise_Lf}}{%
  \footnoteinset{3.45}{0.3}{\eqref{eq:Kan_left_pointwise_Lf}}{%
  \InsertPDF{Kan_left_pointwise_proof_Lgf.pdf}}} \raisebox{1em}{.}
  \label{eq:Kan_left_pointwise_Lgf}
 \end{alignat}
 Therefore, due to the universal property of $\mu^d$ (see Eq.~\eqref{eq:repr_univ_eq}),
 we have $L(gf) = Lg \c Lf$.
 Also, $L \id_d = \id_{Ld}$ obviously holds, and thus $L$ is a functor.
 
 (2): First, let us verify that $\eta$ is a natural transformation from $F$ to $L \b K$.
 This can be seen from
 \begin{alignat}{1}
  \footnoteinset{-3.44}{1.5}{\eqref{eq:Kan_left_pointwise_eta_mu}}{%
  \footnoteinsets{0.00}{1.5}{\eqref{eq:Kan_left_pointwise_Lf}}{\eqref{eq:Kan_left_pointwise_comma_cp_id}}{%
  \footnoteinset{0.00}{-0.85}{\eqref{eq:sliding}}{%
  \footnoteinsets{3.44}{-0.85}{\eqref{eq:Kan_left_pointwise_eta_mu}}{\eqref{eq:Kan_left_pointwise_comma_Pd}}{%
  \InsertPDF{Kan_left_pointwise_proof_eta_nat.pdf}}}}} \raisebox{1em}{,}
  \label{eq:Kan_left_pointwise_eta_nat}
 \end{alignat}
 where the first equality on the second line follows from the fact that ${!} \b h$ is
 the unique morphism $\id_*$ in $\cOne$ regardless of $h$.
 Next, we show that $\braket{L,\eta}$ is a left Kan extension of $F$ along $K$.
 Note that if $\braket{L,\eta}$ is a left Kan extension, then it is pointwise by the definition
 of pointwise Kan extensions.
 We have for each $\braket{c,p} \in K \comma d$,
 \begin{alignat}{1}
  \footnoteinsets{-3.53}{0.3}{\eqref{eq:Kan_left_pointwise_comma_cp_id}}{\eqref{eq:Kan_left_pointwise_Lf}}{%
  \footnoteinset{0.18}{0.3}{\eqref{eq:Kan_left_pointwise_eta_mu}}{%
  \footnoteinset{3.35}{0.3}{\eqref{eq:Kan_left_pointwise_comma_cp}}{%
  \InsertPDF{Kan_left_pointwise_proof_eta.pdf}}}} \raisebox{1em}{,}
  \nonumber \\
  \label{eq:Kan_left_pointwise_proof_eta}
 \end{alignat}
 where the rightmost blue circle represents the canonical cocone $\theta^d$ for $K \comma d$,
 and thus $\mu^d = (L \b \theta^d) (\eta \b P_d)$ (i.e., Eq.~\eqref{eq:Kan_left_pointwise0}) holds.
 It suffices to show that for any natural transformation $\alpha \colon F \nto H \b K$
 (where $H \colon \cD \to \cE$ is arbitrary),
 there exists a unique natural transformation $\beta$ such that $\alpha = (\beta \b K) \c \eta$
 (see Eq.~\eqref{eq:Kan_left_universal}).
 Assuming such $\beta$ exists, we have for each $d \in \cD$,
 \begin{alignat}{1}
  \footnoteinset{-0.27}{0.4}{\eqref{eq:sliding}}{%
  \footnoteinset{2.38}{0.4}{\eqref{eq:Kan_left_pointwise_proof_eta}}{%
  \InsertPDF{Kan_left_pointwise_proof_beta.pdf}}} \raisebox{1em}{.}
  \label{eq:Kan_left_pointwise_proof_beta}
 \end{alignat}
 In contrast, since $\mu^d$ is a colimit cocone,
 there exists a unique morphism $\ol{\alpha}_d$ satisfying
 \begin{alignat}{1}
  \footnoteinset{-0.33}{0.3}{\eqref{eq:limit_repr_terminal_op}}{%
  \InsertPDF{Kan_left_pointwise_proof_alpha_univ.pdf}} \raisebox{1em}{,}
  \label{eq:Kan_left_pointwise_proof_alpha_univ}
 \end{alignat}
 where we use the fact that the left-hand side is a cocone from $F \b P_d$ to $Hd$.
 From Eqs.~\eqref{eq:Kan_left_pointwise_proof_beta} and \eqref{eq:Kan_left_pointwise_proof_alpha_univ},
 we have $\beta_d = \ol{\alpha}_d$.
 Thus, $\beta$ is uniquely determined as
 $\beta = \{ \ol{\alpha}_d \}_{d \in \cD} \eqqcolon \ol{\alpha}$.
 Also, if $\ol{\alpha}$ is a natural transformation, then
 \begin{alignat}{1}
  \footnoteinset{-3.52}{0.4}{\eqref{eq:Kan_left_pointwise_comma_c1}}{%
  \footnoteinset{-0.47}{0.4}{\eqref{eq:Kan_left_pointwise_proof_alpha_univ}}{%
  \footnoteinset{3.24}{0.4}{\eqref{eq:Kan_left_pointwise_eta_mu}}{%
  \InsertPDF{Kan_left_pointwise_proof_alpha.pdf}}}}
  \label{eq:Kan_left_pointwise_proof_alpha}
 \end{alignat}
 holds for each $c \in \cC$, and thus $\alpha = (\ol{\alpha} \b K) \c \eta$ holds.
 Therefore, it remains to show that $\ol{\alpha}$ is a natural transformation.
 Since we have for any morphism $f \colon d \to d'$ in $\cD$,
 \begin{alignat}{1}
  \footnoteinset{-1.78}{1.85}{\eqref{eq:Kan_left_pointwise_proof_alpha_univ}}{%
  \footnoteinset{1.13}{1.85}{\eqref{eq:Kan_left_pointwise_comma_fp}}{%
  \footnoteinset{-1.78}{-1.45}{\eqref{eq:Kan_left_pointwise_proof_alpha_univ}}{%
  \footnoteinset{1.77}{-1.45}{\eqref{eq:Kan_left_pointwise_Lf}}{%
  \InsertPDF{Kan_left_pointwise_proof_alpha_nat2.pdf}}}}}
  \label{eq:Kan_left_pointwise_proof_alpha_nat2}
 \end{alignat}
 by the universality of the colimit cocone $\mu^d$ (see Eq.~\eqref{eq:repr_univ_eq}),
 the two morphisms enclosed by the auxiliary lines,
 $Hf \c \ol{\alpha}_d$ and $\ol{\alpha}_{d'} \c Lf$, are equal.
 Therefore, $\ol{\alpha}$ is a natural transformation.
\end{proof}
\begin{supplemental}
 Let us point out that this theorem is similar to Lemma~\ref{lemma:ReprFuncH}.
 This lemma asserts that a functor $F$ is uniquely determined and that $\braket{F,\eta}$
 is a universal morphism.
 In contrast, this theorem asserts that a functor $L$ is uniquely determined
 and that $\braket{L,\eta}$ is a universal morphism
 (i.e., a left Kan extension of $F$ along $K$).
 Note that Eq.~\eqref{eq:Kan_left_pointwise_eta_mu} corresponds to Eq.~\eqref{eq:repr_funcH_univ}.
 The proof methods used also exhibit many similarities.
\end{supplemental}

This theorem shows that if $F \b P_d$ has a colimit $\braket{e_d,\mu^d}$ for each $d \in \cD$,
then there exists a pointwise left Kan extension $\braket{L,\eta}$ of $F$ along $K$ such that
$Ld = e_d$ and Eq.~\eqref{eq:Kan_left_pointwise_Lf} hold.
Conversely, the following corollary shows that for any pointwise left Kan extension $\braket{L,\eta}$
of $F$ along $K$, there exists such a colimit $\braket{e_d,\mu^d}$.
\begin{cor}{}{KanPointwiseNas0}
 Let us consider two functors $K \colon \cC \to \cD$ and $F \colon \cC \to \cE$.
 Assume that there exists a pointwise left Kan extension $\braket{L,\eta}$ of $F$ along $K$.
 Let $P_d$ be the forgetful functor from $K \comma d$ to $\cC$ and
 $\mu^d \coloneqq (L \b \theta^d) (\eta \b P_d)$ (see Eq.~\eqref{eq:Kan_left_pointwise0})
 for each $d \in \cD$.
 Then, Eq.~\eqref{eq:Kan_left_pointwise_Lf} holds and for each $d \in \cD$,
 $\braket{Ld,\mu^d}$ is a colimit of $F \b P_d$.
\end{cor}
\begin{proof}
 By assumption, $F \b P_d$ has a colimit for each $d \in \cD$.
 Let $\mu'{}^d$ be its colimit cocone and $\braket{L',\eta'}$ be
 $\braket{L,\eta}$ obtained by Theorem~\ref{thm:KanPointwise} with $\mu^d = \mu'{}^d$.
 Then, $\mu^d$ satisfies
 \begin{alignat}{1}
  \footnoteinset{4.44}{0.3}{\eqref{eq:Kan_left_pointwise0}}{%
  \InsertPDF{report_Kan_left_pointwise_uniq_mu.pdf}} \raisebox{1em}{,}
  \label{eq:Kan_left_pointwise_uniq_mu}
 \end{alignat}
 where the leftmost blue circle represents $\eta$,
 the other blue circles represent the canonical cocone $\theta^d$ for $K \comma d$,
 and the blue ellipse represents $\mu'{}^d$.
 The second equality follows from the fact that since both $\braket{L,\eta}$ and $\braket{L',\eta'}$
 are left Kan extensions of $F$ along $K$, there exists a natural isomorphism
 $\gamma \colon L' \ntocong L$ such that $\eta = (\gamma \b K) \c \eta'$
 (see Eq.~\eqref{eq:Kan_left_unique}).
 Thus, $\braket{Ld,\mu^d}$, as well as $\braket{L'd,\mu'{}^d}$, is a colimit of $F \b P_d$
 (see Eq.~\eqref{eq:limit_cong}).
 Also, Eq.~\eqref{eq:Kan_left_pointwise_Lf} holds from
 \begin{alignat}{1}
  \footnoteinset{-2.95}{0.3}{\eqref{eq:Kan_left_pointwise_uniq_mu}}{%
  \footnoteinset{0.23}{0.3}{\eqref{eq:Kan_left_pointwise_Lf}}{%
  \footnoteinset{3.53}{0.3}{\eqref{eq:Kan_left_pointwise_uniq_mu}}{%
  \InsertPDF{Kan_left_pointwise_unique_mu_Lf.pdf}}}} \raisebox{1em}{.}
  \label{eq:Kan_left_pointwise_unique_mu_Lf}
 \end{alignat}
\end{proof}

\subsubsection{Cocones from $F \b P_d$}

For a pointwise left Kan extension $L$, $Ld$ is a colimit object of $F \b P_d$
for each $d \in \cD$.
Therefore, when considering pointwise left Kan extensions, it is useful to deepen
our understanding of cocones from $F \b P_d$ for future discussions.

Let us arbitrarily choose $e \in \cE$ and a cocone $\alpha \in \Cocone(F \b P_d,e)$.
$\alpha$ can be represented by
\begin{alignat}{1}
 \InsertPDF{Kan_pointwise_cocone_alpha.pdf} \raisebox{1em}{,}
 \label{eq:Kan_pointwise_cocone_alpha}
\end{alignat}
where we used $P_d \braket{c,p} = c$.
The cocone $\alpha = \{ \alpha_{\braket{c,p}} \}_{\braket{c,p} \in K \comma d}$
can be easily seen to correspond one-to-one with
$\tilde{\alpha} \coloneqq \{ \{ \alpha_{\braket{c,p}} \}_{p \in \cD(Kc,d)} \}_{c \in \cC}$,
where we used the fact that $\braket{c,p} \in K \comma d$ is equal to
$c \in \cC$ and $p \in \cD(Kc,d)$.
While $\alpha$ maps each pair $\braket{c,p}$ of $c \in \cC$ and $p \in \cD(Kc,d)$
to the morphism $\alpha_{\braket{c,p}} \in \cE(Fc,e)$,
$\tilde{\alpha}$ maps each $c \in \cC$ to the map (i.e., the morphism in $\Set$)
$\tilde{\alpha}_c \colon \cD(Kc,d) \ni p \mapsto \alpha_{\braket{c,p}} \in \cD(Fc,e)$.

As will be shown in Lemma~\ref{lemma:KanPointwiseCoconeCong}, this correspondence ensures that
the $\tilde{\alpha}$ corresponding to the cocone $\alpha \in \Cocone(F \b P_d,e)$
is an element of $\hat{\cC}(\cD(K\Endash,d),\cE(F\Endash,e))$
(i.e., a natural transformation) and the map
\begin{alignat}{1}
 \Psi_e \colon \Cocone(F \b P_d,e) \ni \alpha \mapsto \tilde{\alpha}
 \in \hat{\cC}(\cD(K\Endash,d),\cE(F\Endash,e))
 \label{eq:Kan_pointwise_cocone_alpha_Psi}
\end{alignat}
is invertible.
It is also understood that $\{ \Psi_e \}_{e \in \cE}$ is a natural isomorphism.
\begin{lemma}{}{KanPointwiseCoconeCong}
 Consider an object $d \in \cD$ and two functors $K \colon \cC \to \cD$ and $F \colon \cC \to \cE$.
 Let $P_d$ be the forgetful functor from $K \comma d$ to $\cC$; then, there exists an isomorphism
 \begin{alignat}{1}
  \Cocone(F \b P_d,e) &\cong \hat{\cC}(\cD(K\Endash,d),\cE(F\Endash,e))
  \label{eq:Kan_pointwise_cocone_cong}
 \end{alignat}
 that is natural in $e \in \cE$.
\end{lemma}
The isomorphism of Eq.~\eqref{eq:Kan_pointwise_cocone_cong} can be represented by
\begin{alignat}{1}
 \InsertPDF{Kan_pointwise_cocone_cong.pdf} \raisebox{1em}{,}
 \label{eq:Kan_pointwise_cocone_cong_picture}
\end{alignat}
where we used $\cD(K\Endash,d) = \yonedaop{d} \b K$ and $\cE(F\Endash,e) = \yonedaop{e} \b F$,
and replaced $e$ with $\Endash$.
Intuitively, the wire $P_d$ on the left-hand side corresponds to the two wires $K$ and
$\yonedaop{d}$ on the right-hand side.
\begin{proof}
 Let $T_d \coloneqq F \b P_d$.
 For any $e \in \cE$, consider the following map for
 $\tau \in \hat{\cC}(\cD(K\Endash,d),\cE(F\Endash,e))$:
 \begin{alignat}{1}
  \Phi_e \colon \tau = \{ \{ \tau_{\braket{c,p}} \}_{p \in \cD(Kc,d)} \}_{c \in \cC}
  \mapsto \{ \tau_{\braket{c,p} }\}_{c \in \cC, p \in \cD(Kc,d)}.
  \label{eq:Kan_pointwise_cocone_Phi_def}
 \end{alignat}
 It suffices to show that for each $e$, $\Phi_e$ is a bijection from
 $\hat{\cC}(\cD(K\Endash,d),\cE(F\Endash,e))$ to $\Cocone(T_d,e)$,
 and that these isomorphisms are natural in $e$.
 Note that $\Phi_e$ is the inverse of $\Psi_e$.

 We have for any $\tau \in \hat{\cC}(\cD(K\Endash,d),\cE(F\Endash,e))$,
 \begin{alignat}{1}
  \InsertPDF{Kan_pointwise_cocone_Phi.pdf} \raisebox{1em}{,}
  \label{eq:Kan_pointwise_cocone_Phi}
 \end{alignat}
 the black and blue circles represent $\id_d$ and the canonical cocone $\theta^d$ for $K \comma d$,
 respectively.
 The first equality follows from the fact that each component $\tau_c$ of $\tau$ is a map
 from $\cD(Kc,d)$ to $\cE(Fc,e)$, and $\cE(Fc,e)$ is represented in the form
 corresponding to the dual (i.e., ``upside-down'') of the third expression of
 Eq.~\eqref{eq:repr_yoneda_init_summary_op}.
 See Eq.~\eqref{eq:Kan_left_pointwise_comma_theta} for the map $\Phi_e$.
 Therefore, $\Phi_e(\tau) \in \Cocone(\cE(T_d\Endash,e), \{*\}) = \Cocone(T_d,e)$ holds,
 where the equality is due to Lemma~\ref{lemma:LimitCone1cD}.
 Also, it is easy to see that this is natural in $e$; indeed, the structure of the equation will not change
 even if we apply $\yonedaop{f}$ (where $f \in \cE(e,e')$ is arbitrary) from below at the position
 of the wire $\yonedaop{e}$.
 It remains to show that
 $\Phi_e \colon \hat{\cC}(\cD(K\Endash,d),\cE(F\Endash,e)) \to \Cocone(T_d,e)$ is a bijection.
 $\Phi_e$ is obviously injective; indeed, for any distinct $\tau, \tau' \in \Cocone(T_d,e)$,
 there exists $\braket{c,p} \in K \comma d$ satisfying
 $\tau_{\braket{c,p}} \neq \tau'_{\braket{c,p}}$.
 In what follows, we show that $\Phi_e$ is surjective.
 For each $\sigma \in \Cocone(T_d,e)$, let
 \begin{alignat}{1}
  \ol{\sigma} &\coloneqq \{ \{ \sigma_{\braket{c,p}} \}_{p \in \cD(Kc,d)} \}_{c \in \cC}.
  \label{eq:Kan_pointwise_cocone_sigma}
 \end{alignat}
 Let us verify that $\ol{\sigma} \in \hat{\cC}(\cD(K\Endash,d),\cE(F\Endash,e))$ holds,
 i.e., $\ol{\sigma}$ satisfies naturality.
 For any $p' \in \cD(Kc',d)$ and $h \in \cC(c,c')$, let $p \coloneqq p' \c Kh$; then,
 we have
 \begin{alignat}{1}
  \footnoteinset{0.40}{-1.25}{\eqref{eq:Kan_left_pointwise_comma_Pd}}{%
  \InsertPDF{Kan_pointwise_cocone_sigma_nat.pdf}} \raisebox{1em}{.}
  \label{eq:Kan_pointwise_cocone_sigma_nat}
 \end{alignat}
 Therefore,
 \begin{alignat}{1}
  \InsertPDF{Kan_pointwise_cocone_sigma_nat2.pdf}
  \label{eq:Kan_pointwise_cocone_sigma_nat2}
 \end{alignat}
 holds, and thus $\ol{\sigma}$ satisfies naturality.
 It follows from Eqs.~\eqref{eq:Kan_pointwise_cocone_Phi_def} and \eqref{eq:Kan_pointwise_cocone_sigma}
 that $\Phi_e(\ol{\sigma}) = \sigma$ holds, so $\Phi_e$ is surjective.
\end{proof}

The map $\Psi_e$ of Eq.~\eqref{eq:Kan_pointwise_cocone_alpha_Psi} can be expressed by
\begin{alignat}{1}
 \InsertMidPDF{Kan_pointwise_cocone_alpha_diagram.pdf}
 &\qquad\xmapsto{\Psi_e}\qquad
 \InsertMidPDF{Kan_pointwise_cocone_alpha_diagram2.pdf},
 \label{eq:Kan_pointwise_cocone_alpha_diagram}
\end{alignat}
where the natural transformation $\tilde{\alpha}$ is represented by
a block in the shape of ``$\raisebox{-.1em}{\InsertPDFtext{report_text_Kan_KdFe.pdf}}$''.
The equality in Eq.~\eqref{eq:Kan_pointwise_cocone_alpha_diagram} follows from
\begin{alignat}{1}
 \footnoteinset{2.76}{0.3}{\eqref{eq:Kan_left_pointwise_comma_cp}}{%
 \InsertPDF{Kan_pointwise_cocone_alpha_cp.pdf}}
 \label{eq:Kan_pointwise_cocone_alpha_cp}
\end{alignat}
for each $\braket{c,p} \in K \comma d$, where the second equality
follows from $\alpha_{\braket{c,p}} = \tilde{\alpha}_c(p)$.
Intuitively, the map $\Psi_e$ of Eq.~\eqref{eq:Kan_pointwise_cocone_alpha_diagram} replaces
$\theta^d$ with the dotted box representing $\cD(K\Endash,d)$.

We note that the equation on the left-hand side of the symbol ``$\xmapsto{\Psi_e}$''
in Eq.~\eqref{eq:Kan_pointwise_cocone_alpha_diagram} can be rewritten as
\begin{alignat}{1}
 \InsertMidPDF{Kan_pointwise_cocone_alpha_diagram.pdf}
 \qquad\myLeftrightarrow{equivalent}\qquad
 \InsertMidPDF{Kan_pointwise_cocone_theta.pdf},
 \nonumber \\
 \label{eq:Kan_pointwise_cocone_theta_sup}
\end{alignat}
where the functor $\{ * \} \b {!} \colon (K \comma d)^\op \to \Set$,
which is the horizontal composition of the functors $\{ * \} \colon \cOne \to \Set$
and ${!} \colon (K \comma d)^\op \to \cOne$, is represented by a gray dotted line.
In fact, $\alpha \in \Cocone(F \b P_d,e)$ in Eq.~\eqref{eq:Kan_pointwise_cocone_alpha_diagram}
can be depicted by
\begin{alignat}{1}
 \InsertPDF{Kan_pointwise_cocone_theta_alpha.pdf} \raisebox{1em}{,}
 \label{eq:Kan_pointwise_cocone_theta_alpha}
\end{alignat}
where the first equality follows from the ``upside-down'' version of
Eq.~\eqref{eq:limit_cone_1cD2} with $c = e$ and $D = F \b P_d$.
The second equality follows from $\yonedaop{e} \b F = \cE(F\Endash,e)$.
Note that we use both rectangles and ellipses as block shapes to represent $\alpha$,
but there is no semantic difference between them.
Let $\theta^d$ be the canonical cocone for $K \comma d$; then,
substituting $\alpha = \theta^d$ (and $F = K$, $~e = d$)
into Eq.~\eqref{eq:Kan_pointwise_cocone_theta_alpha} yields
\begin{alignat}{1}
 \InsertPDF{Kan_pointwise_cocone_theta_theta.pdf} \raisebox{1em}{,}
 \label{eq:Kan_pointwise_cocone_theta_theta}
\end{alignat}
where the blue circles represent $\theta^d$.
We can see that Eqs.~\eqref{eq:Kan_pointwise_cocone_theta_alpha} and
\eqref{eq:Kan_pointwise_cocone_theta_theta} give Eq.~\eqref{eq:Kan_pointwise_cocone_theta_sup}.

Equation~\eqref{eq:Kan_pointwise_cocone_theta_sup} can also be considered as
a generalization of the Yoneda lemma (see Corollary~\ref{cor:Yonedaop}).
Choose an arbitrary presheaf $X \colon \cC^\op \to \Set$, and substitute
$\cE = \Set^\op$, $~e = \{ * \}$, and $F = X^\op$ into Eq.~\eqref{eq:Kan_pointwise_cocone_theta_sup};
then, since $\cE(F\Endash,e) = \Set^\op(X^\op\Endash,\{*\}) = \Set(\{*\},X\Endash) = X$ holds,
for any cone $\alpha$ from $\{ * \}$ to $X \b P_d \colon (K \comma d)^\op \to \Set$,
there exists a unique natural transformation $\tilde{\alpha} \colon \cD(K\Endash,d) \nto X$
satisfying
\begin{alignat}{1}
 \footnoteinset{0.06}{0.3}{\eqref{eq:Kan_pointwise_cocone_theta_sup}}{%
 \InsertPDF{Kan_pointwise_cocone_theta2.pdf}} \raisebox{1em}{.}
 \label{eq:Kan_pointwise_cocone_theta2}
\end{alignat}
Furthermore, consider the case of $\cD = \cC$ and $K = \id_\cC$,
and rewrite the object $d$ in $\cC$ as $c$.
$\braket{X,\id_X}$ is a right Kan extension of $X$ along $\id_{\cC^\op}$
and is pointwise since $\Set$ is complete (see Corollary~\ref{cor:KanPointwiseCocomplete}).
Thus, from the dual of Eq.~\eqref{eq:Kan_left_pointwise_univ}, there exists
a unique $\ol{\alpha} \in \Set(\{*\},Xc) = Xc$ satisfying
\begin{alignat}{1}
 \InsertPDF{Kan_pointwise_cocone_theta_1X.pdf} \raisebox{1em}{,}
 \label{eq:Kan_pointwise_cocone_theta_1X}
\end{alignat}
where the blue circle represents $\theta^c$.
Therefore, $\ol{\alpha} \in Xc$ and $\tilde{\alpha} \colon \yonedaop{c} \nto X$ that satisfy
\begin{alignat}{1}
 \footnoteinset{-1.06}{0.3}{\eqref{eq:Kan_pointwise_cocone_theta_1X}}{%
 \footnoteinset{1.85}{0.3}{\eqref{eq:Kan_pointwise_cocone_theta2}}{%
 \InsertPDF{Kan_pointwise_cocone_theta_1X_yoneda.pdf}}}
 \label{eq:Kan_pointwise_cocone_theta_1X_yoneda}
\end{alignat}
correspond one-to-one.
Applying $\braket{c,\id_c} \in (\id_\cC \comma c)^\op$ from the right
to the left- and right-hand sides of Eq.~\eqref{eq:Kan_pointwise_cocone_theta_1X_yoneda}
and using $\theta^c_{\braket{c,\id_c}} = \id_c$
(see Eq.~\eqref{eq:Kan_left_pointwise_comma_c1}) yields
\begin{alignat}{1}
 \InsertPDF{Kan_pointwise_cocone_theta_1X_yoneda2.pdf} \raisebox{1em}{,}
 \label{eq:Kan_pointwise_cocone_theta_1X_yoneda2}
\end{alignat}
where the black circle represents $\id_c$.
The invertible map $\hat{\cC}(\yonedaop{c},X) \ni \tilde{\alpha} \mapsto \ol{\alpha} \in Xc$
obtained by this equation is clearly equal to the Yoneda map $\alpha_{X,c}$
defined by Eq.~\eqref{eq:repr_yonedaop_alpha}.
\begin{supplemental}
 Comparing Eqs.~\eqref{eq:Kan_pointwise_cocone_theta_sup} and
 \eqref{eq:Kan_pointwise_cocone_theta_1X_yoneda2}, it can be seen that $P_d$ corresponds to
 $c$ and $\theta^d$ corresponds to $\id_c$.
 Also, it can be seen that the map $\Phi_e \colon \tilde{\alpha} \mapsto \alpha$
 (which is the map $\Phi_e$ mentioned in the proof of Lemma~\ref{lemma:KanPointwiseCoconeCong})
 corresponds to the Yoneda map $\alpha_{X,c}$.
 In mathematical notation, these maps are expressed as
 \begin{alignat}{3}
  \hat{\cC}(\cD(K\Endash,d),\cE(F\Endash,e)) &\ni \tilde{\alpha}
  &&\qquad\xmapsto{\Phi_e}\qquad &(\tilde{\alpha} \b P_d) \theta^d &\in \Cocone(F \b P_d,e),
  \nonumber \\
  \hat{\cC}(\yonedaop{c},X) &\ni \tilde{\alpha}
  &&\qquad\xmapsto{\alpha_{X,c}}\qquad &(\tilde{\alpha} \b c) \id_c &\in Xc,
 \end{alignat}
 where note that $(\tilde{\alpha} \b c) \id_c = \tilde{\alpha}_c(\id_c)$ holds.
 Also, Eq.~\eqref{eq:Kan_pointwise_cocone_theta2} asserts that
 $\braket{\cD(K\Endash,d), \theta^d}$ is a left Kan extension of
 $\{*\} \b {!} \colon (K \comma d)^\op \to \Set$ along $P_d \colon (K \comma d)^\op \to \cC^\op$.
 This Kan extension is also pointwise since $\Set$ is cocomplete.
 Applying $\braket{c,p} \in K \comma d$ from the right to both sides of
 Eq.~\eqref{eq:Kan_pointwise_cocone_theta2} gives
 $\alpha_{\braket{c,p}} = \tilde{\alpha}_c(p)$, and thus
 $\tilde{\alpha} = \{ \cD(Kc,d) \ni p \mapsto \alpha_{\braket{c,p}} \in Xc \}_{c \in \cC}$
 holds.
 Note that the right-hand side of Eq.~\eqref{eq:Kan_pointwise_cocone_Phi} can be depicted by
 \begin{alignat}{1}
  \InsertPDF{Kan_pointwise_cocone_theta_Phi.pdf}
  \label{eq:Kan_pointwise_cocone_theta_Phi}
 \end{alignat}
 (where the black and blue circles represent $\id_d$ and $\theta^d$, respectively),
 which corresponds to the ``upside-down'' version of the equation on the right-hand side of
 the symbol ``$\myLeftrightarrow{equivalent}$'' in Eq.~\eqref{eq:Kan_pointwise_cocone_theta_sup}.
\end{supplemental}

\subsection{Basic properties of Kan extensions} \label{subsec:Kan_property}

\subsubsection{Functors preserving Kan extensions}

As with the preservation of limits, we can consider the preservation of Kan extensions as follows.
Consider a functor $S \colon \cE \to \cF$.
When there exists a left Kan extension of $F \colon \cC \to \cE$ along $K \colon \cC \to \cD$,
denoted by $\braket{\Lan_K F,\eta}$,
we say that $S$ \termdef{preserves} the left Kan extension $\braket{\Lan_K F,\eta}$ if
$\braket{S \b \Lan_K F, S \b \eta}$ is a left Kan extension of $S \b F$ along $K$.
Also, we call $S$ preserving any left Kan extension if it preserves
the left Kan extensions of any functor $F \colon \cC \to \cE$
along any functor $K \colon \cC \to \cD$, where $\cC$ and $\cD$ are arbitrary.
The preservation of right Kan extensions is defined as its dual.

Assume that there exists a left Kan extension $\braket{\Lan_K F,\eta}$ of $F$ along $K$
and that $S$ preserves this Kan extension; then, for any natural transformation
$\sigma \colon S \b F \nto H \b K$ (where $H \colon \cD \to \cF$ is arbitrary),
there exists a unique $\ol{\sigma} \colon S \b \Lan_K F \nto H$ satisfying
\begin{alignat}{1}
 \InsertPDF{Kan_left_preserve_univ.pdf} \raisebox{1em}{,}
 \label{eq:Kan_left_preserve_univ}
\end{alignat}
where the blue circle represents $\eta$.
This equation corresponds to Eqs.~\eqref{eq:Kan_left_universal} and \eqref{eq:Kan_left_universal2}.
The preservation of the left Kan extension $\braket{\Lan_K F,\eta}$ by $S$ can be
represented by
\begin{alignat}{1}
 \InsertPDF{Kan_left_preserve.pdf} \raisebox{1em}{,}
 \label{eq:Kan_left_preserve}
\end{alignat}
where for any arbitrarily chosen left Kan extension $\braket{\Lan_K (S \b F), \eta'}$
of $S \b F$ along $K$, $\eta'$ is represented by the blue ellipse on the right-hand side.
The diamond block represents the isomorphism $S \b \Lan_K F \cong \Lan_K (S \b F)$.

The following proposition holds:
\begin{proposition}{}{KanAdjPreserve}
 Any left adjoint preserves any left Kan extension.
 Dually, any right adjoint preserves any right Kan extension.
\end{proposition}
Since colimits are left Kan extensions and limits are right Kan extensions as mentioned in
Example~\ref{ex:KanLimit}, this proposition can be regarded as a generalization of
Theorem~\ref{thm:LimitAdjPreserve}.
\begin{proof}
 We show that for any adjunction $S \dashv T \colon \cE \to \cF$,
 $S$ preserves any left Kan extension, denoted by $\braket{\Lan_K F,\eta}$,
 of a functor $F \colon \cC \to \cE$ along a functor $K \colon \cC \to \cD$.
 We have for any natural transformation $\sigma \colon S \b F \nto H \b K$
 (where $H \colon \cD \to \cF$ is arbitrary),
 \begin{alignat}{1}
  \footnoteinset{-3.36}{0.3}{\eqref{eq:adj_conj}}{%
  \footnoteinset{-0.59}{0.3}{\eqref{eq:Kan_left_universal}}{%
  \footnoteinset{2.72}{0.3}{\eqref{eq:adj_conj}}{%
  \InsertPDF{Kan_left_preserve_adj.pdf}}}} \raisebox{1em}{,}
  \label{eq:Kan_left_preserve_adj}
 \end{alignat}
 where, in the first equality, we set the transpose of $\sigma$ to $\ol{\sigma}$
 (note that $S \b \Endash \dashv T \b \Endash$ holds).
 The second equality follows from applying Eq.~\eqref{eq:Kan_left_universal} to $\ol{\sigma}$
 (where we let $\ol{\tau}$ be $\ol{\sigma}$ on the right-hand side of Eq.~\eqref{eq:Kan_left_universal}).
 The blue circle represents $\eta$.
 In the last equality, we set the transpose of $\ol{\tau}$ to $\tau$.
 Since $\tau$ is uniquely determined from $\sigma$,
 $\braket{S \b \Lan_K F,S \b \eta}$ is a left Kan extension of $S \b F$ along $K$.
\end{proof}

\subsubsection{Conditions for having a pointwise Kan extension}

\begin{cor}{}{KanPointwiseCocomplete}
 Consider two functors $K \colon \cC \to \cD$ and $F \colon \cC \to \cE$,
 where $\cC$ is a small category.
 If $\cE$ is cocomplete, then a pointwise left Kan extension of $F$ along $K$ exists.
 Dually, if $\cE$ is complete, then a pointwise right Kan extension of $F$ along $K$ exists.
\end{cor}
\begin{proof}
 Since $\cC$ is small, for any $d \in \cD$, the comma category $K \comma d$ is small.
 Thus, if $\cE$ is cocomplete, then for each $d \in \cD$, a colimit of $F \b P_d$ exists.
 Therefore, a pointwise left Kan extension of $F$ along $K$ obviously exists
 from Theorem~\ref{thm:KanPointwise}.
\end{proof}

\begin{lemma}{}{KanPointwisePreserve}
 If a functor $S \colon \cE \to \cF$ preserves any colimit, then $S$ preserves
 any pointwise left Kan extension of any $F \colon \cC \to \cE$ along any $K \colon \cC \to \cD$.
\end{lemma}
\begin{proof}
 Let $\braket{L,\eta}$ be any pointwise left Kan extension of $F$ along $K$.
 For each $d \in \cD$, let $P_d$ be the forgetful functor from $K \comma d$ to $\cC$,
 and let $\mu^d \coloneqq (L \b \theta^d) (\eta \b P_d)$ (see Eq.~\eqref{eq:Kan_left_pointwise0});
 then, by Corollary~\ref{cor:KanPointwiseNas0}, $\braket{Ld,\mu^d}$ is a colimit of $F \b P_d$.
 Since $S$ preserves any colimit, $\braket{S \b Ld,S \b \mu^d}$ is a colimit of $S \b F \b P_d$.
 This colimit cocone $S \b \mu^d$ is depicted by
 \begin{alignat}{1}
  S \b \mu^d &\qquad\diagram\qquad
  \InsertMidPDF{Kan_left_pointwise_preserve_mu.pdf}
  \label{eq:Kan_left_pointwise_preserve_mu}
 \end{alignat}
 and satisfies
 \begin{alignat}{1}
  \lefteqn{ \{ S \b \mu^{Kc}_{\braket{c,\id_Kc}} \}_{c \in \cC} = S \b \eta } \nonumber \\
  &\diagram\qquad
  \InsertMidPDF{Kan_left_pointwise_preserve_mu_eta.pdf}.
  \label{eq:Kan_left_pointwise_preserve_mu_eta}
 \end{alignat}
 Thus, by Theorem~\ref{thm:KanPointwise}, $\braket{S \b L,S \b \eta}$ is
 a pointwise left Kan extension of $S \b F$ along $K$.
 Therefore, $S$ preserves $\braket{L,\eta}$.
\end{proof}

\begin{proposition}{}{KanPointwiseNas}
 Consider three functors $K \colon \cC \to \cD$, $F \colon \cC \to \cE$,
 and $L \colon \cD \to \cE$.
 The following conditions are equivalent:
 \begin{enumerate}
  \item $L$ is a pointwise left Kan extension of $F$ along $K$.
  \item $L$ is a left Kan extension of $F$ along $K$, and for any $e \in \cE$,
        $\yonedaop{e} \colon \cE \to \Set^\op$ preserves the left Kan extension $L$.
  \item There exists an isomorphism
        \begin{alignat}{1}
         \cE(Ld,e) &\cong \hat{\cC}(\cD(K\Endash,d),\cE(F\Endash,e))
         \label{eq:Kan_left_pointwise_proof_Lde_cong}
        \end{alignat}
        that is natural in $d \in \cD$ and $e \in \cE$.
 \end{enumerate}
 Dually, the following conditions are equivalent:
 \begin{enumerate}
  \item $R$ is a pointwise right Kan extension of $F$ along $K$.
  \item $R$ is a right Kan extension of $F$ along $K$, and for any $e \in \cE$,
        $\yoneda{e} \colon \cE \to \Set$ preserves the right Kan extension $R$.
  \item There exists an isomorphism
        \begin{alignat}{1}
         \cE(e,Rd) &\cong \Func{\cC}{\Set}(\cD(d,K\Endash),\cE(e,F\Endash))
         \label{eq:Kan_right_pointwise_proof_eRd_cong}
        \end{alignat}
        that is natural in $d \in \cD$ and $e \in \cE$.
 \end{enumerate}
\end{proposition}
\begin{proof}
 $(1) \Rightarrow (2)$:
 From Corollary~\ref{cor:LimitPresheaf2}, $\yonedaop{e}$ preserves any colimit.
 Therefore, the case where $S$ is $\yonedaop{e}$ in Lemma~\ref{lemma:KanPointwisePreserve}
 makes it appear.
 
 $(2) \Rightarrow (3)$:
 We have for any $\sigma \in \hat{\cC}(\cD(K\Endash,d),\cE(F\Endash,e))$,
 \begin{alignat}{1}
  \InsertMidPDF{Kan_left_pointwise_proof_Lde.pdf}
  &\quad\xmapsto{\eqref{eq:Kan_left_preserve_univ}}\quad
  \InsertMidPDF{Kan_left_pointwise_proof_Lde2a.pdf}
  \quad\xmapsto{\eqref{eq:repr_yonedaop_alpha}}\quad
  \InsertMidPDF{Kan_left_pointwise_proof_Lde2b.pdf}.
  \label{eq:Kan_left_pointwise_proof_Lde}
 \end{alignat}
 Note that instead of $\yonedaop{e} \colon \cE \to \Set^\op$,
 its dual $\yonedaop{e} \colon \cE^\op \to \Set$ is used (the same applies to $\yonedaop{d}$),
 and thus some parts of the diagram are ``upside-down''.
 The right ``$\mapsto$'' is the Yoneda map mentioned in Corollary~\ref{cor:Yonedaop}.
 The last expression (i.e., $\ol{\sigma}_d(\id_d)$) is an element of
 $\yonedaop{e} \b Ld = \cE(Ld,e)$.
 Thus, the isomorphism in Eq.~\eqref{eq:Kan_left_pointwise_proof_Lde_cong} is obtained.
 Naturality in $d$ and $e$ can be easily seen from Eq.~\eqref{eq:Kan_left_pointwise_proof_Lde};
 indeed, the structure of the equation will not change even if
 $\yonedaop{f}$ (where $f \in \cD(d',d)$ is arbitrary) is applied from
 below to the position of the wire $\yonedaop{d}$, and even if
 $\yonedaop{g}$ (where $g \in \cE(e,e')$ is arbitrary) is applied from
 above to the position of the wire $\yonedaop{e}$.

 $(3) \Rightarrow (1)$:
 For each $d \in \cD$, $\cE(Ld,\Endash) \cong \Cocone(F \b P_d,\Endash)$ holds
 from Lemma~\ref{lemma:KanPointwiseCoconeCong}, and thus,
 from Lemma~\ref{lemma:LimitRepr}, $F \b P_d$ has a colimit.
 Therefore, from Theorem~\ref{thm:KanPointwise}, there exists
 a pointwise left Kan extension $\Lan_K F$.
 This yields an isomorphism
 \begin{alignat}{1}
  \cE(Ld,e) \cong \hat{\cC}(\cD(K\Endash,d),\cE(F\Endash,e)) \cong \cE((\Lan_K F)d,e)
 \end{alignat}
 that is natural in $d \in \cD$ and $e \in \cE$.
 The right isomorphism follows from the fact that
 Eq.~\eqref{eq:Kan_left_pointwise_proof_Lde_cong} with $L$ replaced by $\Lan_K F$ holds,
 which is obtained by ``$(1) \Rightarrow (3)$''.
 Therefore, since $L \cong \Lan_K F$ from Lemma~\ref{lemma:YonedaopCongFunc},
 $L$ is a pointwise left Kan extension of $F$ along $K$.
\end{proof}

We give another proof of ``$(1) \Rightarrow (3)$'' in Proposition~\ref{pro:KanPointwiseNas}.
Let $L$ be a pointwise left Kan extension of $F$ along $K$; then,
by mapping both sides of Eq.~\eqref{eq:Kan_left_pointwise_univ} with the map $\Psi_e$
of Eq.~\eqref{eq:Kan_pointwise_cocone_alpha_diagram}, it can be seen that
for any natural transformation $\tilde{\alpha} \in \hat{\cC}(\cD(K\Endash,d),\cE(F\Endash,e))$
(where $d \in \cD$ and $e \in \cE$ are arbitrary),
there exists a unique morphism $\ol{\alpha} \in \cE(Ld,e)$ satisfying
\begin{alignat}{1}
 \InsertPDF{Kan_left_pointwise_univ_yoneda.pdf} \raisebox{1em}{,}
 \label{eq:Kan_left_pointwise_univ_yoneda}
\end{alignat}
where the blue circle represents $\eta$.
The area enclosed by the auxiliary line corresponds to a colimit cocone of $F \b P_d$.
Equation~\eqref{eq:Kan_left_pointwise_univ_yoneda} is essentially the same as
Eq.~\eqref{eq:Kan_left_pointwise_univ}.
Note that it follows from Eq.~\eqref{eq:Kan_left_pointwise_univ_yoneda} that it satisfies naturality in $e$
since applying a morphism in $\cE$ from above to the wire $e$ will not change
the structure of the equation.
Naturality in $d$ is also easily shown by Eq.~\eqref{eq:Kan_left_pointwise_univ_yoneda}.

Equation~\eqref{eq:Kan_left_pointwise_proof_Lde_cong} can be intuitively understood by representing
it as (see also Eq.~\eqref{eq:Kan_pointwise_cocone_cong_picture})
\begin{alignat}{1}
 \InsertPDF{Kan_left_pointwise_univ_yoneda_func.pdf} \raisebox{1em}{,}
 \label{eq:Kan_left_pointwise_univ_yoneda_func}
\end{alignat}
where the central expression should be thought of as a functor that
maps $\braket{d,e} \in \cD^\op \times \cE$ to $\hat{\cC}(\cD(K\Endash,d), \cE(F\Endash,e))$.
Note that in the case of $\cD = \cOne$ and $K = {!}$,
$L$ is a colimit of $F$ as mentioned in Example~\ref{ex:KanLimit},
in which case the first isomorphism of Eq.~\eqref{eq:Kan_left_pointwise_univ_yoneda_func}
corresponds to Eq.~\eqref{eq:limit_repr_cong_op}.

\begin{ex}{}{KanLimitPointwise}
 Any colimit is a pointwise left Kan extension.
 Indeed, as stated in Example~\ref{ex:KanLimit}, a colimit of $F \colon \cC \to \cE$
 is a left Kan extension of $F$ along ${!}$, and from Corollary~\ref{cor:LimitPresheaf2},
 any $\yonedaop{e} \colon \cE \to \Set^\op$ preserves any colimit,
 satisfying Condition~(2) of Proposition~\ref{pro:KanPointwiseNas}.
\end{ex}

A left (or right) Kan extension $L$ of $F \colon \cC \to \cE$ along $K \colon \cC \to \cD$
is called \termdef{absolute} if
any functor $H \colon \cE \to \cF$ (where $\cF$ is arbitrary) preserves $L$.
Any absolute Kan extension is pointwise since this condition is stronger than
Condition~(2) of Proposition~\ref{pro:KanPointwiseNas}.

\subsubsection{Dense functors} \label{subsubsec:Kan_property_dense}

Let us consider a functor $K \colon \cC \to \cD$.
When the identity functor $\id_\cD$ is a pointwise left Kan extension of $K$ along $K$
(or equivalently, when a pointwise left Kan extension $\Lan_K K$ exists and is naturally
isomorphic to $\id_\cD$),
$K$ is called \termdef{dense}.
As a dual concept, when $\id_\cD$ is a pointwise right Kan extension of $K$ along $K$,
$K$ is called \termdef{codense}.

\begin{proposition}{}{KanDense}
 For a functor $K \colon \cC \to \cD$, the following are equivalent:
 \begin{enumerate}
  \item $K$ is dense.
  \item $\braket{\id_\cD,\id_K}$ is a pointwise left Kan extension of $K$ along $K$.
  \item For each $d \in \cD$, $\braket{d,\theta^d}$ is a colimit of $K \b P_d$
        (where $\theta^d$ is the canonical cocone for $K \comma d$).
 \end{enumerate}
\end{proposition}
\begin{proof}
 $(1) \Rightarrow (2)$:
 Since $K$ is dense, there exists some $\eta \colon K \nto K$ such that $\braket{\id_\cD,\eta}$
 is a pointwise left Kan extension of $K$ along $K$.
 It suffices to show that $\eta$ is a natural isomorphism, i.e., each component
 $\eta_c$ $~(c \in \cC)$ is invertible.
 In fact, in this case, the natural transformation obtained by applying $\eta^{-1}$ to $\eta$,
 i.e., $\id_K$, is also a universal morphism,
 so $\braket{\id_\cD,\id_K}$ is a pointwise left Kan extension of $K$ along $K$.

 Let us arbitrarily choose $c \in \cC$.
 Since $\braket{\id_\cD,\eta}$ is a pointwise left Kan extension,
 there exists a unique $\phi \in \cD(Kc,Kc)$ satisfying
 \begin{alignat}{1}
  \InsertPDF{Kan_dense_Kc_univ.pdf} \raisebox{1em}{,}
  \label{eq:Kan_dense_Kc_univ}
 \end{alignat}
 where the blue circle represents the canonical cocone $\theta^{Kc}$ for $K \comma Kc$.
 The area enclosed by the auxiliary lines is denoted by $\mu$,
 and the first equality follows from the universality of $\mu$
 (note that $\mu$ is a colimit cocone of $K \b P_{Kc}$ by Corollary~\ref{cor:KanPointwiseNas0}).
 Applying $\braket{c,\id_{Kc}} \in K \comma Kc$ from the right to the left and central expressions
 gives
 \begin{alignat}{1}
  \InsertPDF{Kan_dense_Kc_univ_id.pdf} \raisebox{1em}{.}
  \label{eq:Kan_dense_Kc_univ_id}
 \end{alignat}
 Also, we have
 \begin{alignat}{1}
  \footnoteinset{-3.16}{0.97}{\eqref{eq:Kan_dense_Kc_univ}}{%
  \footnoteinset{-0.26}{0.97}{\eqref{eq:Kan_left_pointwise_comma_fp}}{%
  \footnoteinset{2.79}{0.97}{\eqref{eq:Kan_dense_Kc_univ}}{%
  \footnoteinsets{2.79}{-1.3}{\eqref{eq:Kan_left_pointwise_comma_fp}}{\eqref{eq:Kan_dense_Kc_univ_id}}{%
  \InsertPDF{Kan_dense_Kc_univ_eta.pdf}}}}} \raisebox{1em}{.}
  \label{eq:Kan_dense_Kc_univ_eta}
 \end{alignat}
 Applying the universality of $\mu$ to the first and last expressions gives $\eta_c \phi = \id_{Kc}$.
 Therefore, from this equation and Eq.~\eqref{eq:Kan_dense_Kc_univ_id}, $\eta_c$ is invertible,
 and $\phi$ is its inverse.

 $(2) \Rightarrow (3)$:
 From Corollary~\ref{cor:KanPointwiseNas0} and $(\id_\cD \b \theta^d) (\id_K \b P_d) = \theta^d$,
 $\braket{d,\theta^d}$ is a colimit of $K \b P_d$.

 $(3) \Rightarrow (1)$:
 From Theorem~\ref{thm:KanPointwise}, there exists a unique functor $L \colon \cD \to \cD$ satisfying
 \begin{alignat}{1}
  \InsertPDF{Kan_dense_proof_theta.pdf} \raisebox{1em}{,}
  \label{eq:Kan_dense_proof_theta}
 \end{alignat}
 where the left and right blue circles represent the canonical cocones $\theta^{d'}$ and $\theta^d$
 for $K \comma d'$ and $K \comma d$, respectively,
 and $L$ is a pointwise left Kan extension of $K$ along $K$.
 Also, from Eq.~\eqref{eq:Kan_left_pointwise_comma_fp},
 the equation obtained by substituting $L = \id_\cD$ into Eq.~\eqref{eq:Kan_dense_proof_theta} holds,
 so the uniquely determined $L$ is $\id_\cD$.
 Therefore, $\id_\cD$ is a pointwise left Kan extension of $K$ along $K$, and thus $K$ is dense.
\end{proof}

From this proposition, for a dense functor $K$, $\theta^d$ with each $d \in \cD$ is
a colimit cocone of $K \b P_d$.
Thus, from the dual of Eq.~\eqref{eq:limit_cong}, any colimit cocone of $K \b P_d$ can be expressed
in the following form
\begin{alignat}{1}
 \InsertPDF{Kan_dense.pdf} \raisebox{1em}{,}
 \label{eq:Kan_dense}
\end{alignat}
where the blue circle is $\theta^d$.
The diamond block represents the isomorphism $d \cong \colim (K \b P_d)$.

If $K$ is dense, then from Eq.~\eqref{eq:Kan_left_pointwise_proof_Lde_cong} with
$F = K$ and $L = \id_\cD$, there exists an isomorphism
\begin{alignat}{1}
 \hat{\cC}(\cD(K\Endash,d),\cD(K\Endash,d')) \cong \cD(d,d')
 \label{eq:kan_dense_cong}
\end{alignat}
that is natural in $d \in \cD$ and $d' \in \cD$.
Also, since, from Proposition~\ref{pro:KanDense}, $\braket{\id_\cD,\id_K}$ is
a pointwise left Kan extension of $K$ along $K$,
the equation obtained by substituting $L = \id_\cD$ and $\eta = \id_K$ into
Eq.~\eqref{eq:Kan_left_pointwise_univ_yoneda} yields that
for any $\tau \in \hat{\cC}(\cD(K\Endash,d),\cD(K\Endash,d'))$,
there exists a unique morphism $\ol{\tau} \in \cD(d,d')$ satisfying
\begin{alignat}{1}
 \InsertPDF{Kan_dense_univ.pdf} \raisebox{1em}{.}
 \label{eq:Kan_dense_univ}
\end{alignat}
The one-to-one correspondence between $\tau$ and $\ol{\tau}$ obtained from
Eq.~\eqref{eq:Kan_dense_univ} gives the isomorphism in Eq.~\eqref{eq:kan_dense_cong}.

\subsubsection{Relationship with adjoints} \label{subsubsec:Kan_property_adj}

We have already mentioned in Example~\ref{ex:KanLimitPointwise} that
limits and colimits are special cases of pointwise Kan extensions.
Here, we show that left and right adjoints are also special cases of pointwise Kan extensions.
In particular, they are absolute Kan extensions.
We schematically represent these relationships in the following diagram:
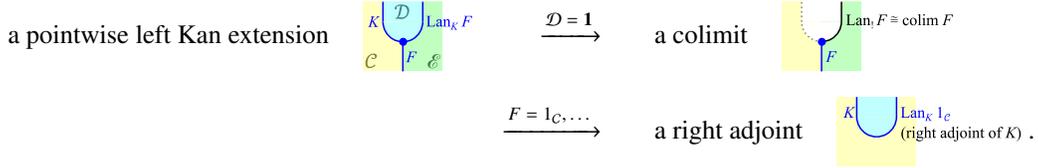
\begin{alignat}{1}
 \text{a pointwise left Kan extension}\quad \InsertMidPDF{Kan_all_concepts_Kan.pdf}
 \qquad\xrightarrow{\text{$\cD = \cOne$}}\qquad
 &\text{a colimit}\quad \InsertMidPDF{Kan_all_concepts_colim.pdf} \nonumber \\
 \qquad\xrightarrow{\text{$F = \id_\cC, \dots$}}\qquad
 &\text{a right adjoint}\quad \InsertMidPDF{report_Kan_all_concepts_adj.pdf}.
 \label{eq:Kan_all_concepts}
\end{alignat}

\begin{proposition}{}{KanAdj}
 For any adjunction $F \dashv G \colon \cC \to \cD$ and its unit $\eta$,
 $\braket{G,\eta}$ is an absolute left Kan extension of $\id_\cC$ along $F$.
\end{proposition}
\begin{proof}
 For any $H \colon \cD \to \cE$, $L \colon \cC \to \cE$ (where $\cE$ is arbitrary),
 and $\tau \colon L \nto H \b F$, there exists a unique $\ol{\tau} \colon L \b G \nto H$
 satisfying
 \begin{alignat}{1}
  \InsertPDF{Kan_adj_univ_preserve.pdf} \raisebox{1em}{;}
  \label{eq:Kan_adj_univ_preserve}
 \end{alignat}
 indeed, $\ol{\tau}$ is uniquely determined by
 \begin{alignat}{1}
  \InsertPDF{Kan_adj_univ_preserve2.pdf} \raisebox{1em}{.}
  \label{eq:Kan_adj_univ_preserve2}
 \end{alignat}
 Therefore, by the definition of a left Kan extension, $\braket{L \b G,L \b \eta}$
 is a left Kan extension of $L$ along $F$.
 Note that the unit $L \b \eta$ of this Kan extension is enclosed by the auxiliary line
 in Eq.~\eqref{eq:Kan_adj_univ_preserve}.
 In particular, by substituting $L = \id_\cC$, $\braket{G,\eta}$ is a left Kan extension
 of $\id_\cC$ along $F$.
 This Kan extension is absolute since any $L$ preserves it.
\end{proof}

From this proposition, the right adjoint $G$ to $F$, its unit $\eta$,
and the left Kan extension $\braket{\Lan_F \id_{\cC},\eta'}$ are related as follows:
\begin{alignat}{1}
 \InsertPDF{Kan_adj.pdf} \raisebox{1em}{,}
 \label{eq:Kan_adj}
\end{alignat}
where the blue circle represents $\eta'$, and the diamond block represents
$G \cong \Lan_F \id_{\cC}$.
Thus, $\braket{G,\eta}$ and $\braket{\Lan_F \id_{\cC},\eta'}$ can be identified.

The following proposition shows that the converse of Proposition~\ref{pro:KanAdj} also holds
(note that we have relaxed the conditions).
\begin{proposition}{}{KanAdjInv}
 If $\braket{G,\eta}$ is a left Kan extension of $\id_\cC$ along
 a functor $F \colon \cC \to \cD$, and if $F$ preserves this left Kan extension,
 then $G$ is a right adjoint to $F$, and $\eta$ is the unit of this adjunction.
\end{proposition}
\begin{proof}
 From Theorem~\ref{thm:AdjNasSnake}, it suffices to show that there exists a natural transformation
 $\varepsilon \colon F \b G \nto \id_\cD$ such that $\eta$ and $\varepsilon$ satisfy
 Eq.~\eqref{eq:zigzag}.
 Since $F$ preserves the left Kan extension $\braket{G,\eta}$,
 $\braket{F \b G,F \b \eta}$ is a left Kan extension of $F$ along $F$.
 By the universal property of this Kan extension, there exists $\varepsilon$ satisfying
 \begin{alignat}{1}
  \InsertPDF{Kan_adj_inv1.pdf} \raisebox{1em}{,}
  \label{eq:Kan_adj_inv1}
 \end{alignat}
 where the unit $F \b \eta$ of this Kan extension is enclosed by the auxiliary line.
 Also, we obtain
 \begin{alignat}{1}
  \footnoteinset{1.92}{0.3}{\eqref{eq:Kan_adj_inv1}}{%
  \InsertPDF{Kan_adj_inv2.pdf}} \raisebox{1em}{.}
  \label{eq:Kan_adj_inv2}
 \end{alignat}
 By the universal property of the left Kan extension $\braket{G,\eta}$,
 the two natural transformations enclosed by the auxiliary lines must be equal.
 From this fact and Eq.~\eqref{eq:Kan_adj_inv1}, Eq.~\eqref{eq:zigzag} holds.
\end{proof}

\subsubsection{Monads derived from Kan extensions} \label{subsubsec:Kan_property_monad}

Assume that for a functor $K \colon \cC \to \cD$, there exists a right Kan extension
$\braket{T,\varepsilon}$ of $K$ along $K$.
Let us define $\mu \colon T \b T \nto T$ and $\eta \colon \id_\cD \nto T$ to satisfy
\begin{alignat}{1}
 \InsertPDF{Kan_monad_def.pdf} \raisebox{1em}{,}
 \label{eq:Kan_monad_def}
\end{alignat}
where the blue circles represent $\varepsilon$, and the hollow circles
on the left- and right-hand sides of the comma represent $\mu$ and $\eta$, respectively.
Note that for simplicity, we omit the label of the wire representing the functor
$T \colon \cD \to \cD$ in diagrams.
By the universal property of $\braket{T,\varepsilon}$ (see Eq.~\eqref{eq:Kan_right_universal}),
$\mu$ and $\eta$ are uniquely determined.
$\braket{T,\mu,\eta}$ is called a \termdef{codensity monad} of $K$.

\begin{proposition}{}{}
 Any codensity monad is a monad%
 \footnote{$\braket{T,\mu,\eta}$ is called a \termdef{monad} if it satisfies the associativity law,
 i.e., $\mu(\id_T \b \mu) = \mu(\mu \b \id_T)$, and the unit law,
 i.e., $\mu(\id_T \b \eta) = \id_T = \mu(\eta \b \id_T)$.}.
\end{proposition}
\begin{proof}
 Let $\braket{T,\mu,\eta}$ be a codensity monad of $K$.
 From
 \begin{alignat}{1}
  \InsertPDF{Kan_monad_composite.pdf} \raisebox{1em}{,}
  \label{eq:Kan_monad_composite}
 \end{alignat}
 the universal property of $\braket{T,\varepsilon}$ gives the following associativity law:
 \begin{alignat}{1}
  \mu(\id_T \b \mu) = \mu(\mu \b \id_T)
  &\qquad\diagram\qquad
  \InsertMidPDF{Kan_monad_composite2.pdf}.
  \label{eq:Kan_monad_composite2}
 \end{alignat}
 Also,
 \begin{alignat}{1}
  \InsertPDF{Kan_monad_ident.pdf}
  \label{eq:Kan_monad_ident}
 \end{alignat}
 gives the following unit law:
 \begin{alignat}{1}
  \mu(\id_T \b \eta) = \id_T = \mu(\eta \b \id_T)
  &\qquad\diagram\qquad
  \InsertMidPDF{Kan_monad_ident2.pdf}.
  \label{eq:Kan_monad_ident2}
 \end{alignat}
 Therefore, $\braket{T,\mu,\eta}$ is a monad.
\end{proof}


\section{(Co)ends and weighted (co)limits} \label{sec:end}

In this section, we mainly discuss (co)ends and weighted (co)limits,
which can be viewed as generalizations of (co)limits in a certain sense.
These concepts are also closely related to pointwise Kan extensions.
We use string diagrams to illustrate some of the relationships between (co)limits,
pointwise Kan extensions, (co)ends, and weighted (co)limits.
String diagrams for (co)ends are also described in Refs.~\cite{Ril-2018,Rom-2020,Bra-Rom-2023}
and references cited therein.
We propose a new representation method that allows us to perform (co)end calculations intuitively.

\subsection{Diagrams for direct products} \label{subsec:end_diagram}

In this section, we mainly use the ``expression for direct products''
(recall Subsubsection~\ref{subsubsec:category_prod_prod}).
In this expression, for a morphism $f \in \cC(a,b)$ and a functor
$F \colon \cC \to \cD$, $Ff$ is represented by the right-hand side
of the following diagram:
\begin{alignat}{1}
 \InsertPDF{end_picture_func.pdf} \raisebox{1em}{.}
 \label{eq:end_picture_func}
\end{alignat}
For a morphism $g \in \cC^\op(b,a) = \cC(a,b)$ and a contravariant functor
$G \colon \cC^\op \to \cD$, in the expression for direct products,
we represent $Gg$ by the right-hand side of the following diagram:
\begin{alignat}{1}
 \InsertPDF{end_picture_func_op.pdf} \raisebox{1em}{.}
 \label{eq:end_picture_func_op}
\end{alignat}
In the expression for direct products, we represent objects and morphisms in the opposite category
$\cC^\op$ as ``upside-down'' versions of the corresponding objects and morphisms in $\cC$.
For convenience, we label the category $\cC^\op$ with ``$\cC$'' in such expressions,
but instead represent objects with downward-pointing arrows;
we can distinguish whether it represents $\cC$ or $\cC^\op$ from the direction of the arrows.
For two blue arrows representing a functor $F \colon \cC \to \cD$, its dual
$F \colon \cC^\op \to \cD^\op$ can be represented by vertically flipping the direction of the arrows.

Also, products are represented by placing elements side by side.
For example, the product $\braket{f,g} \in \mor (\cC^\op \times \cC)$
of $f \in \cC^\op(a,b)$ and $g \in \cC(c,d)$ is represented by
\begin{alignat}{1}
 \InsertPDF{end_picture_func2.pdf} \raisebox{1em}{.}
 \label{eq:end_picture_func2}
\end{alignat}
The gray dashed line representing the product ``$\times$'' can be omitted.

In summary, compared to the main expression, the expression for direct products has
the following main differences:
\begin{enumerate}
 \item It represents products (and tensor products, as will be shown in Subsection~\ref{subsec:end_tensor})
       by placing elements side by side.
 \item It represents applying a functor by drawing two parallel arrows to the left and right.
 \item It represents objects and functors with upward or downward arrows instead of wires.
       A category $\cC$ and its opposite category $\cC^\op$ are distinguished
       by the direction of arrows.
\end{enumerate}

We represent $\cC(a,b)$ as
\begin{alignat}{1}
 \cC(a,b) &\qquad\diagram\qquad
 \InsertMidPDF{end_picture_hom.pdf},
 \label{eq:end_picture_hom}
\end{alignat}
where the dotted arrow is a singleton set $\{ * \} \in \Set$.
Note that in this diagram, $\cC(a,b)$ is represented as a morphism from $\{ * \}$ to $\cC(a,b)$.
As shown on the right-hand side, dotted arrows representing $\{ * \}$ are often omitted.
The dark gray rectangle represents the hom-functor
$\cC(\Endash,\Enndash) \colon \cC^\op \times \cC \to \Set$,
with $a^\op \in \cC^\op$ represented as an upside-down version of $a \in \cC$ with a downward arrow.
For a contravariant functor $F \colon \cC^\op \to \cD$, $\cD(Fc,d)$ can be represented by
\begin{alignat}{1}
 \cD(Fc,d) &\qquad\diagram\qquad
 \InsertMidPDF{end_opposite_ex.pdf}.
 \label{eq:end_opposite_ex}
\end{alignat}
This diagram may be easier to understand if viewed as the mapping of $(Fc)^\op \in \cD^\op$
through $\cD(\Endash,d) \colon \cD^\op \to \Set$; the $Fc$ part is vertically flipped.
\begin{supplemental}
 This representation enables us to easily depict properties, such as
 ``returning to its original state after undergoing two upside-down flips''.
 Therefore, this representation will frequently prove valuable in subsequent discussions.
\end{supplemental}

We represent $\cC(a,\Endash)$ as the following two expressions:
\begin{alignat}{1}
 \cC(a,\Endash) &\qquad\diagram\qquad
 \InsertMidPDF{end_picture_hom_func.pdf}.
 \label{eq:end_picture_hom_func}
\end{alignat}
It could be more comprehensible if you think of it this way:
$\cC(a,\Endash)$ can be viewed as an element of $\Func{\cC}{\Set}$, as shown
on the left-hand side of the diagram,
and when viewed as a functor from $\cC$ to $\Set$, it is represented as shown on the right-hand side.

\subsection{Definition of (co)ends} \label{subsec:end_def}

\subsubsection{Wedges and ends}

\begin{define}{wedges}{Wedge}
 For an object $c \in \cC$ and a bifunctor $D \colon \cJ^\op \times \cJ \to \cC$,
 a collection of morphisms $\alpha \coloneqq \{ \alpha_j \colon c \to D(j,j) \}_{j \in \cJ}$
 in $\cC$ is called a \termdef{wedge} from $c$ to $D$, or simply a wedge to $D$,
 if it satisfies
 \begin{alignat}{1}
  \lefteqn{ D(i,h) \c \alpha_i = D(h,j) \c \alpha_j } \nonumber \\
  &
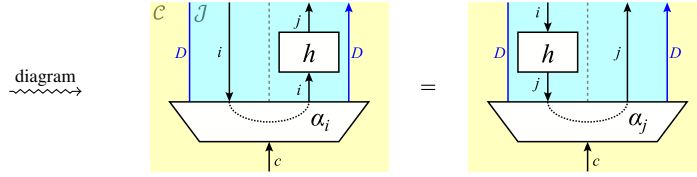
\diagram\qquad \InsertMidPDF{end_wedge.pdf}
  \label{eq:end_wedge}
 \end{alignat}
 for any morphism $h \in \cJ(i,j)$.
 In diagrams, we represent the component $\alpha_j$ of the wedge $\alpha$ as
 \begin{alignat}{1}
  \alpha_j &\qquad
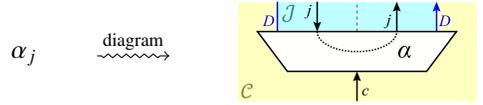
\diagram\qquad \InsertMidPDF{end_wedge_picture.pdf},
  \label{eq:end_wedge_picture}
 \end{alignat}
 where a dotted semicircle is drawn inside the block.
\end{define}

We write $\Wedge(c,D)$ for the collection of all wedges from $c$ to $D$.
For a bifunctor $D \colon \cJ^\op \times \cJ \to \cC$, let us consider the following category:
\begin{itemize}
 \item Its each object is a pair, $\braket{c,\alpha}$, of $c \in \cC$ and $\alpha \in \Wedge(c,D)$.
 \item Its each morphism from an object $\braket{c,\alpha}$ to an object $\braket{c',\alpha'}$
       is a morphism $f \in \cC(c,c')$ in $\cC$ that satisfies
       $\alpha_j = \alpha'_j f$ for all $j \in \cJ$.
 \item The composite of its morphisms is the composite of morphisms in $\cC$,
       and its identity morphism is the identity morphism in $\cC$.
\end{itemize}
This category is called the \termdef{category of wedges} to $D$, and we write $\Wedge_D$.
Roughly speaking, $\Wedge_D$ is a category whose objects are wedges to $D$.

\begin{define}{ends}{End}
 Consider a bifunctor $D \colon \cJ^\op \times \cJ \to \cC$.
 A terminal object $\braket{d,\kappa}$ in $\Wedge_D$ is called an \termdef{end} of $D$.
 In other words, a pair, $\braket{d,\kappa}$, of an object $d \in \cC$ and a wedge
 $\kappa \in \Wedge(d,D)$ is called an end of $D$ if
 for any wedge $\alpha \in \Wedge(c,D)$ (where $c \in \cC$ is also arbitrary),
 there exists a unique morphism $\ol{\alpha} \in \cC(c,d)$ satisfying
 \begin{alignat}{1}
  \lefteqn{ \alpha_j = \kappa_j \ol{\alpha} } \nonumber \\
  &
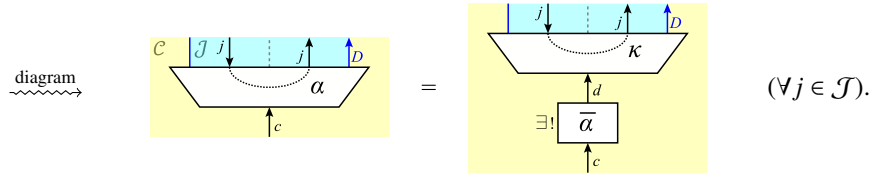
\diagram\qquad \InsertMidPDF{end_universal.pdf} \qquad (\forall j \in \cJ).
  \label{eq:end_universal}
 \end{alignat}
 Sometimes $d$ is simply called an end.
\end{define}

The end $d$ is often written as $\int_{j \in \cJ} D(j,j)$.
From the above universal property, $d \in \cD$ is an end of $D$ if and only if
\begin{alignat}{1}
 \cC(\Endash,d) &\cong \Wedge(\Endash,D)
\end{alignat}
holds, where $\Wedge(\Endash,D) \colon \cC^\op \to \Set$ is a presheaf
(in the case where $\Wedge(c,D)$ is a set for each $c \in \cC$) defined as follows:
\begin{itemize}
 \item It maps each object $c$ in $\cC^\op$ to $\Wedge(c,D)$.
 \item It maps each morphism $f \in \cC^\op(b,a) = \cC(a,b)$ to the map
       (i.e., the morphism in $\Set$)
       $\Wedge(b,D) \ni \{ \alpha_j \}_{j \in \cJ} \mapsto
       \{ \alpha_j f \}_{j \in \cJ} \in \Wedge(a,D)$.
\end{itemize}

\begin{ex}{limits are ends}{end_limit}
 For a functor $D' \colon \cJ \to \cC$, let $D \colon \cJ^\op \times \cJ \to \cC$
 be the bifunctor defined by $D(f,g) \coloneqq D'(g)$ $~(f \in \mor \cJ^\op, ~g \in \mor \cJ)$.
 Then, for any $c \in \cC$, a cone from $c$ to $D'$ is equivalent to a wedge from $c$ to $D$,
 i.e., $\Cone(c,D') = \Wedge(c,D)$ holds.
 We can also readily obtain $\Cone_{D'} = \Wedge_D$.
 Furthermore, a limit of $D'$ is equivalent to an end of $D$
 since the former is a terminal object in $\Cone_{D'}$ while the latter
 is a terminal object in $\Wedge_D$.
\end{ex}

\subsubsection{Cowedges and coends}

Cowedges and coends can be defined as the duals of wedges and ends.

\begin{define}{cowedges}{coWedge}
 For an object $c \in \cC$ and a bifunctor $D \colon \cJ^\op \times \cJ \to \cC$,
 a collection of morphisms $\alpha \coloneqq \{ \alpha_j \colon D(j,j) \to c \}_{j \in \cJ}$
 in $\cC$ is called a \termdef{cowedge} from $D$ to $c$, or simply a cowedge from $D$,
 if it satisfies
 \begin{alignat}{1}
  \lefteqn{ \alpha_i \c D(i,h) = \alpha_j \c D(h,j) } \nonumber \\
  &
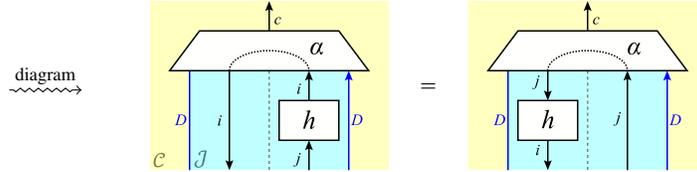
\diagram\qquad \InsertMidPDF{end_cowedge.pdf}
  \label{eq:end_cowedge}
 \end{alignat}
 for any morphism $h \in \cJ(j,i)$.
 In diagrams, we represent the component $\alpha_j$ of the cowedge $\alpha$ as
 \begin{alignat}{1}
  \alpha_j &\qquad
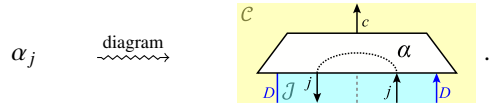
\diagram\qquad \InsertMidPDF{end_cowedge_picture.pdf}.
  \label{eq:end_cowedge_picture}
 \end{alignat}
\end{define}

We write $\Cowedge(D,c)$ for the collection of all cowedges from $D$ to $c$.
Also, the category defined as follows is called the \termdef{category of cowedges} from $D$,
and we write $\Cowedge_D$:
\begin{itemize}
 \item Its each object is a pair, $\braket{c,\alpha}$, of $c \in \cC$ and $\alpha \in \Cowedge(D,c)$.
 \item Its each morphism from an object $\braket{c,\alpha}$ to an object $\braket{c',\alpha'}$
       is a morphism $f \in \cC(c,c')$ in $\cC$ that satisfies
       $\alpha'_j = f \alpha_j$ for all $j \in \cJ$.
 \item The composite of its morphisms is the composite of morphisms in $\cC$,
       and its identity morphism is the identity morphism in $\cC$.
\end{itemize}

\begin{define}{coends}{coEnd}
 Consider a bifunctor $D \colon \cJ^\op \times \cJ \to \cC$.
 An initial object $\braket{d,\kappa}$ in $\Cowedge_D$ is called a \termdef{coend} of $D$.
 In other words, a pair, $\braket{d,\kappa}$, of an object $d \in \cC$ and a cowedge
 $\kappa \in \Cowedge(D,d)$ is called a coend of $D$ if
 for any cowedge $\alpha \in \Cowedge(D,c)$ (where $c \in \cC$ is also arbitrary),
 there exists a unique morphism $\ol{\alpha} \in \cC(d,c)$ satisfying
 \begin{alignat}{1}
  \lefteqn{ \alpha_j = \ol{\alpha} \kappa_j } \nonumber \\
  &
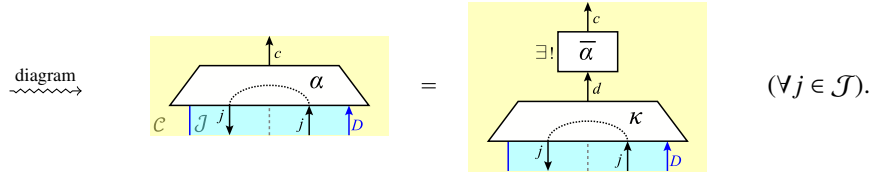
\diagram\qquad \InsertMidPDF{end_universal_op.pdf} \qquad (\forall j \in \cJ).
  \label{eq:end_universal_op}
 \end{alignat}
 Sometimes $d$ is simply called a coend.
\end{define}

The coend $d$ is often written as $\int^{j \in \cJ} D(j,j)$.
From the above universal property, $d \in \cD$ is a coend of $D$ if and only if
\begin{alignat}{1}
 \cC(d,\Endash) &\cong \Cowedge(D,\Endash)
 \label{eq:coend_cong}
\end{alignat}
holds, where $\Cowedge(D,\Endash) \colon \cC \to \Set$ is a functor
(in the case where $\Cowedge(D,c)$ is a set for each $c \in \cC$)
defined similarly to $\Wedge(\Endash,D)$.

\begin{ex}{colimits are coends}{}
 Considering the dual of Example~\ref{ex:end_limit},
 we can see that any cocone is a cowedge and any colimit is a coend.
\end{ex}

\subsubsection{Diagrams for ends}

For a bifunctor $D \colon \cJ^\op \times \cJ \to \cC$, we represent its end and coend as
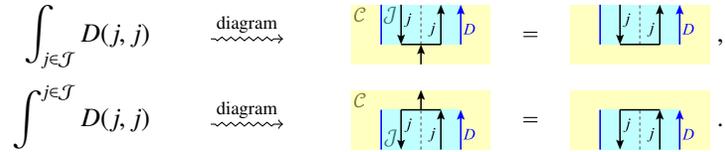
\begin{alignat}{1}
 \int_{j \in \cJ} D(j,j) &\qquad\diagram\qquad
 \InsertMidPDF{end_f_end.pdf}, \nonumber \\
 \int^{j \in \cJ} D(j,j) &\qquad\diagram\qquad
 \InsertMidPDF{end_f_coend.pdf}.
 \label{eq:end_f_end}
\end{alignat}
Note that they are objects in $\cC$ that are independent of $j$.
The black horizontal lines in the diagrams represent ``$\int_{j \in \cJ}$''
and ``$\int^{j \in \cJ}$''.
As shown on their right-hand sides, the arrows below the horizontal line representing
``$\int_{j \in \cJ}$'' and above the horizontal line representing ``$\int^{j \in \cJ}$''
are often omitted.
Also, the label ``$j$'' in these diagrams is often omitted.

Consider a functor $D \colon \cJ^\op \times \cJ \times \cD \to \cC$.
For any $d \in \cD$, we can define a functor
$D(\Endash,\Enndash,d) \colon \cJ^\op \times \cJ \to \cC$
(recall Example~\ref{ex:FunctorBifunc}).
We represent an end of this functor as
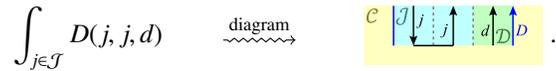
\begin{alignat}{1}
 \int_{j \in \cJ} D(j,j,d) &\qquad\diagram\qquad
 \InsertMidPDF{end_f_end2.pdf}.
 \label{eq:end_f_end2}
\end{alignat}
Other (co)ends, such as $\int^{j \in \cJ} D(j,j,d)$ or $\int_{j \in \cJ} E(j,d,j,d')$,
can also be considered similarly.

\subsection{Basic properties of (co)ends} \label{subsec:end_property}

\subsubsection{Properties of (co)ends as (co)limits}

It is known that ends are special cases of limits, and coends are special cases of colimits
(e.g., \cite{Mac-2013}).
Therefore, properties that hold for (co)limits hold for (co)ends.
Also, concepts related to (co)limits, such as their preservation, can be applied directly to (co)ends.
Here, we provide several examples.

For each $c \in \cC$, $\yoneda{c} = \cC(c,\Endash)$ preserves any limit
(recall Theorem~\ref{thm:LimitPresheaf}), so it preserves any end.
That is, for any $c \in \cC$ and bifunctor $D \colon \cJ^\op \times \cJ \to \cC$,
\begin{alignat}{1}
 \lefteqn{ \cC\left( c, \int_{j \in \cJ} D(j,j) \right) \cong \int_{j \in \cJ} \cC(c, D(j,j)) }
 \nonumber \\
 &\diagram\qquad \InsertMidPDF{end_commute_end.pdf}
 \label{eq:end_commute_end}
\end{alignat}
holds.
Note that the area enclosed by the auxiliary line represents
the bifunctor $\cC(c,D(\Endash,\Enndash))$.
Since an end is a limit, this isomorphism is natural in $c$ and $D$.
In Eq.~\eqref{eq:end_commute_end}, there are short horizontal lines attached to arrows
representing $c$ and $D$, which indicate that the isomorphism is natural in $c$ and $D$.
In what follows, we use such expressions to represent naturality.
Intuitively, this diagram can be interpreted as allowing the horizontal solid line
representing ``$\int_{j \in \cJ}$'' to freely pass through the dark gray rectangle representing
the hom-functor $\cC(\Endash, \Enndash)$.

As the dual of Eq.~\eqref{eq:end_commute_end},
$\yonedaop{c} = \cC(\Endash,c)$ preserves any colimit
(recall Corollary~\ref{cor:LimitPresheaf2}), so it preserves any coend.
That is,
\begin{alignat}{1}
 \lefteqn{ \cC\left( \int^{j \in \cJ} D(j,j), c \right) \cong \int_{j \in \cJ} \cC(D(j,j), c) }
 \nonumber \\
 &\diagram\qquad \InsertMidPDF{end_commute_coend.pdf}
 \label{eq:end_commute_coend}
\end{alignat}
holds%
\footnote{Although the formula does not demonstrate that this isomorphism is natural in $D$ and $c$,
this information is illustrated in the diagram.}.
Note that by representing the coend $\int^{j \in \cJ} D(j,j)$ as an ``upside-down'' version,
we intuitively show how a coend is transformed into an end.

Also, if a category $\cC$ is complete, then for any small categories $\cI$ and $\cJ$ and
any functor $D \colon \cI^\op \times \cI \times \cJ^\op \times \cJ \to \cC$,
\begin{alignat}{1}
 \lefteqn{ \int_{i \in \cI} \int_{j \in \cJ} D(i,i,j,j)
 \cong \int_{j \in \cJ} \int_{i \in \cI} D(i,i,j,j) } \nonumber \\
 &\diagram\qquad \InsertMidPDF{end_commute.pdf}
 \label{eq:end_commute}
\end{alignat}
holds, where $D_\cI \coloneqq \int_{j \in \cJ} D(\Endash,\Enndash,j,j)$
and $D_\cJ \coloneqq \int_{i \in \cI} D(i,i,\Endash,\Enndash)$.
Note that the gray dashed lines in the horizontal direction are natural isomorphisms.
This isomorphism justifies that both the left- and right-hand sides of Eq.~\eqref{eq:end_commute} are
represented by
\begin{alignat}{1}
 \InsertPDF{end_commute2.pdf} \raisebox{1em}{.}
 \label{eq:end_commute2}
\end{alignat}

\subsubsection{Collections of natural transformations are ends}

The following lemma is known to hold (the proof is omitted; see, e.g., \cite{Mac-2013}).
\begin{lemma}{}{EndNat}
 Assume that $\Func{\cC}{\cD}$ is locally small.
 We have for any two functors $F,G \colon \cC \to \cD$,
 \begin{alignat}{1}
  \Func{\cC}{\cD}(F,G) \cong \int_{c \in \cC} \cD(Fc,Gc)
  &\qquad\diagram\qquad
  \InsertMidPDF{end_nat.pdf}. \nonumber \\
  \label{eq:end_nat}
 \end{alignat}
\end{lemma}

Using this lemma, for objects $c,d$ in $\cC$ and a functor $F \colon \cC \to \Set$,
the Yoneda lemma can be represented as
\begin{alignat}{1}
 \int_{c \in \cC} \Set(\cC(d,c),Fc) \cong Fd
 &\qquad\diagram\qquad
 \InsertMidPDF{end_yoneda.pdf}.
 \nonumber \\
 \label{eq:end_yoneda}
\end{alignat}
Equation~\eqref{eq:end_yoneda} may be more visually intuitive when interpreted as
the following picture (which is not a strict diagram):
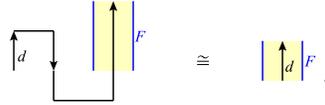
\begin{alignat}{1}
 \InsertPDF{end_yoneda_intuitive.pdf} \raisebox{1em}{.}
 \label{eq:end_yoneda_intuitive}
\end{alignat}
Also, considering the dual, we have for a presheaf $G \colon \cC^\op \to \Set$,
\begin{alignat}{1}
 \int_{c \in \cC} \Set(\cC(c,d),Gc) \cong Gd
 &\qquad\diagram\qquad \InsertMidPDF{end_yoneda2.pdf},
 \nonumber \\
 \label{eq:end_yoneda2}
\end{alignat}
which can also be expressed as
\begin{alignat}{1}
 \hat{\cC}(\yonedaop{\cC}\Endash,G) \cong G
 &\qquad\diagram\qquad
 \InsertMidPDF{end_yoneda2_another.pdf}.
 \label{eq:end_yoneda2_another}
\end{alignat}

\subsection{(Co)tensor products} \label{subsec:end_tensor}

\subsubsection{Definition of (co)tensor products}

As a preliminary, we introduce (co)tensor products.
For a set $X$ and an object $c \in \cC$, if there exists
an object, denoted by $X \ot c$, in $\cC$ such that
\begin{alignat}{1}
 \cC(X \ot c, c') &\cong \Set(X, \cC(c,c'))
 \label{eq:tensor_product}
\end{alignat}
is natural in $c'$, then $X \ot c$ is called a \termdef{tensor product} (or copower)
of $X$ and $c$.
Also, if there exists a functor $\ot \colon \Set \times \cC \to \cC$ that maps
$\braket{X,c} \in \Set \times \cC$ to $X \ot c \in \cC$ such that
Eq.~\eqref{eq:tensor_product} is natural in $X,c,c'$, then $\ot$ is called the tensor product.
We have for any $X \in \Set$ and $c \in \cC$,
$X \ot c \cong \coprod_{x \in X} c$ (where $\coprod$ is the coproduct).
Therefore, if $\cC$ is cocomplete, then $\cC$ has any tensor product.
In particular, for $\cC = \Set$, $X \ot c \cong X \times c$ holds.
If $\cC$ has any tensor product, then for each $c \in \cC$, $\cC(c,\Endash)$ has
a left adjoint $\Endash \ot c$.

The dual of a tensor product can also be considered.
Specifically, for a set $X$ and an object $c' \in \cC$,
if there exists an object, denoted by $X \cot c'$, in $\cC$ such that
\begin{alignat}{1}
 \cC(c, X \cot c') &\cong \Set(X, \cC(c,c'))
 \label{eq:tensor_coproduct}
\end{alignat}
is natural in $c$, then $X \cot c'$ is called a \termdef{cotensor product} (or power)
of $X$ and $c'$.
Also, if there exists a functor $\cot \colon \Set^\op \times \cC \to \cC$ that maps
$\braket{X,c'} \in \Set^\op \times \cC$ to $X \cot c' \in \cC$ such that
Eq.~\eqref{eq:tensor_coproduct} is natural in $X,c,c'$, then $\cot$ is called the cotensor product.
We have for any $X \in \Set$ and $a \in \cC$, $X \cot a \cong \prod_{x \in X} a$
(where $\prod$ is the direct product).
Therefore, if $\cC$ is complete, then $\cC$ has any cotensor product.
In particular, for $\cC = \Set$, $X \cot a \cong a^X \coloneqq \Set(X,a)$ holds.

Consider the tensor product $\ot \colon \Set \times \cC \to \cC$ and
the cotensor product $\cot \colon \Set^\op \times \cC \to \cC$.
Like direct products (see Eq.~\eqref{eq:end_picture_func2}),
we represent (co)tensor products by placing elements side by side.
However, unlike direct products, we do not insert dashed lines.
Specifically, for $X \in \Set$ and $c \in \cC$, we represent $X \ot c$ and $X \cot c$ by
\begin{alignat}{1}
 X \ot c &\qquad\diagram\qquad
 \InsertMidPDF{end_tensor_def.pdf},
 \qquad
 X \cot c \qquad\diagram\qquad
 \InsertMidPDF{end_tensor_def_op.pdf}.
 \label{eq:end_tensor_def}
\end{alignat}
Note that $X$ is an object in $\Set$, not in $\cC$.
If $\cC$ has any tensor product and cotensor product, then
from Eqs.~\eqref{eq:tensor_product} and \eqref{eq:tensor_coproduct},
there exists an isomorphism
\begin{alignat}{1}
 \lefteqn{ \cC(X \ot c, c') \cong \Set(X, \cC(c,c')) \cong \cC(c, X \cot c') } \nonumber \\
 &\diagram\qquad
 \InsertMidPDF{end_tensor.pdf}
 \label{eq:end_tensor}
\end{alignat}
that is natural in $X \in \Set$ and $c,c' \in \cC$.
In particular, when $\cC = \Set$, the first isomorphism of Eq.~\eqref{eq:end_tensor} is equivalent to
\begin{alignat}{1}
 \Set(X \times Y, Z) \cong \Set(X, Z^Y)
 &\qquad\diagram\qquad
 \InsertMidPDF{end_picture_hom2.pdf}. \nonumber \\
 \label{eq:end_picture_hom2}
\end{alignat}

\subsubsection{Topics related to (co)tensor products}

\myparagraph{Expressions of pointwise Kan extensions by (co)tensor products}

\begin{proposition}{}{KanTensor}
 Consider two functors $K \colon \cC \to \cD$ and $F \colon \cC \to \cE$, where $\cC$ is small.
 If $\cE$ is cocomplete, then a pointwise left Kan extension $\Lan_K F$ exists and satisfies
 \begin{alignat}{1}
  \Lan_K F \cong \int^{c \in \cC} \cD(Kc,\Endash) \ot Fc
  &\qquad\diagram\qquad
  \InsertMidPDF{end_Kan.pdf}.
  \label{eq:end_Lan}
 \end{alignat}
 Dually, if $\cE$ is complete, then a pointwise right Kan extension $\Ran_K F$ exists and satisfies
 \begin{alignat}{1}
  \Ran_K F \cong \int_{c \in \cC} \cD(\Endash,Kc) \cot Fc
  &\qquad\diagram\qquad
  \InsertMidPDF{end_Kan_right.pdf}.
  \label{eq:end_Ran}
 \end{alignat}
\end{proposition}
\begin{proof}
 From Corollary~\ref{cor:KanPointwiseCocomplete}, a pointwise left Kan extension $\Lan_K F$ exists.
 Also, we have
 \begin{alignat}{1}
  \footnoteinsets{-0.27}{1.05}{\eqref{eq:end_tensor}}{\eqref{eq:end_commute}}{%
  \footnoteinset{-0.27}{-0.7}{\eqref{eq:end_yoneda}}{%
  \InsertPDF{end_Kan_proof.pdf}}} \raisebox{1em}{,}
  \label{eq:end_Kan_proof}
 \end{alignat}
 whose mathematical notation is
 \begin{alignat}{1}
  \int_{d \in \cD} \cE \left( \int^{c \in \cC} \cD(Kc,d) \ot Fc,Hd \right)
  &\stackrel{\eqref{eq:end_commute_coend}}{~~\cong~~}
  \int_{d \in \cD} \int_{c \in \cC} \cE( \cD(Kc,d) \ot Fc,Hd) \nonumber \\
  &\stackrel{\eqref{eq:end_tensor}}{~~\cong~~}
  \int_{d \in \cD} \int_{c \in \cC} \Set(\cD(Kc,d),\cE(Fc,Hd)) \nonumber \\
  &\stackrel{\eqref{eq:end_commute}}{~~\cong~~}
  \int_{c \in \cC} \int_{d \in \cD} \Set(\cD(Kc,d),\cE(Fc,Hd)) \nonumber \\
  &\stackrel{\eqref{eq:end_yoneda}}{~~\cong~~}
  \int_{c \in \cC} \cE(Fc, H(Kc)).
 \end{alignat}
 Therefore, from Eq.~\eqref{eq:end_nat}, there exists an isomorphism
 \begin{alignat}{1}
  \Func{\cD}{\cE}\left( \int^{c \in \cC} \cD(Kc,\Endash) \ot Fc, H \right)
  &\cong \Func{\cC}{\cE}(F, H \b K)
 \end{alignat}
 that is natural in $F$ and $H$.
 Thus, from Eq.~\eqref{eq:Kan_left_cong}, Eq.~\eqref{eq:end_Lan} holds.
\end{proof}

Note that Proposition~\ref{pro:KanPointwiseNas} asserts that
if a functor $L \colon \cD \to \cE$ is a pointwise left Kan extension of $F$ along $K$,
then the isomorphism
\begin{alignat}{1}
 \lefteqn{ \cE(Ld,e) \cong \hat{\cC}(\cD(K\Endash,d),\cE(F\Endash,e))} \nonumber \\
 &\diagram\qquad
 \InsertMidPDF{end_Kan_dual.pdf}
 \label{eq:end_Kan_dual}
\end{alignat}
is natural in $d$ and $e$.
It can be seen that Eq.~\eqref{eq:end_Kan_dual} is obtained by applying
$\cE(\Endash(d),e)$ to Eq.~\eqref{eq:end_Lan}.

\myparagraph{The co-Yoneda lemma}

\begin{cor}{the co-Yoneda lemma}{Coyoneda}
 We have for objects $c,d$ in a small category $\cC$ and a functor $F \colon \cC \to \Set$,
 \begin{alignat}{1}
  \lefteqn{\int^{c \in \cC} \cC(c,d) \times Fc \cong Fd
  \cong \int^{c \in \cC} Fc \times \cC(c,d)} \nonumber \\
  &\diagram\qquad
  \InsertMidPDF{end_yoyoneda.pdf}.
  \label{eq:end_yoyoneda}
 \end{alignat}
\end{cor}
\begin{proof}
 It is sufficient to show the first isomorphism since the right-hand side is simply
 a rearrangement of the direct product on the left-hand side.
 This can be easily seen by considering the case of $\cD = \cC$, $\cE = \Set$, and $K = \id_\cC$
 in Proposition~\ref{pro:KanTensor}.
 Indeed, since $\Set$ is cocomplete, $\Lan_{\id_\cC} F$ exists and satisfies Eq.~\eqref{eq:end_Lan}.
 Thus, from $\Lan_{\id_\cC} \cong \id_{\Func{\cC}{\Set}}$, we obtain Eq.~\eqref{eq:end_yoyoneda}.
\end{proof}
Equation~\eqref{eq:end_yoyoneda} can be rewritten as
\begin{alignat}{1}
 \int^{c \in \cC} \yoneda{c} \times Fc &\cong F \cong \int^{c \in \cC} Fc \times \yoneda{c}.
\end{alignat}

Dually, the following corollary holds.
\begin{cor}{the co-Yoneda lemma}{Coyonedaop}
 We have that for objects $c,d$ in a small category $\cC$ and a presheaf $G \colon \cC^\op \to \Set$,
 \begin{alignat}{1}
  \lefteqn{\int^{c \in \cC} \cC(d,c) \times Gc \cong Gd
  \cong \int^{c \in \cC} Gc \times \cC(d,c)} \nonumber \\
  &\diagram\qquad
  \InsertMidPDF{end_yoyonedaop.pdf}.
  \label{eq:end_yoyonedaop}
 \end{alignat}
\end{cor}
Equation~\eqref{eq:end_yoyonedaop} can be rewritten as
\begin{alignat}{1}
 \int^{c \in \cC} \yonedaop{c} \times Gc &\cong G \cong \int^{c \in \cC} Gc \times \yonedaop{c}.
 \label{eq:end_yoyonedaop2}
\end{alignat}

\begin{proposition}{}{LanKYoneda}
 Consider any functor $K \colon \cC \to \cD$, where $\cC$ is small;
 then, we have
 \begin{alignat}{1}
  (\Lan_K \yonedaop{\cC})d \cong \cD(K \Endash,d)
  &\qquad\diagram\qquad
  \InsertMidPDF{end_Kan_yoneda.pdf}.
  \label{eq:end_KanYoneda}
 \end{alignat}
\end{proposition}
Since $\hat{\cC}$ is cocomplete%
\footnote{\label{ft:HatCCocomplete}This can be easily seen by substituting
$\cC^\op$ and $\Set$ for $\cI$ and $\cC$, respectively, in Corollary~\ref{cor:LimitBifuncFunc}
and using the fact that $\Set$ is cocomplete.},
a pointwise left Kan extension $\Lan_K \yonedaop{\cC}$ always exists
(recall Corollary~\ref{cor:KanPointwiseCocomplete}).
\begin{proof}
 The proof is immediate from
 \begin{alignat}{1}
  \footnoteinset{-4.57}{0.3}{\eqref{eq:end_Lan}}{%
  \footnoteinset{3.92}{0.3}{\eqref{eq:end_yoyonedaop}}{%
  \InsertPDF{report_end_Kan_yoneda_proof.pdf}}} \raisebox{1em}{,}
  \label{eq:end_Kan_yoneda_proof}
 \end{alignat}
 whose mathematical notation is
 \begin{alignat}{1}
  (\Lan_K \yonedaop{\cC})d
  &\stackrel{\eqref{eq:end_Lan}}{~~\cong~~}
  \int^{c \in \cC} \cD(Kc,d) \ot \yonedaop{\cC}(c)
  \stackrel{}{~~=~~}
  \int^{c \in \cC} \cD(Kc,d) \times \yonedaop{c}
  \stackrel{\eqref{eq:end_yoyonedaop2}}{~~\cong~~}
  \cD(K\Endash,d).
 \end{alignat}
\end{proof}

\begin{cor}{}{LanKYoneda}
 For a small category $\cC$, the Yoneda embedding $\yonedaop{\cC}$ is dense.
\end{cor}
\begin{proof}
 We have $\Lan_{\yonedaop{\cC}} \yonedaop{\cC} \cong \id_\cC$
 since we have for any $X \in \hat{\cC}$,
 \begin{alignat}{1}
  \footnoteinset{-1.12}{0.3}{\eqref{eq:end_KanYoneda}}{%
  \footnoteinset{2.06}{0.3}{\eqref{eq:end_yoneda2_another}}{%
  \InsertPDF{end_Kan_yoneda_dense.pdf}}} \raisebox{1em}{,}
  \label{eq:end_Kan_yoneda_dense}
 \end{alignat}
 whose mathematical notation is
 \begin{alignat}{1}
  (\Lan_{\yonedaop{\cC}} \yonedaop{\cC})X
  &\stackrel{\eqref{eq:end_KanYoneda}}{~~\cong~~}
  \hat{\cC}(\yonedaop{\cC}\Endash,X)
  \stackrel{\eqref{eq:end_yoneda2_another}}{~~\cong~~} X.
 \end{alignat}
\end{proof}

\myparagraph{The category of presheaves is cartesian closed}

\begin{proposition}{}{PresheafCCC}
 For a small category $\cC$, the category of presheaves $\hat{\cC}$ is cartesian closed%
 \footnote{A category $\cD$ is called \termdef{cartesian closed} if it has any finite product
 and there exists an exponential object $b^a$ for each $a,b \in \cD$.
 An exponential object $b^a$ is an object satisfying
 $\cD(\Endash,b^a) \cong \cD(\Endash \times a,b)$.}.
\end{proposition}
\begin{proof}
 Since $\hat{\cC}$ is complete%
 \footnote{This can be shown in the same way as footnote~\ref{ft:HatCCocomplete}
 using the fact that $\Set$ is complete.},
 it has any finite product.
 Therefore, it suffices to show that for each two presheaves $F,G \in \hat{\cC}$,
 there exists an exponential object $G^F$, i.e., an object $G^F$ satisfying
 \begin{alignat}{1}
  \hat{\cC}(\Endash, G^F) \cong \hat{\cC}(\Endash \times F, G)
  &\qquad\diagram\qquad
  \InsertMidPDF{end_yoneda_CCC_sufficient.pdf}.
  \label{eq:end_yoyoneda_CCC_sufficient}
 \end{alignat}
 We define $G^F$ as a presheaf satisfying%
 \footnote{If an exponential object $G^F$ exists, then there must exist an isomorphism
 $G^F(c) \cong \hat{\cC}(\yonedaop{c},G^F) \cong \hat{\cC}(\yonedaop{c} \times F,G)$
 that is natural in $c$,
 where the first isomorphism is due to the Yoneda lemma (see Corollary~\ref{cor:Yonedaop}).
 Therefore, $G^F$ must be defined to satisfy Eq.~\eqref{eq:presheaf_CCC_GF}.}
 \begin{alignat}{1}
  G^F(\Endash) &= \hat{\cC}(\yonedaop{\Endash} \times F,G).
  \label{eq:presheaf_CCC_GF}
 \end{alignat}
 Then, Eq.~\eqref{eq:end_yoyoneda_CCC_sufficient} is obtained from
 \begin{alignat}{1}
  \footnoteinsets{-3.17}{1.15}{\eqref{eq:presheaf_CCC_GF}}{\eqref{eq:end_tensor}}{%
  \footnoteinsets{1.04}{1.15}{\eqref{eq:end_commute_coend}}{\eqref{eq:end_nat}}{%
  \footnoteinset{1.04}{-0.85}{\eqref{eq:end_yoyonedaop}}{%
  \InsertPDF{end_yoneda_CCC.pdf}}}} \raisebox{1em}{,}
  \label{eq:end_yoyoneda_CCC}
 \end{alignat}
 whose mathematical notation is
 \begin{alignat}{1}
  \int_{c \in \cC} \Set(Hc, G^Fc)
  &\stackrel{\eqref{eq:presheaf_CCC_GF}}{~~\cong~~}
  \int_{c \in \cC} \Set(Hc, \hat{\cC}(\yonedaop{c} \times F, G)) \nonumber \\
  &\stackrel{\eqref{eq:end_tensor}}{~~\cong~~}
  \int_{c \in \cC} \hat{\cC}(Hc \ot (\yonedaop{c} \times F), G) \nonumber \\
  &\stackrel{\eqref{eq:end_commute_coend}}{~~\cong~~}
  \hat{\cC} \left( \int^{c \in \cC} Hc \ot (\yonedaop{c} \times F), G \right) \nonumber \\
  &\stackrel{\eqref{eq:end_nat}}{~~\cong~~}
  \int_{d \in \cC^\op} \Set\left( \int^{c \in \cC} Hc \times \cC(d,c) \times Fd, Gd \right) \nonumber \\
  &\stackrel{\eqref{eq:end_yoyonedaop}}{~~\cong~~}
  \int_{d \in \cC^\op} \Set(Hd \times Fd, Gd).
 \end{alignat}
 Note that in this diagram, the dashed line representing the product $\yonedaop{c} \times F$
 is omitted. 
\end{proof}

\myparagraph{Adjunctions between the category of presheaves}

\noindent
It can be shown that for a small category $\cC$,
any adjunction $F \dashv G \colon \hat{\cC} \to \cD$ can be expressed using a certain Kan extension
(see Proposition~\ref{pro:KanNerveRealization2}),
and the converse also holds (see Proposition~\ref{pro:KanNerveRealization}).
\begin{proposition}{}{KanNerveRealization}
 Consider a functor $K \colon \cC \to \cD$, where $\cC$ is small.
 If a pointwise left Kan extension
 $F \coloneqq \Lan_{\yonedaop{\cC}} K \colon \hat{\cC} \to \cD$ exists,
 then $F \dashv \Lan_K \yonedaop{\cC}$ holds.
\end{proposition}
\begin{proof}
 Since $\hat{\cC}$ is cocomplete, a pointwise left Kan extension
 $G \coloneqq \Lan_K \yonedaop{\cC}$ exists (recall Corollary~\ref{cor:KanPointwiseCocomplete}).
 $F \dashv G$ holds since we have
 \begin{alignat}{1}
  \footnoteinset{-2.72}{0.3}{\eqref{eq:end_Kan_dual}}{%
  \footnoteinsets{2.58}{0.3}{\eqref{eq:end_yoneda2_another}}{\eqref{eq:end_KanYoneda}}{%
  \InsertPDF{end_nerve_realization.pdf}}} \raisebox{1em}{,}
  \label{eq:end_nerve_realization}
 \end{alignat}
 whose mathematical notation is
 \begin{alignat}{1}
  \cD(FX,d)
  &\stackrel{\eqref{eq:end_Kan_dual}}{~~\cong~~}
  \hat{\cC}(\hat{\cC}(\yonedaop{\cC}\Endash,X), \cD(K\Endash,d))
  \stackrel{\eqref{eq:end_yoneda2_another}}{~~\cong~~}
  \hat{\cC}(X,\cD(K\Endash,d))
  \stackrel{\eqref{eq:end_KanYoneda}}{~~\cong~~}
  \hat{\cC}(X,Gd).
 \end{alignat}
\end{proof}

\begin{proposition}{}{KanNerveRealization2}
 Consider an adjunction $F \dashv G \colon  \hat{\cC} \to \cD$, where $\cC$ is small.
 There exists a functor $K \colon \cC \to \cD$ satisfying
 $F \cong \Lan_{\yonedaop{\cC}} K$ and $G \cong \Lan_K \yonedaop{\cC}$.
\end{proposition}
Intuitively, the functor $K$ can be said to have information about both $F$ and $G$.
\begin{proof}
 Let $K \coloneqq F \b \yonedaop{\cC}$.
 From $\Lan_{\yonedaop{\cC}} \yonedaop{\cC} \cong \id_{\hat{\cC}}$
 (recall Corollary~\ref{cor:LanKYoneda}) and the fact that the left adjoint $F$ preserves any
 left Kan extension (recall Proposition~\ref{pro:KanAdjPreserve}), we obtain
 \begin{alignat}{1}
  F \cong F \b \Lan_{\yonedaop{\cC}} \yonedaop{\cC}
  \cong \Lan_{\yonedaop{\cC}} (F \b \yonedaop{\cC}) \cong \Lan_{\yonedaop{\cC}} K.
 \end{alignat}
 For each $X \in \hat{\cC}$, let $P_X$ be the forgetful functor from
 $\yonedaop{\cC} \comma X$ to $\cC$; then,
 $\yonedaop{\cC} \b P_X \colon \yonedaop{\cC} \comma X \to \hat{\cC}$ has a colimit
 since $\hat{\cC}$ is cocomplete.
 Thus, $F \b \yonedaop{\cC} \b P_X = K \b P_X$ also has a colimit
 since the left adjoint $F$ preserves any colimit
 (recall Theorem~\ref{thm:LimitAdjPreserve}).
 Therefore, the Kan extension $F \cong \Lan_{\yonedaop{\cC}} K$ is pointwise,
 and thus from Proposition~\ref{pro:KanNerveRealization},
 $\Lan_K\yonedaop{\cC}$ is a right adjoint to $F$.
 Since right adjoints are essentially unique, we have $G \cong \Lan_K\yonedaop{\cC}$.
\end{proof}

\subsection{Weighted (co)limits} \label{subsec:end_wlim}

Weighted (co)limits can be considered as generalizations of (co)limits and (co)ends.

\subsubsection{Definition of weighted (co)limits} \label{subsec:end_wlim_def}

\begin{define}{weighted (co)limits}{Wlim}
 Arbitrarily choose two functors $W \colon \cC \to \Set$ and
 $F \colon \cC \to \cE$.
 An object, denoted by $\lim^W F$, in $\cE$ is called a \termdef{limit of $F$ weighted by $W$}
 if the isomorphism
 \begin{alignat}{1}
  \lefteqn{ \cE(e, \lim^W F) \cong \Func{\cC}{\Set}(W, \cE(e, F \Endash)) } \nonumber \\
  &\diagram\qquad
  \InsertMidPDF{end_wlim_def.pdf}
  \label{eq:end_wlim_def}
 \end{alignat}
 is natural in $e \in \cE$.
 Dually, arbitrarily choose a presheaf $W' \colon \cC^\op \to \Set$.
 An object, denoted by $\colim^{W'} F$, in $\cE$ is called
 a \termdef{colimit of $F$ weighted by $W'$} if the isomorphism
 \begin{alignat}{1}
  \lefteqn{ \cE(\colim^{W'} F, e) \cong \hat{\cC}(W', \cE(F \Endash, e)) } \nonumber \\
  &\diagram\qquad
  \InsertMidPDF{end_wcolim_def.pdf}
  \label{eq:end_wcolim_def}
 \end{alignat}
 is natural in $e \in \cE$.
\end{define}

\begin{ex}{}{}
 A limit of $F$ is a limit of $F$ weighted by $\Delta_\cC \{ * \}$,
 where $\{ * \}$ is a singleton set.
 Indeed, we have
 \begin{alignat}{1}
  \cE(e, \lim F) &\cong \Cone(e, F) \cong \Cone(\{ * \}, \cE(e,F\Endash))
  \cong \Func{\cC}{\Set}(\Delta_\cC \{ * \}, \cE(e,F\Endash)),
 \end{alignat}
 where the first and second isomorphisms follow from Eqs.~\eqref{eq:limit_repr_cong}
 and \eqref{eq:limit_Cone_1cD}, respectively.
 Dually, a colimit of $F$ is a colimit of $F$ weighted by $\Delta_{\cC^\op} \{ * \}$.
\end{ex}

In particular, when $\cE = \Set$, we have
\begin{alignat}{1}
 \lim^W F &\cong \Func{\cC}{\Set}(W,F).
 \label{eq:end_wlim_Set}
\end{alignat}
Indeed, we have for any $X \in \Set$
\begin{alignat}{1}
 \footnoteinset{-2.31}{1.2}{\eqref{eq:end_nat}}{%
 \footnoteinsets{2.44}{1.2}{\eqref{eq:end_picture_hom2}}{\eqref{eq:end_commute_end}}{%
 \footnoteinset{-2.31}{-0.65}{\eqref{eq:end_nat}}{%
 \InsertPDF{end_wlim_Set_proof.pdf}}}} \raisebox{1em}{,}
 \label{eq:end_wlim_Set_proof}
\end{alignat}
whose mathematical notation is
\begin{alignat}{2}
 \Func{\cC}{\Set}(W\Endash,\Set(X,F\Endash))
 &\stackrel{\eqref{eq:end_nat}}{~~\cong~~}
 \int_{c \in \cC} \Set(Wc,\Set(X,Fc))
 &&\stackrel{\eqref{eq:end_picture_hom2}}{~~\cong~~}
 \int_{c \in \cC} \Set(X,\Set(Wc,Fc)) \nonumber \\
 &\stackrel{\eqref{eq:end_commute_end}}{~~\cong~~}
 \Set\left( X,\int_{c \in \cC} \Set(Wc,Fc) \right)
 &&\stackrel{\eqref{eq:end_nat}}{~~\cong~~}
 \Set(X,\Func{\cC}{\Set}(W,F)),
\end{alignat}
and thus we obtain Eq.~\eqref{eq:end_wlim_Set} from Eq.~\eqref{eq:end_wlim_def}.

\subsubsection{Basic properties of weighted (co)limits}

\myparagraph{Weighted (co)limits are (co)ends}

\begin{proposition}{}{WlimTensor}
 Consider two functors $W \colon \cC \to \Set$ and $F \colon \cC \to \cE$,
 where $\cC$ is small.
 We have
 \begin{alignat}{1}
  \lim^W F \cong \int_{c \in \cC} Wc \cot Fc
  &\qquad\diagram\qquad
  \InsertMidPDF{end_wlim_picture.pdf}
  \label{eq:wlim_cot}
 \end{alignat}
 whenever the cotensor products $Wc \cot Fc$ exist.
 Dually, consider two functors $W' \colon \cC^\op \to \Set$ and $F \colon \cC \to \cE$,
 where $\cC$ is small.
 We have
 \begin{alignat}{1}
  \colim^{W'} F \cong \int^{c \in \cC} W'c \ot Fc
  &\qquad\diagram\qquad
  \InsertMidPDF{end_wlim_picture_op.pdf} \nonumber \\
  \label{eq:wlim_ot}
 \end{alignat}
 whenever the tensor products $Wc' \ot Fc$ exist.
\end{proposition}
\begin{proof}
 We have
 \begin{alignat}{1}
  \footnoteinset{-1.98}{0.3}{\eqref{eq:end_nat}}{%
  \footnoteinset{2.78}{0.3}{\eqref{eq:end_tensor}}{%
  \InsertPDF{end_wlim_tensor.pdf}}} \raisebox{1em}{,}
  \label{eq:end_wlim_tensor}
 \end{alignat}
 whose mathematical notation is
 \begin{alignat}{1}
  \Func{\cC}{\Set}(W\Endash,\cE(e,F\Endash))
  &\stackrel{\eqref{eq:end_nat}}{~~\cong~~}
  \int_{c \in \cC} \Set(Wc,\cE(e,Fc))
  \stackrel{\eqref{eq:end_tensor}}{~~\cong~~}
  \int_{c \in \cC} \cE(e,Wc \cot Fc) \nonumber \\
  &\stackrel{\eqref{eq:end_commute_end}}{~~\cong~~}
  \cE\left( e, \int_{c \in \cC} Wc \cot Fc \right).
 \end{alignat}
 Thus, from the definition of weighted limits, we have Eq.~\eqref{eq:wlim_cot}.
\end{proof}

\myparagraph{Ends are weighted limits}

\begin{proposition}{}{EndWlim}
 We have that for any bifunctor $F \colon \cC^\op \times \cC \to \cD$,
 \begin{alignat}{1}
  \int_{c \in \cC} F(c,c) &\cong \lim^{\cC(\Endash,\Enndash)} F
  \label{eq:end_wlim_end}
 \end{alignat}
 whenever the weighted limit $\lim^{\cC(\Endash,\Enndash)} F$ exists.
 Dually, we have
 \begin{alignat}{1}
  \int^{c \in \cC} F(c,c) &\cong \colim^{\cC(\Enndash,\Endash)} F
  \label{eq:end_wcolim_end}
 \end{alignat}
 (where $\cC(\Enndash,\Endash) \colon \cC \times \cC^\op \to \Set$)
 whenever the weighted colimit $\colim^{\cC(\Enndash,\Endash)} F$ exists.
\end{proposition}
\begin{proof}
 We have
 \begin{alignat}{1}
  \footnoteinset{0.00}{1.1}{\eqref{eq:end_nat}}{%
  \footnoteinset{0.00}{-0.88}{\eqref{eq:end_yoneda}}{%
  \InsertPDF{end_wlim_end_proof.pdf}}} \raisebox{1em}{,}
  \label{eq:end_wlim_end_proof}
 \end{alignat}
 whose mathematical notation is
 \begin{alignat}{1}
  \Func{\cC^\op \times \cC}{\Set}(\cC(\Endash,\Enndash),\cD(d,F(\Endash,\Enndash)))
  &\stackrel{\eqref{eq:end_nat}}{~~\cong~~}
  \int_{c \in \cC} \int_{c' \in \cC} \Set(\cC(c,c'),\cD(d,F(c,c'))) \nonumber \\
  &\stackrel{\eqref{eq:end_yoneda}}{~~\cong~~}
  \int_{c \in \cC} \cD(d, F(c,c))
  \stackrel{\eqref{eq:end_commute_end}}{~~\cong~~}
  \cD\left(d, \int_{c \in \cC} F(c,c) \right).
 \end{alignat}
 Thus, we obtain Eq.~\eqref{eq:end_wlim_end} from Eq.~\eqref{eq:end_wlim_def}.
\end{proof}

\myparagraph{Weighted (co)limits can be expressed using pointwise Kan extensions}

\begin{proposition}{}{WcolimKan}
 For any small category $\cC$, cocomplete category $\cE$, presheaf $W' \colon \cC^\op \to \Set$,
 and functor $F \colon \cC \to \cE$, we have
 \begin{alignat}{1}
  \colim^{W'} F &\cong (\Lan_{\yonedaop{\cC}} F)W'.
  \label{eq:end_ColimKan_proof}
 \end{alignat}
\end{proposition}
\begin{proof}
 The proof is immediate from
 \begin{alignat}{1}
  \footnoteinset{-3.27}{0.3}{\eqref{eq:end_Lan}}{%
  \footnoteinset{0.70}{0.3}{\eqref{eq:end_yoneda2_another}}{%
  \footnoteinset{3.87}{0.3}{\eqref{eq:wlim_ot}}{%
  \InsertPDF{end_wlim_Kan.pdf}}}} \raisebox{1em}{,}
  \label{eq:end_wlim_Kan}
 \end{alignat}
 whose mathematical notation is
 \begin{alignat}{1}
  (\Lan_{\yonedaop{\cC}}{F})W'
  &\stackrel{\eqref{eq:end_Lan}}{~~\cong~~}
  \int^{c \in \cC} \hat{\cC}(\yonedaop{\cC}(c),W') \ot Fc
  \stackrel{\eqref{eq:end_yoneda2_another}}{~~\cong~~}
  \int^{c \in \cC} W'c \ot Fc
  \stackrel{\eqref{eq:wlim_ot}}{~~\cong~~}
  \colim^{W'} F.
 \end{alignat}
\end{proof}

\section*{Acknowledgements}

I would like to thank Osamu~Hirota, Masaki~Sohma, and Kentaro~Kato from Tamagawa University,
as well as Tsuyoshi~Sasaki~Usuda from Aichi Prefectural University
and Tomohiro~Sogabe from Nagoya University,
for their invaluable support and insightful discussions.
Furthermore, I had the opportunity to discuss the contents of this paper with
Keima~Akasaka, Tamotsu~Basseda, Koreto~Endo, Yusuke~Furuta, Yo~Ikeda,
Yuto~Ikeda, Tomohiro~Karube, Ryosuke~Koganezawa, Takayuki~Kuriyama, Hayato~Nasu,
Tomoya~Nishikata, Ray~D.~Sameshima, Kosei~Sawada, Yoshitsugu~Sekine, Kazuhiro~Shigekuni,
Kayo~Tei, Keisuke~Tomita, Haruki~Toyota, and Aiki~Ueno.

\end{document}